\documentclass[11pt]{amsart}

\usepackage{latexsym}
\usepackage{amssymb}
\usepackage{amsmath}
\usepackage{color}
\usepackage{bbm}

\usepackage{kantlipsum} % for text filler

\calclayout

\usepackage[pdftex,colorlinks]{hyperref}
\usepackage{graphicx}
\usepackage{tcolorbox}
\usepackage{tabularx}

\usepackage{tikz}
\usepackage{xparse}% So that we can have two optional parameters

\NewDocumentCommand\DownArrow{O{2.0ex} O{black}}{%
   \mathrel{\tikz[baseline] \draw [<-, line width=0.5pt, #2] (0,0) -- ++(0,#1);}
}
\NewDocumentCommand\UpArrow{O{2.0ex} O{black}}{%
   \mathrel{\tikz[baseline] \draw [->, line width=0.5pt, #2] (0,0) -- ++(0,#1);}
}

\newtheorem{theorem}{Theorem}[section]
\newtheorem{lemma}[theorem]{Lemma}
\newtheorem{proposition}[theorem]{Proposition}
\newtheorem{corollary}[theorem]{Corollary}
\newtheorem{definition}[theorem]{Definition}

\theoremstyle{remark}
\newtheorem{example}[theorem]{Example}
\newtheorem{remark}[theorem]{Remark}

\newtheorem{question}[theorem]{Question}

\numberwithin{equation}{section}
\DeclareMathOperator{\normdot}{\| \cdot \|}

\DeclareMathOperator{\Cn}{\mathbb{C}^{n}}
\DeclareMathOperator{\R}{\mathbb{R}}
\DeclareMathOperator{\C}{\mathbb{C}}
\DeclareMathOperator{\N}{\mathbb{N}}
\DeclareMathOperator{\Hh}{ \mathcal{H}}
\DeclareMathOperator{\Kk}{ \mathcal{K}}

\DeclareMathOperator{\Bh}{\mathcal{B}(\mathcal{H})} %bounded linear on H%
\DeclareMathOperator{\Bk}{\mathcal{B}(\mathcal{K})}
\DeclareMathOperator{\Aa}{\mathcal{A}} % A algebra%

\DeclareMathOperator{\G}{\mathbb{G}}
\DeclareMathOperator{\Hbb}{\mathbb{H}}
\DeclareMathOperator{\Cc}{\mathcal{C}}
\DeclareMathOperator{\Dd}{\mathcal{D}}
\DeclareMathOperator{\GH}{\mathbb{G}\times \mathbb{H}}

\DeclareMathOperator{\rmt}{\mathrm{t}}

\DeclareMathOperator{\sfr}{\mathfrak{s}}

\DeclareMathOperator{\Ll}{\mathcal{L}}
\DeclareMathOperator{\Ss}{\mathcal{S}}
\DeclareMathOperator{\Tt}{\mathcal{T}}
\DeclareMathOperator{\Rr}{\mathcal{R}}
\DeclareMathOperator{\Gg}{\mathcal{G}}
\DeclareMathOperator{\Qq}{\mathcal{Q}}
\DeclareMathOperator{\Qqc}{\mathcal{\Tilde{Q}}}

\DeclareMathOperator{\Vv}{\mathcal{V}}
\DeclareMathOperator{\Mm}{\mathcal{M}}

\DeclareMathOperator{\Jj}{\mathcal{J}}
\DeclareMathOperator{\Jjc}{\mathcal{\Tilde{J}}}

\DeclareMathOperator{\Pp}{\mathcal{P}}

\DeclareMathOperator{\B}{\mathcal{B}}
\DeclareMathOperator{\trace}{\mathrm{Tr}}

\DeclareMathOperator{\id}{\mathrm{id}}

\newcommand{\Mn}{M_{n}}

\newcommand{\inn}[1]{\left\langle#1\right\rangle}

\newcommand{\rank}{\mathop{\operator@font rank}}
\newcommand{\tr}[1]{\trace\left(#1\right)}

\newcommand{\e}{{\varepsilon}}
\newcommand{\vertiii}[1]{{\left\vert\kern-0.25ex\left\vert\kern-0.25ex\left\vert #1 
    \right\vert\kern-0.25ex\right\vert\kern-0.25ex\right\vert}}
\makeatother

\newcommand{\bb}[1]{\mathbb{#1}}
\newcommand{\cl}[1]{\mathcal{#1}}
\newcommand{\nor}[1]{\left\Vert #1\right\Vert}    %\nor{x^2}
\newcommand{\sca}[1]{\langle#1\rangle} %

\newcommand{\spann}{\mathrm{span}}

\begin{document}
%%%%%%%%%%%%%%%%%%%%%%%%%%%%%%%%%%%%%%%%%%%%%%%%%%

%%%%%%%%%%%%%%%%%%%%%%% Start Roman numbering %%%%

\title{Operator systems, contextuality and non-locality}

\date{12 September 2025}

\author[M. Anoussis]{Michalis Anoussis}
\address{Department of Mathematics, 
University of the Aegean, Karlovassi, Samos 83200, Greece}
\email{mano@aegean.gr}

\author[A. Chatzinikolaou]{Alexandros Chatzinikolaou}
\address{
Department of Mathematics, National and Kapodistrian University of Athens, Ath\-ens 157 84, Greece}
\email{achatzinik@math.uoa.gr}

%\author[A. Katavolos]{Aristides Katavolos}

\author[I. G. Todorov]{Ivan G. Todorov}
\address{
School of Mathematical Sciences, University of Delaware, 501 Ewing Hall,
Newark, DE 19716, USA}
\email{todorov@udel.edu}

\begin{abstract}
We introduce an operator system, universal for 
the probabilistic models of a contextuality scenario, and 
identify its maximal C*-cover via the right C*-algebra of a 
canonical ternary ring of operators, arising from a hypergraph 
version of stochastic operator matrices. 
We study dilating contextuality scenarios, which have the property 
that each positive operator representation thereof
admits a dilation to a projective representation on a larger Hilbert space, 
and characterise them via the equality of the aforementioned 
universal operator system and the operator system arising from the 
canonical generators of the respective hypergraph C*-algebra. 
We characterise the no-signalling probabilistic models over 
a pair of contextuality scenarios of different types,
which arise from either the positive operator representations 
or from the projective representations of these scenarios,  
in terms of states on operator system tensor products. 
Generalising the notion of a synchronous no-signalling correlation to the 
hypergraph framework, we define coherent probabilistic models associated with 
a given contextuality scenario and characterise various classes 
thereof via different types of traces of the hypergraph C*-algebra, 
associated with the scenario. 
We establish several equivalent formulations of the Connes Embedding Problem 
in terms of no-signalling probabilistic models and hypergraph operator systems. 
\end{abstract}

\maketitle

\vspace{1truecm}

\tableofcontents

%%%%%%%%%%%%%%%%%%%%%%%%%%%%%%%%%%%%%%%%%%%%%%%%%%%%%%%%%%%%%%%%%%%%%%%%%%%%%%%
%%%%%%%%%%%%%%%%%%%%%%%%%%%%%%%%%%%%%%%%%%%%%%%%%%%%%%%%%%%%%%%%%%%%%%%%%%%%%%%

\section{Introduction}\label{s_intro}

The interactions between operator algebras and quantum information theory 
have in the past decade undergone a phase of vibrant development. 
Arguably the most striking examples in this respect are
the proof of the equivalence \cite{Fritz2012, MR2790067, Ozawa_2013}
between the Connes Embedding Problem in 
von Neumann algebra theory \cite{Connes76} and the Tsirelson Problem 
in quantum physics \cite{Tsirelson}, and the resolution of these problems 
(in the negative) in \cite{jnvwy}. At the core of the aforementioned equivalence is Kirchberg's reformulation 
of the Connes Embedding Problem in terms of the equality 
between tensor products of free group C*-algebras \cite{Kirchberg1993}
and a hierarchy of classes of no-signalling correlations 
whose prototype is the fundamental Bell Theorem on 
non-locality \cite{Bell1964}. 

More specifically, the correlated behaviour of 
two parties participating in a quantum mechanical 
experiment may be governed, depending on the resources they 
use, by several main, successively coarser, correlation types; 
this hierarchy is mathematically captured by a 
chain of inclusions of classes of conditional probability distributions, 
\begin{equation}\label{eq_hie}
\cl C_{\rm loc}\subseteq \cl C_{\rm q}\subseteq \cl C_{\rm qa}\subseteq \cl C_{\rm qc}\subseteq \cl C_{\rm ns}.
\end{equation}
Here, 
the class $\cl C_{\rm loc}$ corresponds to the use of classical physical
resources (determinism and shared randomness), 
$\cl C_{\rm q}$ to finite dimensional entanglement, 
$\cl C_{\rm qa}$ to liminal entanglement (in that the class $\cl C_{\rm qa}$ is the 
closure of the class $\cl C_{\rm q}$), 
$\cl C_{\rm qc}$ to the employment 
of the commuting operator model of quantum mechanics, while 
$\cl C_{\rm ns}$ is the most general class, corresponding to 
generalised probabilistic theories. 
Strict inclusions occur at all points of the hierarchy (\ref{eq_hie}) -- 
the inequality $\cl C_{\rm loc}\neq \cl C_{\rm q}$ is a reformulation of Bell's Theorem 
\cite{Bell1964}, 
$\cl C_{\rm q}\neq \cl C_{\rm qa}$ is due to 
the work of Slofstra \cite{Slofstra2020} (see also \cite{dpp, Musat-Rordam}), 
$\cl C_{\rm qa}\neq \cl C_{\rm qc}$ is due to the aforementioned groundbreaking 
work \cite{jnvwy}, while the last inequality, 
$\cl C_{\rm qc}\neq \cl C_{\rm ns}$, is due to Tsirelson 
(see \cite{Tsirelson} and \cite{fkptCMP}). 
The operator algebraic approach to the study of the 
hierarchy (\ref{eq_hie}) consists in 
the description of the elements of each one of its terms 
via states on tensor products of C*-algebras \cite{Fritz2012} 
and operator systems \cite{Lupini-etal}, 
or via traces on C*-algebras \cite{PAULSEN20162188, 10.1063/1.4996867}. 

Non-locality in quantum mechanics is often considered along with 
another of its fundamental features, contextuality. 
Although formally non-locality can be expressed as a special case of 
contextuality, the latter is usually studied in parallel with the former, 
and resides in the impossibility to make coherent assignments for outcomes 
of quantum mechanical experiments, independently of the 
utilised measurement devices, 
often modeled by sets of projective measurements. 
The latter viewpoint is used in \cite{Acn2015ACA}, where 
the notion of a {\it contextuality scenario} is introduced;
mathematically, this concept amounts to a hypergraph, whose 
vertices represent possible outcomes of an experiment, and whose 
hyperedges represent sets of outcomes that can be measured simultaneously.
The products of contextuality scenarios with mutually disjoint edges 
assumed to have the same cardinality (called
Bell scenarios in \cite{Acn2015ACA})
lead to the hierarchy (\ref{eq_hie}), as
positive operator valued measures (POVM's) 
can be simultaneously dilated to projection-valued measures. 

However, not all contextuality scenarios enjoy such a dilatability property.
Indeed, it was proved in \cite{doi:10.1063/5.0022344} that 
(finite dimensionally acting)
quantum magic squares do not necessarily dilate to finite dimensional 
quantum permutation matrices, and the argument
is extendable (as we show herein), allowing to conclude 
that a dilation in infinite dimensions does not either always exist.  
Thus, while sufficient for Bell scenarios, 
projective representations of contextuality scenarios 
(that is, coherent projective measurements of the simultaneously measurable outcomes)
do not capture the interplay between non-locality and contextuality
in its full generality. 

This is the starting point of the current investigation. 
Along with projective representations (PR's) of contextuality scenarios, 
studied in \cite{Acn2015ACA}, here 
we consider positive operator representations (POR's), that is,
collections of positive operators acting on a Hilbert space, 
which define a POVM
over every edge of the underlying hypergraph. 
(We note that connections of POR's with convex polytopes were recently explored in 
\cite{Bluhm-etal}.)
Given a non-trivial 
contextuality scenario $\bb{G}$, we construct an operator system 
$\cl S_{\bb{G}}$ with the (universal) property that the unital 
completely positive maps from $\cl S_{\bb{G}}$ correspond precisely to 
POR's of $\bb{G}$. 
We further identify an operator system $\cl T_{\bb{G}}$
(as an operator subsystem of the hypergraph C*-algebra $C^*(\bb{G})$
associated to $\bb{G}$ \cite{Acn2015ACA}) with the (universal) property that 
the unital completely positive maps from $\cl S_{\bb{G}}$ 
correspond precisely to POR's of $\bb{G}$ that can be dilated to 
PR's of $\bb{G}$. 
This allows us to equivalently express the property that every POR of a contextuality 
scenario can be dilated to a PR 
via the equality between $\cl S_{\bb{G}}$ and 
$\cl T_{\bb{G}}$ in the operator system category
(such contextuality scenarios are called dilating herein). 

The operator system $\cl T_{\bb{G}}$ lives inside a 
natural C*-cover, namely $C^*(\bb{G})$; we show that $C^*(\bb{G})$ is in fact 
the C*-envelope of $\cl T_{\bb{G}}$. 
In contrast, the operator system $\cl S_{\bb{G}}$ 
is not a priori given as an operator subsystem of a C*-algebra. 
We construct explicitly its universal (also known as maximal) C*-cover, 
showing that it is *-isomorphic to 
the right C*-algebra of a ternary ring of operators 
generated by $\bb{G}$-stochastic operator matrices. 
The latter notion is a generalisation of 
stochastic and bistochastic operator matrices, 
studied in \cite{Todorov2020QuantumNC} and \cite{BHTT2}, respectively. 
The identification of the C*-envelope of $\cl S_{\bb{G}}$, 
in the cases where $\bb{G}$ is not dilating, remains an open question. 

A natural contextuality scenario counterpart of the 
notion of a no-signalling correlation was introduced in \cite{Acn2015ACA}; 
for Bell scenarios, no-signalling in the contextuality paradigm reduces to 
the classical concept. 
The existence of non-dilating contextuality scenarios
leads to the necessity to consider (at least) two versions of the hierarchy  
(\ref{eq_hie}), in which correlations are defined using PR's or POR's, 
respectively, 
uncovering a new level of generality 
that allows to study non-locality and contextuality 
within the same setup. 
We characterise the main types of no-signalling correlations
via states on tensor products of the operator systems 
of the form $\cl S_{\bb{G}}$ (resp. $\cl T_{\bb{G}}$). 
These results extend to the contextuality paradigm 
the operator representation results 
for the correlation types (\ref{eq_hie}) established in 
\cite{Paulsen2013QUANTUMCN} and \cite{Lupini-etal}.

An important source of motivation for the study of 
no-signalling correlations and the associated hierarchy (\ref{eq_hie})
is their connection with non-local games. 
These are cooperative games, played by two players against a verifier, 
under the condition that no communication takes places between the players 
and they are limited to the use of physical resources of certain 
chosen type. 
Synchronous games, whose study was initiated in \cite{PAULSEN20162188}, 
are characterised by the requirement that the players' outputs 
should be perfectly correlated, provided that the inputs are so.
The latter class of games include a plethora of important examples, such as graph 
homomorphism/isomorphism games \cite{MR3471851, MR3926289} and 
linear binary constraint system games
\cite{10.1063/1.4996867}. 
Its associated perfect strategies  
are synchronous correlations \cite{PAULSEN20162188}.
We define the notion of a coherent no-signalling probabilistic model 
over a given contextuality scenario $\bb{G}$, 
which reduces to the notion of a synchronous correlation 
in the case where $\bb{G}$ is a Bell scenario. 
One of the cornerstone operator algebraic result 
in non-local game theory is the expression of 
synchronous correlations in terms of traces of a 
canonical group C*-algebra \cite{PAULSEN20162188}. 
We provide a contextuality scenario 
version of this result, showing that coherent correlations 
of quantum commuting type over $\bb{G}$
(that is, the hypergraph counterpart of the term 
$\cl C_{\rm qc}$ in the correlation chain (\ref{eq_hie}))
correspond to traces on the hypergraph C*-algebra 
$C^*(\bb{G})$. Similarly to the Bell scenario case, 
the approximate quantum type (the analogue of the class $\cl C_{\rm qa}$)
is described by amenable traces of the C*-algebra $C^*(\bb{G})$.

Finally, we provide new equivalences of the Connes Embedding and the Tsirelson
Problems in the setup of no-signalling probabilistic models over general contextuality 
scenarios. These results inscribe in a more general effort to 
identify an operator theoretic approach to the latter problems, in light of 
the fact that the proof available currently \cite{jnvwy} relies 
on complexity theory methods. 
We show that the Connes Embedding Problem is equivalent to 
the equality of appropriate tensor products between the operator 
systems of the form $\cl S_{\bb{G}}$ and $\cl T_{\bb{G}}$, 
as well as to the equality between the classes of quantum approximate and quantum commuting 
probabilistic models over general hypergraphs. 
This substantially enlarges the pool of tentative counterexamples, 
reducing the solution of the problem 
to the identification of a single hypergraph 
with distinct quantum approximate and quantum commuting correlation classes
(in either the positive operator or the projective framework). 

We now describe the organisation of the paper in more detail. 
After collecting some necessary preliminaries in Section \ref{s_prel}, 
in Section \ref{s_opsys} we define and examine the universal 
operator system $\cl S_{\bb{G}}$ of general probabilistic models 
of a (non-trivial) contextuality scenario $\bb{G}$. 
The operator system $\cl S_{\bb{G}}$
is constructed as a quotient of a finite dimensional abelian C*-algebra;
this allows the identification of its dual operator system 
as a matricial operator system if the underlying hypergraph 
is uniform.  
The description of $\cl S_{\bb{G}}$ is made concrete 
in the case the underlying hypergraph is uniformisable, in that 
the kernel with respect to which the quotient is taken is 
described explicitly as a subspace with canonical 
generators (see Proposition \ref{p_nullsubspace}). 
The universal operator systems $\cl T_{\bb{G}}$ (resp. 
$\cl R_{\bb{G}}$) of quantum (respectively classical)
probabilistic models of $\bb{G}$ are also defined, and the 
equality between different classes of probabilistic models of $\bb{G}$
is characterised via the equality of different operator system structures 
on the linear space, underlying all three operator systems 
$\cl S_{\bb{G}}$, $\cl T_{\bb{G}}$ and $\cl R_{\bb{G}}$ (Theorem \ref{p_OMINs}).

In Section \ref{s_ds} we examine the property of a 
contextuality scenario being dilating. We show that the 
scenario giving rise to quantum magic squares \cite{doi:10.1063/5.0022344} 
is not dilating and exhibit both necessary and sufficient conditions 
for dilatability. This allows us to realise, in a particular special case, 
the operator system $\cl S_{\bb{G}}$
as a modular co-product in the sense of \cite{coprod23}. 
A further, C*-algebraic, characterisation of dilating contextuality scenarios is 
given in Section \ref{ss_univcov} (see Corollary \ref{c_dicha}), 
after examining stochastic operator matrices
relative to a hypergraph, and their universal ternary rings of operators. 
The latter operator matrices reduce to stochastic 
operator matrices in the sense of 
\cite{Todorov2020QuantumNC} when the hypergraph is a Bell scenario, 
and to bistochastic operator matrices \cite{BHTT2}
in the case of magic square hypergraphs; 
in these two cases, they lie at the base of the quantum analogues 
of the correlation chain (\ref{eq_hie})
studied in \cite{Todorov2020QuantumNC} and \cite{BHTT2}.

In Section \ref{s_NSpm} we examine no-signalling probabilistic models, 
defined over a given pair of contextuality scenarios.
We distinguish probabilistic models arising from employing 
projective or general positive operator representations. 
We show that in the case where the contextuality scenario is dilating,
these two paradigms lead to the same probabilistic models of 
quantum and approximately quantum types. 
The main results in this section are the operator algebraic characterisations, 
contained in Theorems \ref{th_cteprop}, \ref{th_qcqatilde} and \ref{th_Qqcstates}. 
 
In Section \ref{ss_coherent}, we define coherent probabilistic models, and 
provide a tracial characterisation of the quantum commuting, the approximately 
quantum, the quantum and the local class in terms of traces on the hypergraph 
C*-algebra of different kind. 
Finally, Section \ref{s_CEP} is dedicated to the aforementioned 
equivalences of the Connes Embedding Problem in terms of contextuality scenarios and 
their probabilistc models. 

\smallskip

We finish this section with setting basic notation. 
For a Hilbert space $\cl H$, denote by $\cl B(\cl H)$ the 
C*-algebra of all bounded linear operators on $\cl H$. 
We write $\cl B(\cl H)^+$ for the cone of positive operators in $\cl B(\cl H)$,
denote by $I_{\cl H}$ the identity operator 
on $\cl H$, and write $I = I_{\cl H}$ when $\cl H$ is understood from the context.
We write $\cl T(\cl H)$ for the ideal of trace class operators on $\cl H$, and 
$\cl T(\cl H)^+$ for its canonical positive cone. 
We use $\langle \cdot,\cdot\rangle$ to denote both the inner product of a Hilbert space 
(assumed linear on the first variable)
and the duality between $\cl B(\cl H)$ and $\cl T(\cl H)$, arising from the 
canonical identification $\cl T(\cl H)^* = \cl B(\cl H)$ in the Banach space sense. 
If $\xi,\eta\in \cl H$, we let $\xi\eta^*$ be the rank one operator, 
given by $(\xi\eta^*)(\zeta) = \langle\zeta,\eta\rangle \xi$. 
We let $M_{n,m}$ denote the space of all $n$ by $m$ complex matrices, 
and write $M_n = M_{n,n}$. 
Making the natural identification $M_n = \cl B(\bb{C}^n)$, 
we write $I_n = I_{\bb{C}^n}$. For $A\in M_n$, we let $\tr{A}$ be the 
(non-normalised) trace of $A$. 
If $V$ is a finite set, we set $\bb{C}^V = \bb{C}^{|V|}$, 
$\bb{R}^V = \bb{R}^{|V|}$ and $M_V = M_{|V|}$, and 
let $(\delta_x)_{x\in V}$ the canonical orthonormal basis of 
$\bb{C}^V$; thus,
$ (\delta_{x}\delta_{x}^{*})_{x,x'}$ is the canonical matrix unit system in $ M_{V}$. 
We denote by $\cl X^{\rm d}$ the dual of a normed space $\cl X$; in case $\cl X$ is an 
operator system, $\cl X^{\rm d}$ is equipped with the canonical dual matrix ordering.

%%%%%%%%%%%%%%%%%%%%%%%%%%%%%%%%%%%%%%%%%%%%%%%%%%%%%%%%%%%%%%%%%%%%%%%%%%%%%%%
%%%%%%%%%%%%%%%%%%%%%%%%%%%%%%%%%%%%%%%%%%%%%%%%%%%%%%%%%%%%%%%%%%%%%%%%%%%%%%%

\section{Preliminaries}\label{s_prel}

A hypergraph is a pair $ \G=(V,E)$, where $ V$ is a finite set and 
$E$ is a non-empty set of subsets of $V$.  The elements of $V$ are called vertices and the elements of $E$, hyperedges (or simply edges). 
A \textit{contextuality scenario} \cite{Acn2015ACA}
is a hypergraph $\G=(V,E) $ such that $ \cup_{e \in E}e = V$.
Given a contextuality scenario $ \G=(V,E)$, its vertices 
will be called 
\textit{outcomes} and its edges -- \textit{measurements}.

Let $\G=(V,E)$ be a contextuality scenario. 
A \textit{probabilistic model} of $\G$ 
is an assignment $p : V \rightarrow [0,1]$ such that
\begin{align*}
    \sum_{x \in e}p(x)=1, \; \text{ for every } \; e \in E.
\end{align*}
The set of all probabilistic models of $\G$ will be denoted by $ \Gg(\G)$. 
Note that not every contextuality scenario admits a probabilistic model (see Figure \ref{fig:my_label}). So, $ \Gg(\G)$ is a convex, possibly  empty, subset of $ \R^{V}$.

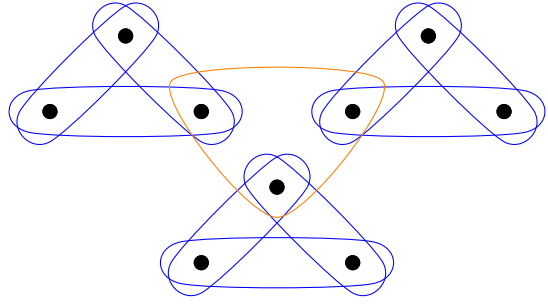
\begin{figure}
    \centering
  \begin{tikzpicture}
   \node (v1) at (0,2) {};
    \node (v2) at (1,3) {};
    \node (v3) at (2,2) {};
    \node (v4) at (4,2) {};
    \node (v5) at (5,3) {};
    \node (v6) at (6,2) {};
    \node (v7) at (3,1) {};
    \node (v8) at (2,0) {};
    \node (v9) at (4,0) {};

    \draw [blue] plot  [smooth cycle] coordinates {(1,3.4) (1.4,3) (0,1.6) (-0.4,2) };
    
     \draw [blue]  plot [smooth cycle] coordinates { (5,3.4)  (5.4,3) (4,1.6) (3.6,2)  }; 
     
      \draw [blue]  plot [smooth cycle] coordinates {(3,1.4) (3.4,1) (2,-0.4) (1.6,0) };
      
       \draw [blue]  plot [smooth cycle] coordinates { (-0.275,2.275) (-0.275,1.725) (2.275,1.725) (2.275,2.275) };
 
     \draw [blue]  plot [smooth cycle] coordinates {(3.725,2.275) (3.725,1.725) (6.275,1.725) (6.275,2.275) };
     
      \draw [blue]  plot [smooth cycle] coordinates {(1.725,0.275) (1.725,-0.275) (4.275,-0.275) (4.275,0.275) };
      
      \draw [blue]  plot [smooth cycle] coordinates {(1,3.4) (0.6,3) (2,1.6) (2.4,2) };
      
      \draw [blue]  plot [smooth cycle] coordinates {(5,3.4) (4.6,3) (6,1.6) (6.4,2) };

      \draw [blue]  plot [smooth cycle] coordinates {(3,1.4) (2.6,1) (4,-0.4) (4.4,0) };
      
      \draw [orange!100] plot [smooth cycle] coordinates {(1.6,2.4) (4.4,2.4) (3,0.6)  };

    \foreach \v in {1,2,...,9} {
        \fill (v\v) circle (0.1);
    }

    \fill (v1) circle (0.1) node [right] { };
    \fill (v2) circle (0.1) node [below left] { };
    \fill (v3) circle (0.1) node [left] { };
    \fill (v4) circle (0.1) node [below] { };
    \fill (v5) circle (0.1) node [below right] { };
    \fill (v6) circle (0.1) node [left] { };
   \fill (v7)  circle (0.1) node [right]
   { };
   \fill (v8)  circle (0.1) node [right]
   { };
   \fill (v9)  circle (0.1) node [left]
   { };

\end{tikzpicture}
    \caption{A scenario without a probabilistic model \cite{Acn2015ACA}.}
    \label{fig:my_label}
\end{figure}

A contextuality scenario that admits a probabilistic model will be called \textit{non-trivial};
all contextuality scenarios in the paper will be assumed to be non-trivial.
Let $\G=(V,E)$ be a contextuality scenario. 
A \textit{positive operator representation} \textit{(POR)} of $\G$ on a 
Hilbert space $\Hh$ is a collection $ (A_{x})_{x\in V} \subseteq \Bh^{+}$ 
such that 
\begin{align*}
     \sum_{x\in e}A_{x} = I_{\Hh}, \; \text{ for every } \; e\in E.
\end{align*}
A POR of $\G$ will be called a \textit{projective representation (PR)}, if $A_{x}$ is a projection for every $ x \in V$, \textit{finite dimensional}, if $ \Hh$ is 
finite dimensional, and \textit{classical} if $A_x A_y = A_y A_x$, $x,y\in V$.

We note that 
any probabilistic model on a contextuality scenario $\G$ 
is a POR on the 
one dimensional Hilbert space $\bb{C}$. 
Conversely, if $\Hh$ is a Hilbert space, $(A_{x})_{x\in V}$ is a POR of 
$\G$ and $\psi \in \Hh$ is a unit vector, 
then the assignment $ p(x)=\sca{A_{x}\psi,\psi}$, $x\in V$, is a probabilistic model of $\G$.

Let $(A_x)_{x\in V}$ and $(B_x)_{x\in V}$ be POR's of the 
contextuality scenario $\bb{G} = (V,E)$, 
acting on Hilbert spaces $H$ and $K$, respectively. 
We say that $(B_x)_{x\in V}$ is a {\it dilation} of $(A_x)_{x\in X}$ if there exists an 
isometry $V : H\to K$ such that $A_x = V^* B_x V$ for all $x\in V$. 
A POR $(A_x)_{x\in V}$ of $\bb{G}$ is called {\it dilatable} if
it admits a dilation to a PR of $\bb{G}$. 
We note that there exist contextuality scenarios that admit a POR but not a PR.
An example is furnished by the contextuality scenario $\G_1$ on 
Figure \ref{fig:NoPRs}: 
given a Hilbert space $\cl H$, the scenario $\G_1$ admits a unique POR  
$(A_i)_{i=1}^3$ on $\cl H$, given by
$A_1 = A_2 = A_3 = \frac{1}{2} I_{\Hh}$. 
In particular, none of the POR's of $\bb{G}_1$ admits a dilation to a PR.

\begin{figure}
    \centering
  \begin{tikzpicture}
   \node (v1) at (0,2) {};
    \node (v2) at (1,3) {};
    \node (v3) at (2,2) {};

    \draw [blue]  plot [smooth cycle] coordinates {(1,3.4) (1.4,3) (0,1.6) (-0.4,2) };

       \draw [blue]  plot [smooth cycle] coordinates { (-0.275,2.275) (-0.275,1.725) (2.275,1.725) (2.275,2.275) };

      \draw [blue]  plot [smooth cycle] coordinates {(1,3.4) (0.6,3) (2,1.6) (2.4,2) };

    \foreach \v in {1,2,3} {
        \fill (v\v) circle (0.1);
    }

    \fill (v1) circle (0.1) node [below] { };
    \fill (v2) circle (0.1) node [below left] { };
    \fill (v3) circle (0.1) node [left] { };
   
   \node at (0,1.3) {$x_1$};
\node at (2,1.3) {$x_2$};
\node at (1.7,3.2) {$x_3$};
\node at (-0.5,3) {};

\end{tikzpicture}
    \caption{A scenario $\G_1$ with no projective representation \cite{Acn2015ACA}.}
    \label{fig:NoPRs}
\end{figure}
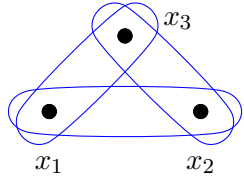

Let $\G = (V,E)$ be a contextuality scenario.
A probabilistic model $p\in \Gg(\G)$ is called {\it deterministic} 
if $ p(x)\in \{0,1\}$ for every $x \in V$, 
and {\it classical} if it is a convex combination of deterministic 
models, that is, if there exist an integer $ k \geq 0$, 
deterministic models $ p_{i} $ 
and scalars
$\lambda_{i}\in [0,1]$, $i \in [k]$, 
such that $\sum_{i=1}^{k}\lambda_{i}=1$ 
and
    \[ p(x) =\sum_{i=1}^{k}\lambda_{i}p_{i}(x), \ \ \ x\in V.   \] 
We denote the set of all classical probabilistic models of $\G$, by $\Cc(\G)$, 
and note that it is a convex polytope.
An application of Carath\'eodory's Theorem shows that $\Cc(\G)$ is closed.

A probabilistic model $p$ of the contextuality scenario $\G = (V,E)$ 
is called {\it quantum} if there exists a Hilbert space $\Hh$, 
a PR $(P_{x})_{x\in V}$ of $\bb{G}$ on a Hilbert space 
$\Hh$ and a unit vector $\psi \in \Hh$, such that 
$$p(x) = \sca{P_{x}\psi,\psi}, \ \ \ x\in V.$$
The set of all quantum models on $\G$ is denoted by $\Qq(\G)$.
We note that the elements of $\Qq(\G)$ can equivalently be defined 
using an arbitrary normal state instead of a vector state, that is, letting 
$p(x) = \sca{P_{x},\omega}$, where $\omega \in \cl T(H)^+$ has trace one; this 
follows from considering the purification of $\omega$ in the Hilbert space $H\otimes\ell^2$.

By \cite[Proposition 5.1.2]{Acn2015ACA}, 
$\Qq(\G)$ is convex; we have the inclusions 
\begin{equation}\label{eq_inccqg}
    \Cc(\G) \subseteq \Qq(\G) \subseteq  \Gg(\G).
\end{equation}
The first inclusion in (\ref{eq_inccqg}) can be strict; indeed, by a
reformulation of the Kochen-Specker Theorem in the context of hypergraphs
\cite{CABELLO1996183, Cabello2008ExperimentallyTS, Acn2015ACA}, there exists 
a contextuality scenario $\G_{KS}$ such that 
$\Cc(\G_{KS})=\emptyset$, while $ \Qq(\G_{KS}) \neq \emptyset$.
We note that it is not known if $\Qq(\G_{KS})=\Gg(\G_{KS})$ and,
more generally, if there exists a contextuality scenario $\G$ such that 
$\Cc(\G)= \emptyset$ while $\Qq(\G)=\Gg(\G) \neq \emptyset$. 
Further, a contextuality 
scenario for which the second inclusion in (\ref{eq_inccqg}) is strict
is displayed in Figure \ref{fig:NoPRs}, where 
$\Qq(\G_{1}) = \emptyset$, while $\Gg(\G_{1})$ is a singleton.

We recall some basic facts and notions 
related to operator systems, and refer the 
reader to \cite{Pa} for details. 
An \textit{operator space} is a subspace $\cl S\subseteq\cl B(\cl H)$ for some 
Hilbert space $\cl H$. We let $M_n(\cl S)$ be the 
subspace of all $n$ by $n$ matrices with entries in $\cl S$; we note that 
$M_n(\cl S)$ can be viewed as a subspace of $\cl B(\cl H^n)$  identifying $M_n(\cl B(H))$ with $\cl B(\cl H^n)$. 
If $\cl S$ and $\cl T$ are operator spaces and 
$\phi : \cl S\to \cl T$ is a linear map, 
we let $\phi^{(n)} : M_n(\cl S)\to M_n(\cl T)$ be the (linear) map, given by 
$\phi^{(n)}([a_{i,j}]_{i,j}) = [\phi(a_{i,j})]_{i,j}$.
A \textit{(concrete) operator system} is an operator space 
$\cl S \subseteq \cl B(\cl H)$,
such that $I_{\cl H}\in \cl S$ and $A\in \cl S \Rightarrow A^*\in \cl S$. 
Every operator system $\cl S$ is an {\it abstract operator system} in the sense that 
(a) $\cl S$ is a linear *-space; (b) the real vector 
space $M_n(\cl S)_h$ of all hermitian elements in the 
*-space $M_n(\cl S)$ of all $n$ by $n$ matrices with entries in $\cl S$
is equipped with a proper cone $M_n(\cl S)^+$
in that $M_n(\cl S)^+ \cap (-M_n(\cl S)^+) = \{0\}$, $n\in \bb{N}$, and
(c) the family 
$(M_n(\cl S)^+)_{n\in \bb{N}}$ is a \textit{matrix ordering} in that 
$T^* M_n(\cl S)^+ T\subseteq M_m(\cl S)^+$ for all $n,m\in\bb{N}$ and all $T\in M_{n,m}$, 
which admits an Archimedean order unit. 
If $\cl S$ and $\cl T$ are abstract operator systems, 
a linear map $\phi : \cl S\to \cl T$
is called \textit{positive} if $\phi (\cl S^+)\subseteq \cl T^+$,
\textit{completely positive} if $\phi^{(n)}$ is positive for every $n\in \bb{N}$, 
\textit{unital} if $\phi(I) = I$, 
and a \textit{complete order isomorphism} if $\phi$ is completely positive, bijective, and 
its inverse $\phi^{-1}$ is completely positive. 
By virtue of the Choi-Effros Theorem \cite[Theorem 13.1]{Pa}, 
every abstract operator system is completely order isomorphic to an
operator system. 
A \textit{state} of an operator system 
$\cl S$ is a positive unital linear functional; 
we denote by $S(\cl S)$ the (convex) set of all states of $\cl S$. 
By \cite[Corollary 4.5]{CE2}, if 
$\cl S$ is a finite dimensional operator system then its 
(matrix ordered) dual $\cl S^{\rm d}$ is an operator system.
%with an Archimedean order unit a suitable faithful state of $\cl S$.

Recall that a C*-cover of an operator system $\cl S$ is a pair
$(\cl A,\iota)$, where $\cl A$ is a unital C*-algebra and $\iota : \cl S\to \cl A$
is a unital complete order isomorphism, such that $\iota(\cl S)$ generates $\cl A$
as a C*-algebra. The C*-envelope 
$(C^{*}_{\rm e}(\cl S),\iota_{\rm e})$ of $\cl S$
is the C*-cover with the property that if $(\cl A,\iota)$ is a C*-cover of $\cl S$ then 
there exists a *-epimorphism $\pi : \cl A\to C^*_{\rm e}(\cl S)$ such that 
$\pi\circ \iota = \iota_{\rm e}$ (see e.g. \cite[Chapter 15]{Pa}).

A \textit{kernel} \cite{Kavruk2010QuotientsEA} of an operator system $\cl S$ is a 
subspace $\cl J\subseteq \cl S$ of the form $\cl J = \ker(\phi)$ for some 
completely positive map $\phi$ defined on $\cl S$; a subspace 
$\cl J \subseteq \cl S$ is a kernel if and only if there exists a family 
$\{f_{\alpha}\}_{\alpha\in \bb{A}}$ of states of $\cl S$ such that 
$\cl J = \cap_{\alpha\in \bb{A}} \ker(f_{\alpha})$
\cite[Proposition 3.1]{Kavruk2010QuotientsEA}. 
If $\cl J$ is a kernel in $\cl S$, the linear quotient $\cl S/\cl J$ admits 
a natural operator system structure \cite{Kavruk2010QuotientsEA}
with the (universal) property that, 
for every operator system $\cl T$ and every completely positive map 
$\phi : \cl S\to \cl T$ with $\cl J\subseteq \ker(\phi)$, the 
naturally induced map $\dot{\phi} : \cl S/\cl J\to \cl T$ is completely positive.

Let $\cl S$ be an operator system. 
Given a linear subspace $J\subseteq \cl S$, we define its {\it kernel cover} $\cl J$ by letting
$$\cl J = \cap\{\ker(\phi) : \phi \mbox{ a state of } \cl S \mbox{ with } J\subseteq \ker(\phi)\}.$$
By \cite[Proposition 3.1]{Kavruk2010QuotientsEA}, $\cl J$ is a kernel in 
$\cl S$. 
We will denote by $q_{\cl J}$ the canonical quotient map from an operator system $\cl S$ onto its quotient $\cl S/\cl J$. 
%Standard arguments imply t
The following statement will be used in 
Section \ref{s_opsys}.

\begin{proposition}\label{p_cov}
Let $\cl S$ be an operator system, $J\subseteq \cl S$ be a linear subspace and $\cl J$ be its kernel cover. 
If $\cl J\neq \cl S$, then $\cl S/\cl J$ is the unique,
up to a complete order isomorphism, operator system such that, whenever 
$\cl T$ is an operator system and $\phi : \cl S\to \cl T$ is a unital completely positive map with $J\subseteq \ker(\phi)$ then 
there exists a unital completely positive map $\psi : \cl S/\cl J \to \cl T$ such that $\phi(a) = \psi(q_{\cl J}(a))$, $a\in \cl S$.
\end{proposition}

\begin{proof}
Set $\tilde{\cl S} = \cl S/\cl J$. We first show that $\tilde{\cl S}$ has the stated universal property. 
Suppose that $\cl T$ is an operator system and $\phi : \cl S \to \cl T$ is a completely positive map with $J\subseteq \ker(\phi)$. 
We show that $\cl J\subseteq \ker(\phi)$. Let $x\in \cl J$.
In order to show that $\phi(x) = 0$, it suffices to show that 
$t(\phi(x)) = 0$ for every state $t : \cl T\to \cl C$; 
indeed, since every $t'\in \cl T^{\rm d}$ is a linear combination of 
four states, this would imply that $t'(\phi(x)) = 0$ for every $t'\in \cl T^{\rm d}$, 
and hence that $\phi(x) = 0$. 
Let, therefore, $t\in S(\cl T)$. Then $t\circ \phi$ is a positive functional 
on $\cl S$ that annihilates $J$. Thus $(t\circ \phi)(x) = 0$ as desired.
Now the universal property of the quotient \cite{Kavruk2010QuotientsEA}
implies that 
there exists a (unique) completely positive map $\psi : \tilde{\cl S}\to \cl T$ such that 
$\phi(a) = \psi(q_{\Jj}(a))$, $a\in \cl S$. 

Suppose that $\cl R$ is a quotient operator system of $\cl S$ with a corresponding quotient map $q_{\cl R}$ 
whose kernel contains $J$, such that if
$\phi : \cl S\to \cl T$ is a completely positive map with $J\subseteq \ker(\phi)$ then 
there exists a completely positive map $\psi : \cl R \to \cl T$ such that $\phi(a) = \psi(q_{\Rr}(a))$, $a\in \cl S$.
Taking $\cl T = \tilde{\cl S}$, we obtain a unital completely positive map
$\psi : \cl R\to \tilde{\cl S}$ such that $q_{\cl R}(a) = \psi(q_{\Jj}(a))$, $a\in \cl S$.
On the other hand, taking $\cl T = \cl R$, by the previous paragraph, there exists a unital completely positive map 
$\theta : \tilde{\cl S}\to \cl R$ such that $q_{\Jj}(a) = \theta(q_{\cl R}(a))$, $a\in \cl S$.
It is now clear that $\theta$ is a (unital) complete order isomorphism. 
\end{proof}

%%%%%%%%%%%%%%%%%%%%%%%%%%%%%%%%%%%%%%%%%%%%%%%%%%%%%%%%%%%%%%%%%%%%%%%%%%%%
%%%%%%%%%%%%%%%%%%%%%%%%%%%%%%%%%%%%%%%%%%%%%%%%%%%%%%%%%%%%%%%%%%%%%%%%%%%%

\section{Operator systems from hypergraphs}\label{s_opsys}

In this section, we introduce operator systems, universal 
for the different types of operator representations of a contextuality 
scenario in that the latter correspond precisely to the 
unital completely positive maps defined on them. 
This allows us to express properties such as dilatability in terms of 
operator system equalities. The results will be used subsequently 
to provide operator representations of no-signalling correlations over 
product contextuality scenarios.

%%%%%%%%%%%%%%%%%%%%%%%%%%%%%%%%%%%%%%%%%%%%%%%%%%%%%%%%%%%%%%%%%%%%%%%%%%%%

\subsection{The universal operator system of a positive operator representation}
\label{universalopsection}

We fix a contextuality 
scenario $\G= (V,E)$, and we write  
$E = \{e_1,e_2,\dots,e_d\}$. 
For each $e\in E$, we denote by $\ell^\infty_e$ a copy of $\ell^\infty_{|e|}$, 
and set  
$$\Ss := \ell^{\infty}_{e_1} \oplus \cdots \oplus \ell^{\infty}_{e_d}\, .$$ 
For $x\in V$, 
we denote by $\delta_{x}^{e}$ the element of $\ell^{\infty}_{e}$ 
whose $x$-th entry is $1$, and the remaining entries are zero. 
For $i,j\in [d]$ with $i < j$, 
and $x\in e_i\cap e_j$, we set 
$$\alpha_{j} = 
\underbrace{1\hspace{-0.05cm}\oplus 0 \hspace{-0.05cm}\oplus \cdots  
\oplus 0 \hspace{-0.05cm} \oplus \hspace{-0.05cm} -1 \hspace{-0.05cm}\oplus 0 
\hspace{-0.05cm}\oplus \cdots\oplus 
\hspace{-0.05cm} 0}_{d \mbox{ \tiny terms}},$$
where the entry $-1$ occupies the $j$-th coordinate, and 
$$\beta_{i,j}^{x} = 
\underbrace{0 \oplus\cdots \oplus 0 \oplus \delta_{x}^{e_{i}}\oplus 0 \oplus\cdots 
\oplus 0 \oplus -\delta_{x}^{e_{j}} \oplus 0 \oplus\cdots \oplus 0}_{d \mbox{ \tiny terms}},$$
where the term $\delta_{x}^{e_{i}}$ (resp. $-\delta_{x}^{e_{j}}$) is at the 
$i$-th (resp. $j$-th) coordinate. 
Define
\begin{equation}
\label{thekernel}
    \Jj 
    \hspace{-0.07cm} := \hspace{-0.05cm} 
    \spann \{\alpha_{j}, \beta_{i,j}^{x}
    :  i, j \in [d], \, i < j \text{ and } x\in e_{i}\cap e_{j}\} 
\end{equation}
and set 
\[ \Jjc := \{ u\in \Ss : \inn{p,u} = 0, \ p \in \Gg(\G) \} \]
where $$ \inn{p,u} = \sum_{e\in E}\sum_{x\in e} p(x) u_{x}^{e}, \quad u = \oplus_{e\in E} (u_{x}^{e})_{x\in e}. $$

Let also 
    \begin{multline}
    \mathcal{L}_{\G} = \Big\{ (\lambda_{x}^{1})_{x\in e_{1}}\oplus \cdots \oplus (\lambda_{x}^{d})_{x\in e_{d}}: \sum_{x\in e_{i}}\lambda_{x}^{i}=\sum_{x\in e_{j}}\lambda_{x}^{j} \\
    \text{ and } \lambda_{x}^{i}=\lambda_{x}^{j} \text{ for all } x\in e_{i}\cap e_{j},
    i,j\in [d]\Big\},
\end{multline}
viewed as a selfadjont subspace of $\Ss$. 
It is straightforward to check that $ \cl J^{\perp} = \Ll_{\G}$. 
Thus, if $f : \cl S\to \bb{C}$ is a positive functional that annihilates 
$\cl J$ then $f\in \mathcal{L}_{\G}^+$. 
It follows that $f$ annihilates $\tilde{\cl J}$, and hence 
$\tilde{\cl J}$ coincides with the kernel cover of $\cl J$ in the sense
of Proposition \ref{p_cov}. 
Set $\cl S_{\G} = \cl S/\tilde{\cl J}$.

\begin{remark}\label{r_Jj0}
If the hyperedges of $\bb{G}$ are mutually disjoint
then no elements of the form $\beta_{i,j}^{x}$ appear on the 
right hand side of (\ref{thekernel}). It follows 
\cite{bbf61d43a5de485bad4fe7b6c908f9a2} that $\Jj$ is a kernel in $\cl S$ and 
that the quotient $\cl S_{\G}$ coincides with the coproduct
$\ell^{\infty}_{e_1}\oplus_1 \ell^{\infty}_{e_2} \oplus_1\cdots\oplus_1 \ell^{\infty}_{e_d}$  
of the operator systems 
$\ell^{\infty}_{e_1}, \ell^{\infty}_{e_2},\cdots,\ell^{\infty}_{e_d}$
as defined in \cite{bbf61d43a5de485bad4fe7b6c908f9a2}. 
\end{remark}

\begin{remark} \label{r_null}
We will shortly exhibit examples of 
contextuality scenarios $\G$, for which the subspace $ \Jj$ is a null subspace
\cite{Kavruk2014}, that is, it does not contain non-zero positive elements. 
In this case, $\Jj$ is a 
completely proximinal kernel \cite{Kavruk2010QuotientsEA};
in particular, $\Ss_{\G} = \Ss/\Jj$ completely order isomorphically
and
\[ M_n(\Ss/\Jj)^+ \hspace{-0.1cm} = \hspace{-0.1cm}
\{ [u_{i,j} \hspace{-0.05cm} + \hspace{-0.05cm} \Jj] \hspace{-0.1cm}\in \hspace{-0.1cm}\Mn(\Ss/\Jj) 
: \exists \ k_{i,j} \in \Jj 
\text{ s.t. } [u_{i,j} \hspace{-0.05cm} + \hspace{-0.05cm} k_{i,j} ] \hspace{-0.1cm}\in \hspace{-0.1cm}\Mn(\Ss)^{+} \}. \]
%we thus have a simpler presentation of $\Ss_{\G}$ and its cones.
\end{remark}

Let $\G=(V,E)$ be a  contextuality scenario. 
For $e\in E$, let 
$\iota_{e} : \ell^{\infty}_{e} \rightarrow \oplus_{f\in E}\ell^{\infty}_{f}$
be the natural embedding and, recalling that 
$ q : \Ss \rightarrow \Ss /\Jjc$ is the quotient map, let 
$i_{e} : \ell^{\infty}_{e} \rightarrow \Ss_{\G}$ be the map given by 
\begin{equation}\label{eq_|E|}
i_{e}(u) = |E| (q\circ \iota_{e})(u), \ \ \ u\in \ell^{\infty}_{e}.
\end{equation}
To ease the notation, we set 
%\begin{equation}\label{eq_a_x}
$$a_{x} := i_{e}(\delta_{x}^{e}), \; \; \; x\in V,$$
%\end{equation}
and note that $a_x$ only depends on $ x\in V$ and not on $e\in E$. 
We have that 
$\Ss_{\G} = \spann\{a_{x} : x\in V\}$.
The next statement shows that 
$\Ss_{\G}$ is the universal operator system for 
positive operator representations of $\G$.

\begin{theorem}\label{universalprop}
Let $ \G = (V,E)$ be a non-trivial contextuality scenario. 
If $ \Phi: \Ss_{\G} \rightarrow \Bh$ is a unital completely positive map
then $(\Phi(a_{x}))_{x\in V}$ is a POR of $\G$.
Conversely, 
if $ (A_{x})_{x\in V} \subseteq \Bh$ is a POR of $ \G$ then there exists a 
unique unital completely positive map $ \Phi: \Ss_{\G} \rightarrow \Bh$ such that $ \Phi(a_{x})=A_{x}$, $ x\in V$.

Suppose that $\Rr$ is an operator system and   \label{uniqueness}
$r_{e} : \ell^{\infty}_{e} \rightarrow \Rr$ 
are unital completely positive maps, $e\in E$, 
with the property that 
for every POR 
$(A_{x})_{x\in V}$ of $\G$ 
there exists a unique unital completely positive map $ \Psi : \Rr \rightarrow 
\Bh$ with $ \Psi(r_{e}(\delta_{x}^{e})) = A_{x}$, $ x\in V$. Then there exists a unital complete order isomorphism $ \theta: \Rr \rightarrow \Ss_{\G}$
such that $\theta \circ r_e = i_e$, $e\in E$.
\end{theorem}

\begin{proof}
%We recall the canonical generators $a_{x}$, $x\in V$, of $\cl S_{\G}$
%(see (\ref{eq_a_x})) and the space $ \Ss = \oplus_{e}\ell^{\infty}_{e}$. 
The first statement follows from the fact that 
$a_x\in \cl S_{\G}^+$, $x\in V$, and $\sum_{x\in e} a_x = 1$, $e\in E$. 
Let $ (A_{x})_{x\in V} \subseteq \Bh$ be a positive operator representation of $ \G$. For each $e\in E$, 
the family $ (A_{x})_{x\in e}$ is a POVM, and 
hence the linear map 
$\phi_{e} : \ell^{\infty}_{e} \rightarrow \Bh$, given by 
$ \phi_{e}(\delta_x^{e}) = A_{x}$, is unital and completely positive. 
Define $\Phi_0 : \Ss \rightarrow \Bh$ by letting
$$\Phi_0\left((u_{e})_{e\in E}\right) = \frac{1}{|E|}\sum_{e\in E}\phi_{e}(u_{e}),
\ \ \ (u_{e})_{e\in E}\in \cl S,$$
and note that $\Phi_0$ is unital, completely positive, and
$\Jj \subseteq \ker \Phi_0$. 
We have that $ \Jjc \subseteq \ker \Phi_0$; hence, 
by Proposition \ref{p_cov}, there exist a unique unital completely positive 
map $\Phi: \Ss/ \Jjc \rightarrow \Bh $ such that 
$\Phi\left((u_{e})_{e\in E}+\Jjc\right) = \Phi_0\left((u_{e})_{e\in E}\right)$. 
Note that $\Phi(a_x) = \phi_{e}(\delta_{x}^{e}) = A_{x}$, $ x\in V$. 

Let $\cl R$ and $r_e$, $e\in E$, be as stated. 
Without loss of generality, assume that 
$\Ss_{\G}\subseteq \cl B(\Hh)$ for some Hilbert space $\Hh$. 
As $(a_{x})_{x\in V}$ is a positive operator representation of $ \G$, 
there exists a unital completely positive map 
$\theta : \Rr \rightarrow \Ss_{\G}$ such that 
$\Psi(r_{e}(\delta_{x}^{e})) = a_{x}$, $x\in V$. 
Swapping the roles of $\cl R$ and $\cl S_{\G}$ and using the 
universal property of $\cl S_{\G}$ established in the previous paragraphs, 
we have that there exists a 
unital completely positive map $ \Phi :\Ss_{\G} \rightarrow \Rr$  
such that $\Phi(a_x) = r_{e}(\delta_{x}^{e})$,
$x\in e$. Finally, note that 
$ (\theta \circ \Phi )(a_x) = a_x$ and 
$(\Phi \circ \theta)(r_{e}(\delta_{x}^{e})) = r_{e}(\delta_{x}^{e}) $, for every $x\in e$, 
$e\in E$. By uniqueness,
$ \Phi \circ \theta = \id_{\cl R}$ and $\theta \circ \Phi = \id_{\cl S_{\G}}$. 
It follows that $\theta$ is a unital complete order isomorphism.
\end{proof}

Our next aim is to provide a concrete description of the 
matricial order structure of the operator system $\cl S_{\G}$.
We make the identification 
$ M_{n}( \oplus_{e\in E}\ell_{{e}}^{\infty}) = 
\oplus_{e\in E} \oplus_{x\in e} M_{n}$, so that an element $ u \in  M_{n}( \oplus_{e\in E}\ell_{{e}}^{\infty}) $ is written as 
$ u = \oplus_{e\in E} (u_{x}^{e})_{x\in e}$, where $ u_{x}^{e}\in M_{n}$, $e\in E$, $x\in e$. 
To ease notation, we will often write $u_{x}^{i} = u_{x}^{e_i}$, $i\in [d]$. 
We call an element 
$\oplus_{e\in E} (u_{x}^{e})_{x\in e}$ of $M_{n}( \oplus_{e\in E}\ell_{{e}}^{\infty})$
{\it consistent} if $u_{x}^{e} = u_{x}^{f}$ whenever $e,f\in E$ and $x\in e\cap f$. 
For $T\in M_n^+$, let 
\begin{multline*}
    \Ll_{n}(T) = \left\{ \Lambda \hspace{-0.1cm} = \hspace{-0.1cm}
    \oplus_{e\in E} (\Lambda_{x}^e)_{x\in e} : 
    \Lambda \geq 0, \Lambda \mbox{ is consistent and }
     \ \sum_{x\in e} \Lambda_{x}^e = T, e\in E\right\},
\end{multline*}
and set 
$\Ll_{n} = \cup_{T\in M_n^+} \Ll_{n}(T)$,  
considered as a subset of the direct sum $\oplus_{e\in E} \oplus_{x\in e} M_{n}$.
An element $ \Lambda \in \Ll_{n}$ gives rise to a map 
(denoted in the same way) 
$\Lambda : V \rightarrow M_{n}^{+}$, 
given by $\Lambda(x) = \Lambda_{x}^e$, for $x\in e$; 
we thus have an identification between $\Ll_{n}(I)$ and the 
set of all POR's of $\bb{G}$ on the Hilbert space $\bb{C}^n$. 
For $\Lambda\in \Ll_{n}$ and 
$u \in  M_{n}( \oplus_{e\in E}\ell_{{e}}^{\infty}) $, let
\begin{align*}
    \inn{\Lambda, u} : =  \sum_{e\in E}\sum_{x\in e} \tr{u_{x}^{e}(\Lambda_{x}^{e})^{\rmt}};
\end{align*}
we thus further view $\Lambda$ as a linear functional 
(denoted in the same way) on 
$M_{n}( \oplus_{e\in E}\ell_{{e}}^{\infty})$.
Note that  if $\Lambda\in \Ll_{n}(T)$, then 
$$\inn{\Lambda, I} = 
\sum_{e\in E}\sum_{x\in e} \tr{\Lambda_{x}^{e}}
= \sum_{e\in E} \tr{T} = d \tr{T};$$
thus, if $\Lambda\in \Ll_{n}(\frac{1}{dn} I_{n})$ then 
$\Lambda$ is a state  on $ M_{n}( \oplus_{e\in E}\ell_{{e}}^{\infty})$. 
We have that 
\begin{equation}\label{eq_dL1}
\Ll_{1}(1) \equiv \Gg(\G), 
\end{equation}
and that, if $p \in \Ll_{1}$ and 
$u = \oplus_{e\in E} (u_{x}^{e})_{x\in e} \in \Ss$, then  
\[ \inn{p,u} = \sum_{e\in E}\sum_{x\in e} p(x) u_{x}^{e}.  \]

We note the canonical identification $\Mn(\Ss/\Jj) = \Mn(\Ss)/\Mn(\Jj)$, 
which will be used in the sequel.
Since $\Jj$ is a selfadjoint subspace of $\cl S$, the involution on $\cl S$
induces an involution on $\Ss /\Jj$ which, in its own turn, 
induces a canonical involution on the real vector
space $\Mn( \Ss /\Jj)_{h}$ of hermitian elements of $\Mn(\Ss/\Jj)$.

\begin{proposition}\label{p_basicprmo}
    Let $\G = (V,E)$ be a contextuality scenario
    and
    \begin{align*}
    C_{n}:= \{ u +\Mn(\Jj) \in \Mn( \Ss /\Jj)_{h} : \inn{\Lambda, u} \geq 0, \ \Lambda \in \Ll_{n}\}, \ \ \  n\in \N.
\end{align*}
    Then the family $(C_{n})_{n\in \bb{N}}$ is a matrix ordering for the 
    $*$-vector space $ \Ss / \Jj$
     with Archimedean matrix order unit $ 1 +\Jj$.
\end{proposition}

\begin{proof}

%It is straightforward that, if $ p \in \Ll_{1}$ then $\inn{p,u} = 0$ 
%for each of the generators $u$ of $\cl J$; thus, 
%$0 + \cl J \in C_{1}$.
Let $T\in M_n^+$ and $ \Lambda \in \Ll_{n}\left(T\right)$, and note that 
$\Lambda = \oplus_{e\in E} \oplus_{x\in e} \Lambda_{x}^{e}$, where 
$\Lambda_{x}^{e} = [\Lambda_{x}^{e}(i,j)]_{i,j=1}^{n} \in \Mn^{+}$. Since 
$\sum_{x\in e}\Lambda_{x}^{e} = T$, $e\in E$, we have that
\begin{equation}\label{eq_LamKr}
\sum_{x\in e}\Lambda_{x}^{e}(i,j)= \langle Te_j,e_i\rangle, \ \ \ i,j\in [n], e\in E.
\end{equation}
Write $ \Lambda = [\Lambda(i,j)]_{i,j=1}^n \in \Mn( \oplus_{e\in E} \ell^{\infty}_{e} )$, where $\Lambda(i,j) = \oplus_{e\in E}\left(\Lambda_{x}^{e}(i,j)\right)_{x\in e}$.
 
Let $ v = [v(i,j)]_{i,j=1}^n \in \Mn(\Jj)$, 
where $ v(i,j) = \oplus_{e\in E} (v_{x}^{e}(i,j))_{x\in e} \in \Jj$. 
We have
\begin{equation}\label{eq_Lv}
\inn{\Lambda, v} =  \sum_{i=1}^{n} \sum_{j=1}^{n} \inn{\Lambda(i,j), v(i,j)}
    = \sum_{i=1}^{n} \sum_{j=1}^{n}  \sum_{e\in E}\sum_{x\in e}  \Lambda_{x}^{e}(i,j) v_{x}^{e}(i,j).  
\end{equation}
We claim that $\langle \Lambda,W\rangle  =0$ whenever $W\in M_n(\cl J)$. 
Writing $W = [w_{i,j}]_{i,j=1}^n$, we see that it suffices to 
verify that $\inn{\Lambda(i,j), w }=0$ for each of the generators $  w \in \Jj$
and all $i,j\in [n]$.
Let $p,q\in [d]$ with $p\neq q$ and $ w = \alpha_q - \alpha_p$; thus, 
$w$ has the element $1 $ in the direct summand corresponding to 
$ e_p\in E$, the element 
$ -1$ to the one corresponding to  $ e_q \in E $, and zeros anywhere else.
Setting $e' = e_p$ and $e'' = e_q$, we have that
$ w_{x}^{e'}=1$ for all $ x\in e'$, $ w_{x}^{e''} = -1$ for all $ x\in e''$, while $ w_{x}^{e}=0$ for all $ e\not\in \{e',e''\}$. 
By (\ref{eq_LamKr}) and (\ref{eq_Lv}), 
\begin{align*}
    \inn{\Lambda(i,j), w } &=  \sum_{x\in e'} \Lambda_{x}^{e'}(i,j) w_{x}^{e'} +\sum_{x\in e''} \Lambda_{x}^{e''}(i,j) w_{x}^{e''} = \\
    &=\sum_{x\in e'} \Lambda_{x}^{e'}(i,j) -\sum_{x\in e''} \Lambda_{x}^{e''}(i,j) = 0.
\end{align*}
%as, if $ i=k$ then $\sum_{x\in e'} \Lambda_{x}^{e'}(i,i) =\sum_{x\in e''} \Lambda_{x}^{e''}(i,i)= \frac{1}{n \cdot |E|} $, while if $ i \neq k$, $\sum_{x \in e'}\Lambda_{x}^{e'}(i,k) =\sum_{x\in e''} \Lambda_{x}^{e''}(i,k)=0$. Now let  $ w \in \Jj$ be the element with $\delta_{y}^{e'} $ in the $e'$-th direct summand and $- \delta_{y}^{e''}  $ in the $ e''$-th direct summand and zero elsewhere, assuming that $ y\in e' \cap e''$.  
Further, if $w = \beta^y_{p,q}$, where $y\in e'\cap e''$, then, 
by the definition of $\cl L_n$, we have 
%\begin{align*}
$$    \inn{\Lambda(i,j), w } =  \sum_{x\in e'} \Lambda_{x}^{e'}(i,j) w_{x}^{e'} +\sum_{x\in e''} \Lambda_{x}^{e''}(i,j) w_{x}^{e''} 
    = \Lambda_{y}^{e'}(i,j) -\Lambda_{y}^{e''}(i,j) =0.$$
%\end{align*}
%since $ \Lambda_{x}^{e'} = \Lambda_{x}^{e''}$ whenever $ x\in e' \cap e''$. 
Hence $\Ll_{n}$ annihilates every element of $ \Mn(\Jj)$, 
in other words, $0 + M_n(\Jj) \in C_{n}$, $n\in \bb{N}$. 

It is clear that $ C_{n}$ is a cone for each $ n\in \N$. 
We claim that the family $(C_{n})_{n\in \bb{N}}$ is compatible. 
Let $ a \in M_{n,m}$, $ u + \Mn(\Jj) \in C_{n}$ and $\Lambda\in \cl L_m(T)$; then 
$a^{\rmt *} \Lambda a^{\rmt}\in \cl L_n(a^{\rmt *} T a^{\rmt})$
and hence 
$\inn{\Lambda , a^{*} u a} = \inn{a^{\rmt *} \Lambda a^{\rmt}, u } \geq 0$.

Finally, we check that  $1 + \Jj$ is an Archimedean matrix order unit
for the family $(C_{n})_{n\in \bb{N}}$.  
Let $ [u_{i,j} + \Jj]_{i,j=1}^n \in \Mn(\Ss/\Jj)_{h}$, 
and set $ u =[u_{i,j}]_{i,j=1}^n \in \Mn(\Ss)$. 
Since $I_{n} \otimes 1$ is an order unit for 
$\Mn(\Ss)$, there exists an $r > 0$, such that 
$r (I_{n} \otimes 1) - u \geq 0.$
So, 
\begin{align*}
    r( I_{n} \otimes (1 + \Jj)) - [u_{i,j} + \Jj] & = \big( r ( I_{n} \otimes 1) + \Mn(\Jj) \big) -  \big( u + \Mn(\Jj) \big) =\\
    &=( r ( I_{n} \otimes 1) - u )+ \Mn(\Jj) \in  C_{n}, 
\end{align*}
since $ \Mn(\Ss)^{+} + \Mn(\Jj) \subseteq C_{n}$. 
Thus $1+ \Jj $ is a matrix order unit. Suppose that
 \begin{align*}
 ( r ( I_{n} \otimes 1) + u )+ \Mn(\Jj) = r( I_{n} \otimes (1 + \Jj)) + [u_{i,j} + \Jj] \in C_{n} \ \mbox{ for all } r>0;
 \end{align*}
this implies 
\[  r{\rm Tr}(T) + \inn{\Lambda,u } =\inn{\Lambda, r ( I_{n} \otimes 1) + u} \geq 0, \ \ \ 
\Lambda \in \Ll_{n}(T), T\in M_n^+, r > 0. \]
We conclude that 
$\inn{\Lambda, u } \geq 0$ for every $\Lambda \in \Ll_{n}$,
  that is, $ u + \Mn(\Jj) \in C_{n}$, and thus the unit is Archimedean.
\end{proof}

By Proposition \ref{p_basicprmo}, 
if the cones $C_n$ are proper, that is, if $ C_{n}\cap -C_{n} =\{0\}$, $n\in \bb{N}$,  
then $\left(\Ss/\Jj, (C_n)_{n\in \bb{N}}, 1 + \Jj\right)$ is an operator system. 
However, as we will see this is not always the case; thus, in order to obtain a 
canonical operator system from the matrix ordering $(C_n)_{n\in \bb{N}}$, 
we define 
$    \mathcal{N}: = C_{1} \cap -C_{1}$,
and set
\[ \Tilde{C}_{n}:= \{ (u_{i,j} + \mathcal{N}) \in \Mn((\Ss/\Jj) / \mathcal{N}) : [u_{i,j}]\in C_{n} \}.   \]
By \cite[Proposition 4.4.]{Araiza2020AnAC}, the quotient 
$\tilde{\cl S}_{\G}$
of $\Ss/\Jj$ by $ \mathcal{N} $ is an operator system, 
when equipped with the induced cones $ \{\Tilde{C}_{n}\}_{n}$.
The next proposition provides a more straightforward description of $\Ss_{\G}$.

\begin{proposition}\label{p_isomoofop}
    Let $\G=(V,E)$ be a contextuality scenario.
    Then $\Ss_{\G} = \tilde{\cl S}_{\G}$, 
    up to a canonical complete order isomorphism.
\end{proposition}

\begin{proof}
We show that $\tilde{\cl S}_{\G}$ 
satisfies the universal property of 
$\Ss_{\G}$. Let $ (A_{x})_{x\in V} \subseteq \Bh$ be a 
POR  of $\G$, and let
$ \phi_{e} : \ell^{\infty}_{e} \rightarrow \Bh$ be the 
unital completely positive maps, given by 
$ \phi_{e}(\delta_{x}^{e})= A_{x}$, $x\in e$, $e\in E$. 
Define the map $ \Phi_{0} : \Ss \rightarrow \Bh$ by letting
    \begin{align*}
        \Phi_0 ((u_{e})_{e\in E})= \frac{1}{|E|} \sum_{e\in E} \phi_{e}(u_{e})
    \end{align*}
and note that $\cl J \subseteq \ker \Phi_0$. 
Let $ \Phi_{1} : \Ss/ \cl J \rightarrow \Bh$ be the 
induced map, given by 
$ \Phi_{1}(u + \cl J) = \Phi_0(u)$, $u\in \Ss$. 

We claim that $\cl N \subseteq \ker \Phi_{1}$. 
Indeed, let $ u + \cl J \in \cl N$, that is, $ u + \cl J \in C_{1}$ and $ u + \cl J \in -C_{1}$. This implies that $ \sca{p,u}=0 $ for all $ p \in \Ll_{1}$. We will show that $\frac{1}{|E|} \sum_{e\in E} \phi_{e}(u_{e}) =0$. 
Let $s : \Bh \to \bb{C}$ be a state; then 
the states $s\circ \phi_{e} : \ell^{\infty}_{e} \to \C$, $e\in E$,
form a compatible family in the sense that 
if $x\in e\cap e'$ for some $e,e'\in E$ then 
$(s\circ \phi_{e})(\delta_{x}^{e}) = (s\circ \phi_{e})(\delta_{x}^{e'})$. 
Hence, $ p := \oplus_{e\in E}(s \circ \phi_{e}(\delta_{x}^{e}))_{x\in e}$ 
defines an element in $ \Ll_1$ and thus 
\[ s\left( \frac{1}{|E|} \sum_{e\in E} \phi_{e}(u_{e})\right) 
= \frac{1}{|E|} \sum_{e\in E} s\circ \phi_{e}(u_{e}) = \sca{p,u}=0,\]
    and since $ s$ was arbitrary, we conclude that $\cl N \subseteq \ker \Phi_{1}$. 
    
Let $\Phi : \tilde{\cl S}_{\G}\to \cl B(H)$ be the map, given by 
$\Phi((u+\cl J) + \cl N) = \Phi_0(u)$. 
Note that $\Phi_{}$ is  unital; 
we will show that it is completely positive. 
For an element $ [u_{i,j} + \cl J]_{i,j} + M_{n}(\cl N) \in \Tilde{C_{n}}$,
we can ensure that $ [u_{i,j} + \cl J]_{i,j} \in C_{n}$ which implies that 
$ \sca{\Lambda, [u_{i,j}]} \geq 0$ for all $ \Lambda \in \Ll_{n}$. 
We will show that $ [\Phi_0(u_{i,j})] \in M_{n}(\Bh)^{+}$, equivalently that 
$f([\Phi_0(u_{i,j})]_{i,j}) \geq 0$ for every state $ f $ on $\Mn(\Bh)$. 
Pick such a state $f$ and write the element 
$u = [u_{i,j}]_{i,j} \in \Mn(\Ss)$ as
$ u = \oplus_{e\in E} \oplus_{x\in e}[u_{x}^{e}(i,j)]_{i,j} $, 
using the identification 
$ \Mn(\oplus_{e\in E}\ell^{\infty}_{e}) = \oplus_{e\in E}\oplus_{x\in e} \Mn$. 
Then 
\begin{align*}
    [\Phi_0(u_{i,j})]_{i,j} = \frac{1}{|E|} \sum_{e\in E} \sum_{x\in e} [u_{x}^{e}(i,j)\phi_{e}(\delta_{x}^{e})]_{i,j},
\end{align*}
equivalently,
\[ \left[\Phi_0(u_{i,j})\right]_{i,j}\textbf{} = 
\left[\frac{1}{|E|} \sum_{e\in E} \phi_{e}(u_{e}(i,j))\right]_{i,j} 
= \frac{1}{|E|} \sum_{e\in E} \phi_{e}^{(n)}\left([u_{e}(i,j)]_{i,j}\right). \]
where $ u_{e}(i,j) = (u_{x}^{e}(i,j))_{x\in e}$.
Let $f_{i,j} : \cl B(\cl H)\to \bb{C}$ be the map, given by 
$f_{i,j}(T) = f(\epsilon_{i,j}\otimes T)$ (here $\{\epsilon_{i,j}\}_{i,j=1}^{n}$
is the canonical matrix unit system in $M_n$). 
We note that $ f \circ \phi_{e}^{(n)} : \Mn(\ell^{\infty}_{e}) \rightarrow \C$ are compatible states that give rise to an element 
$\Lambda = (\Lambda_x)_{x\in V} \in \Ll_{n}$, 
where $\Lambda_x = [f_{i,j}(A_x)]_{i,j}$. 
Thus, 
\[ \left(f \circ \Phi_0^{(n)}\right) (u) =
\frac{1}{|E|} \sum_{e\in E} \left(f\circ \phi_{e}^{(n)}\right)
\left([u_{e}(i,j)]_{i,j}\right) = \sca{\Lambda, u } \geq 0. \]
The proof is complete.
\end{proof}

Now we demostrate a hypergraph for which the cones $C_1$ introduced in Proposition \ref{p_basicprmo} are not proper, equivalently that the subspace $\Jj$ defined in (\ref{thekernel}) is not a kernel.

\begin{example}
    Let $\G=(V,E)$ be the hypergraph in figure (\ref{fig:counterexample_hypergraph}) with 
    $ V= \{x_1,x_2,x_3,$ $x_4,x_5\}$ and \( 
E=\{e_1,e_2,e_3,e_4\} \) for $  
e_1=\{x_1,x_2\} $, $  e_2=\{x_2,x_3\} $, $ e_3=\{x_1,x_2,x_3\} $, $ e_4=\{x_2,x_4,x_5\} $. We note that this scenario has a unique probabilistic model, namely $ p \in \Gg(\G)$ with $ p(x_2)=1$ and $ p(x_i)=0$ for all $ i\neq 2$. Pick the vector $ u = \oplus_{i=1}^{4}(u_{x_{j}}^{e_i})_{x_j\in V} \in \Ss $  that has $ u_{x_4}^{e_4}=1$ and zero elsewhere. Then, $ \langle p, u\rangle =0$ for all models $ p$ (there is only one), and thus for all elements of $ \cl L_{1}$. Thus $ u \in C_1 \cap - C_1$. Note however that $ u \notin \Jj$ because if not, then $ u  $ would be a linear combination of the elements $ \alpha_{j}$ and $ \beta_{i,j}^{x}$ as in (\ref{thekernel}) and only $e_4$ contains $ x_4, x_5$. This means that $ u_{x_4}^{e_4}= u_{x_5}^{e_4}$ in this decomposition yielding a contradiction. 

\end{example}

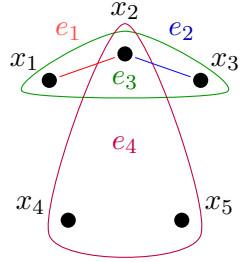
\begin{figure} \label{}
\centering
\begin{tikzpicture}

  % --- vertices (top-center shared)
  \node (vL)  at (1,   3.00) {};  % x1 (left-top)
  \node (vT)  at (2,   3.35) {};  % x2 (top-center)
  \node (vR)  at (3,   3.00) {};  % x3 (right-top)
  \node (vB1) at (1.25,1.15) {};  % x4 (bottom-left, raised)
  \node (vB2) at (2.75,1.15) {};  % x5 (bottom-right, raised)

  % --- 3-vertex edge e3 = {vL, vT, vR} (GREEN), triangular & symmetric
  \draw[green!60!black]
    plot[smooth cycle]
      coordinates {
        (0.65,2.85)   % near vL
        (3.35,2.85)   % near vR (mirror)
        (2.00,3.65)   % above vT (encloses vT)
      };

  % --- 3-vertex edge e4 = {vT, vB1, vB2} (PURPLE)
  \draw[purple]
    plot[smooth cycle]
      coordinates {
        (1.00,0.95)   % left base (below vB1)
        (3.00,0.95)   % right base (below vB2)
        (2.00,3.75)   % apex above vT
      };

  % --- 2-vertex edges as thin straight lines
  \draw[red,  line width=0.35pt] (vT) -- (vL); % e1
  \draw[blue, line width=0.35pt] (vT) -- (vR); % e2

  % --- vertices (dots)
  \foreach \p in {vL,vT,vR,vB1,vB2} \fill (\p) circle (0.1);

  % --- vertex labels
  \node[above left]   at (vL)  {$x_{1}$};
  % x2 label placed ABOVE the purple edge apex
  \node at (2,3.92) {$x_{2}$};
  \node[above right]  at (vR)  {$x_{3}$};
  \node at  (0.75,1.35) {$x_{4}$};
  \node   at (3.25,1.35) {$x_{5}$};

  % --- edge labels
  % e1/e2 near the thin lines; also put e2 outside the green loop
  \node[red!70!]  at (1.25,3.65) {$e_{1}$};
  \node[blue!70!black] at (2.75,3.65) {$e_{2}$}; % outside green

  % e3 centered inside triangle (centroid of vL,vT,vR ≈ (2,3.12))
  \node[green!50!black] at (2,3.00) {$e_{3}$};

  % e4 inside its purple loop
  \node[purple] at (2.00,2.15) {$e_{4}$};

\end{tikzpicture}
\caption{A contextuality scenario for which the cones in Proposition \ref{p_basicprmo} are not proper.}
\label{fig:counterexample_hypergraph}
\end{figure}

\begin{remark} \label{orderembeddingremark}
As a composition of completely positive maps, the maps $i_{e}$ defined in (\ref{eq_|E|}) are 
completely positive. Note also that $i_{e}$ is unital. 
The maps $ i_{e}$ may however fail to be  (complete) order embeddings. 
    Consider the contextuality scenario $\G_{1}=(V,E)$ of 
    Figure \ref{fig:NoPRs}, that is, 
    $ V = \{x_{1},x_{2},x_{3}\}$ and $E= \{ e_{1},e_{2},e_{3}\}$ with $e_1 =\{x_1,x_2\}$, $e_2 = \{x_2,x_3 \}$ and $ e_3 = \{ x_1,x_3\}$. 
It can be easily shown that the space 
$ \Jj$ defined in (\ref{thekernel}) is a null subspace, and so 
$ \Ss_{\G_{1}} 
= \ell^{\infty}_{e_{1}}\oplus\ell^{\infty}_{e_{2}}\oplus\ell^{\infty}_{e_{3}}/\Jj$. 
Let $ u = \begin{pmatrix}
            2\\
            -1
        \end{pmatrix}$, 
        considered as an element of $\ell^{\infty}_{e_{1}}$;
then
    \begin{align*}
        i_{e_{1}}(u) &=  3 \begin{pmatrix}
            2\\
            -1
        \end{pmatrix} \oplus
        \begin{pmatrix}
            0\\
            0
        \end{pmatrix} \oplus 
        \begin{pmatrix}
            0\\
            0
        \end{pmatrix} + \Jj.
    \end{align*}
    If $p \in \Gg(\G_{1})$ is the (only) probabilistic model of $ \G_{1}$, that is, 
    $p(x_{i})= \frac{1}{2}$, $i =1,2,3$, then 
    \[ \inn{p,i_{e_{1}}(u)}= 3\cdot  \frac{1}{2} \cdot 2 + 3 \cdot \frac{1}{2} \cdot (-1)=  \frac{3}{2} \geq 0,\]
which implies (see corollary \ref{p_dualSG}) that $ f (i_{e_{1}}(u)) \geq 0$ for every state $ f $ of $\Ss_{\G_{1}}$. Hence $i_{e_{1}}( u) \in (\Ss/\Jj)^{+}$. Nevertheless, $ u $ is not a positive element in $\ell^{\infty}_{e_{1}}$. 
Thus, the spaces $\ell_{e}^{\infty}$ cannot in general be 
canonically identified with operator subsystems of $ \Ss_{\G}$. 
\end{remark}

In the rest of this section, we will require some 
standard notions from hypergraph theory. If $r\in \bb{N}$, we
say that a hypergraph $\G =(V,E)$ is
\textit{$r$-uniform} 
is every edge contains precisely $r$ vertices and \textit{uniform}, if it is $r$-uniform for some $ r\in \N$. (We note that the Kochen-Specker scenario  (see \cite[Fig. 2] {Acn2015ACA}) is $4$-uniform.)
 Let $x\in V$ and $\lambda \in \N$. 
 The hypergraph that is obtained from $\G$ by \textit{multiplying $x$ by $\lambda$} \cite{berge1984hypergraphs}
arises by replacing the vertex $x \in V$ with the vertices $ (x,1),\cdots,(x,\lambda) $ and every hyperedge $e \in E$ that contains $x$ with the hyperedge $ (e \setminus \{x\}) \cup \{ (x,1),\cdots,(x,\lambda)\}$.
Let $h : V \rightarrow \N$ be a function. The \textit{$h$-multiplication} 
$ \G^{(h)} = (V^{(h)},E^{(h)})$ of $\G$ is the hypergraph obtained by multiplying each vertex $x \in V$ by $h(x)$.
A hypergraph $\G = (V,E)$ is called \textit{uniformisable} \cite{Beckenbach2019}, 
if there exists a function $ h : V \rightarrow \N$ such that $ \G^{(h)}$ uniform.
Clearly, any uniform hypergraph $\G$ is uniformisable with $h(x)=1$, $x\in V$.

\begin{corollary}\label{p_dualSG}
Let $\G$ be a contextuality scenario. 
The following are equivalent for a function $p : V\to [0,1]$:
    \begin{enumerate}
        \item $ p \in \Gg(\G)$;
        \item there is a state $ s : \Ss_{\G} \rightarrow \C$ such that 
        $s(a_{x}) = p(x)$,  $x\in V$.
    \end{enumerate}
Moreover, the following are equivalent:
\begin{enumerate}
        \item[(a)] the hypergraph $\G$ is uniform;
        \item[(b)] $\mathcal{L}_{\G}$ is an operator subsystem of $\cl S$. 
\end{enumerate}
If $\G$ is uniform then $\Ss_{\G}^{\rm d} = \mathcal{L}_{\G}$ up to a 
canonical complete order isomorphism. 
\end{corollary}

\begin{proof}
The equivalence (i)$\Leftrightarrow$(ii) follows from 
Theorem \ref{universalprop} and the fact that 
the elements of $\Gg(\G)$ are POR's of $\G$ on $\bb{C}$. 

Since $\mathcal{L}_{\G}$ is a selfadjoint subspace of $\cl S$, it is an 
operator subsystem of $\cl S$ if and only if it contains the unit of $\cl S$, 
which is satisfied if and only if $\G$ is uniform, that is, 
we have the equivalence (a)$\Leftrightarrow$(b).

Suppose that $\G$ is uniform. We show that $\cl J = \tilde{\cl J}$. 
As the inclusion $\cl J \subseteq \tilde{\cl J}$ is trivial, is suffices to show that 
$\tilde{\cl J}\subseteq \cl J$. To this end, let $u\in \tilde{\cl J}$; 
then $p(u) = 0$ for every $p\in \cl L_{\G}^+$. Since $\cl L_{\G}$ is an 
operator subsystem of $\cl S$, we have that $\cl L_{\G} = {\rm span}(\cl L_{\G}^+)$. 
Thus, $f(u) = 0$ for every $f\in \cl L_{\G}$, that is, 
$u\in \cl J$. 
Now, the quotient map $ q : \Ss \rightarrow \Ss /\cl J$ dualises to a 
complete order embedding 
$ q^{\rm d} : \cl S_{\G}^{\rm d} \rightarrow \Ss^{\rm d}$
\cite[Proposition 1.15]{Farenick_Paulsen_2012}. 
As $\Ss^{\rm d} = \cl S$ up to a canonical complete order isomorphism,
the map $q^{\rm d}$ gives rise to a complete order embedding (denoted in the same way)
$ q^{\rm d} : \cl S_{\G}^{\rm d} \rightarrow \Ss$. 
It remains to observe that the image of $q^{\rm d}$ in $\cl S$ coincides with 
$\cl L_{\G}$. 
\end{proof}

It follows from (the proof of) Corollary \ref{p_dualSG} that
if $\G$ is uniform then the space $\cl J$ is a kernel in $\cl S$. 
In the sequel, we will strengthen this result by both weakening the 
assumption on the uniformity of $\G$ and strengthening the conclusion 
to $\cl J$ being a null space in $\cl S$ (see Remark \ref{r_null}). 
We establish some notation first.
We fix a hypergraph $ \G = (V,E)$, and write
$V = \{x_{1}, \dots,x_{m} \}$ and, as before, 
$ E = \{e_{1}, \dots,e_{d} \}$ (note that $|e_{i}|\leq m$, $i\in [d]$). For $ e\in E$, let $\Cc_{e} = \{ (\chi_{e}(x_{i})\mu_{i} )_{i=1}^{m} : \mu_{i} \in \C\}$  be the  unital $ C^*$-algebra, considered as a subspace of $\C^{m}$, with unit $1= (\chi_{e}(x_{i}))_{i=1}^{m}$, and note that $\Cc_{e}$ isomorphic to
$ \ell^{\infty}_{e} $. Here, $\chi_{e}$ 
is the characteristic function 
of the subset $e\subseteq [m]$. Thus $\Cc_{e}$ consists of vectors of length $m$, 
with non-zero entries only in the $i$-th coordinates for 
which $ x_{i} \in e$. 
After identifying $ \Cc_{e}\cong \ell^{\infty}_{e}$, we consider the elements of $ \oplus_{e\in E} \ell^{\infty}_{e}$ as families 
of $d$ vectors in $\bb{C}^m$; 
an element 
$ u \in \oplus_{e\in E} \ell_{e}^{\infty}$ thus has the form 
$u = \left((u^{1}_{x_{i}})_{i=1}^{m},\cdots,(u^{d}_{x_{i}})_{i=1}^{m}\right)$,
or
\begin{align*}
  u=  \begin{bmatrix}
    (u^{1}_{x_{1}}, u^{2}_{x_{1}}, \cdots, u^{d}_{x_{1}}) \\
    (u^{1}_{x_{2}}, u^{2}_{x_{2}}, \cdots, u^{d}_{x_{2}})\\
    \vdots\\
      (u^{1}_{x_{m}}, u^{2}_{x_{m}}, \cdots,  u^{d}_{x_{m}})
    \end{bmatrix}.
\end{align*}
Setting $ u_{x_{i}} = (u_{x_{i}}^{1}, \cdots,   u_{x_{i}}^{d})$, $ i=1,\dots, m$, 
we thus have $u= (u_{x_{i}})_{i=1}^{m}$;
in other words, we have performed the canonical shuffle 
\[   \left((u_{x}^{e})_{x\in e}\right)_{e\in E} \ \longrightarrow  
\ \left((u_{x}^{e})_{e\in E}\right)_{x\in V}. \]
%by ``making up" some space in the vectors of $ \ell^{\infty}_{e}$.

\begin{proposition} \label{p_nullsubspace}
If $ \G = (V,E)$ is a uniformisable hypergraph
then $\Jj$ is a null subspace.
\end{proposition}

\begin{proof}
We assume first that the hypergraph $\G$ is $r$-uniform.
Fix an element $u =  \left((u_{x}^{e})_{x\in e}\right)_{e\in E}$ of 
$\Jj\subseteq \Ss$, where $ (u_{x}^{e})_{x\in e} \in \ell^{\infty}_{e}$.
Recalling (\ref{thekernel}), write 
    \begin{equation} \label{beforeshuffle}
        u = \sum_{i=2}^{d}\xi_{i-1} \alpha_i + 
        \sum_{k=1}^m \sum_{i<j} \lambda^{i}_{j-i}(k) \beta_{i,j}^{x_k};
    \end{equation}    
more concretely, 
\begin{multline}
    u = \xi_{1}(1 \oplus -1\oplus\cdots\oplus 0) + \cdots +\xi_{d-1}(1 \oplus 0\oplus\cdots\oplus -1) \\
    + \sum_{k=1}^{m}\Big( \lambda_{1}^{1}(k)(\delta_{x_{k}}^{e_{1}}\oplus -\delta_{x_{k}}^{e_{2}}\oplus \cdots \oplus 0)+ \cdots + \lambda_{d-1}^{1}(k)(\delta_{x_{k}}^{e_{1}}\oplus 0\oplus \cdots \oplus-\delta_{x_{k}}^{e_{d}} )\\
   + \lambda_{1}^{2}(k)(0\oplus \delta_{x_{k}}^{e_{2}}\oplus-\delta_{x_{k}}^{e_{3}}\oplus \cdots \oplus 0) + \cdots + \lambda_{d-2}^{2}(k)(0\oplus \delta_{x_{k}}^{e_{2}}\oplus \cdots \oplus-\delta_{x_{k}}^{e_{d}})\\
   \vdots \\
   + \lambda_{1}^{d-1}(k)(0\oplus \cdots \oplus \delta_{x_{k}}^{e_{d-1}}\oplus -\delta_{x_{k}}^{e_{d}}) \Big)
\end{multline}
where, if $ x_{k} \not\in e_i\cap e_j$, 
we have set $\lambda_{l}^{i}(k) = 0$ if $i + l > d$.
It will sometimes be convenient to write $\lambda_{1}^{d-1}(x_k) = \lambda_{1}^{d-1}(k)$, 
$k\in [m]$. 

Write 
 \[\oplus_{e\in E} (v_{x}^{e})_{x\in e} = \sum_{i=2}^{d}\xi_{i-1} \alpha_i \]
 and 
 \[\oplus_{e\in E} (w_{x}^{e})_{x\in e} = \sum_{k=1}^m\sum_{i < j}\lambda^{i}_{j-i}(k) 
\beta_{i,j}^x,
 \]
so that
 \begin{equation} \label{decompose}
     u = \oplus_{e\in E} (v_{x}^{e})_{x\in e}  + \oplus_{e\in E} (w_{x}^{e})_{x\in e}.
 \end{equation}
Since $\G$ is assumed $r$-uniform, 
$\ell^{\infty}_{e} = \ell^{\infty}([r])$, $e\in E$, and 
\[ \sum_{e\in E}\sum_{x\in e} v_{x}^{e} =0 \; \; \text{and} \; \; \sum_{e\in E} \sum_{x \in e} w_{x}^{e}= 0   \]
(we note that, without this assumption, the first sum need not be zero).
Applying the canonical shuffle, discussed before the formulation of 
the proposition, we write $u= (u_{x_{i}})_{i=1}^{m}$, where 
$ u_{x_{i}} = (u_{x_{i}}^{1}, \cdots,  u_{x_{i}}^{d})$, $ i=1,\dots, m$. 

Now assume that $ u \in \Jj\cap \Ss^{+}$; this means that $ u_{x_{i}}^{j}  \in \R^{+}$ for all 
$i\in [m]$ and all $j\in [d]$.  
We will show that $ u= 0$, equivalently, that $ u_{x} = 0$ for all $ x\in V$. 
Fix $ x\in V$ and enumerate the vertices and the edges in such a way that $x$ belongs in the $k$ first hyperedges, that is, 
$ x\in e_{j}$, $j=1,\dots,k$ where $ 1\leq k \leq d$. 

\smallskip

 In the case that $1<k\leq d$, 

\begin{align} \label{u_xexpression}
  u_{x} =   \left(\sum_{i=1}^{d-1}\xi_{i} + \sum_{i=1}^{k-1}\lambda_{i}^{1},
  -\lambda_{1}^{1} -\xi_{1},\cdots,-\lambda_{k-1}^{1}- \xi_{k-1},
  0,\cdots,0\right),
\end{align}
where only the $k$ first entries are possibly non-zero.
Recalling that $u_{x}^{j} \geq 0$, we note that every entry on 
the right hand side of (\ref{u_xexpression}) is a non-negative real number.
Summing the entries gives 
\begin{align} \label{firstsum}
    \sum_{i=k}^{d-1}\xi_{i} \geq 0.
\end{align}
We claim that inequality (\ref{firstsum}) holds with an equality. It will then follow 
from the non-negativity of the entries on the 
right hand side of (\ref{u_xexpression}) 
that $u_{x}^{j}=0$ for every $ i=1,\dots,d$. As $ \sum\limits_{x\in V} \sum\limits_{j=1}^{d} u_{x}^{j} =0$, we have  
 \begin{align*}
   \sum_{x' \neq x} \sum_{j=1}^{d}u_{x'}^{j} = - \sum_{j=1}^{d}u_{x}^{j} = - \sum_{i=k}^{d-1}\xi_{i},
 \end{align*}
and since $u_{x}^{j} \geq 0 $ for all $ x$ and all $ j$, we conclude that 
$\sum\limits_{i=k}^{d-1}\xi_{i} = 0$
 by (\ref{firstsum}). The case  $ k=1$, follows by the same arguments except that in (\ref{u_xexpression}) all $ \lambda_{i}^{1} =0$.

Now relax the assumption that the hypergraph $\G$ is uniform, 
and assume instead that it is uniformisable.
Thus, there exists a function $h : V\to \bb{N}$ such that the 
$h$-multiplication $\G^{(h)}$ of $\G$ is uniform. 
Note that any of the multiplied vertices 
$(x,1),\cdots,(x,h(x))$ will participate in precisely the same hyperedges as $ x $. 
Let 
$\gamma : \oplus_{e\in E}\ell^{\infty}_{e} \rightarrow \oplus_{f\in E^{(h)}}\ell^{\infty}_{f} $ 
be the map which takes an element 
$u = (u_{x})_{x\in V}$ to the element 
$v = (v_{y})_{y\in V^{(h)}}$ with $ v_{y} = u_{x}$, for all $y = (x,1),\dots, (x,h(x))$. 
It is clear that $ \gamma$ is a (linear) order embedding.
Further, denoting temporarily by 
$\beta_{i,j}^{x}(\bb{H})$ and $\alpha_{j}(\bb{H})$ the canonical generators of 
$\cl J_{\bb{H}}$, we have that 
$\gamma\left(\alpha_{j}(\G)\right) = \alpha_{j}(\G^{(h)})$ and 
$$\gamma\left(\beta_{i,j}^{x}(\G)\right) = \sum_{l=1}^{h(x)} \beta_{i,j}^{x}(\G^{(h)}), 
\ \ \ x\in V, \ i,j\in [d], i < j.$$
It follows that $ \gamma(\Jj_{\G}) \subseteq \Jj_{\G^{(h)}}$. 
So if $ u \in \Jj_{\G}^{+}$, then $ \gamma(u) \in \Jj_{\G^{(h)}}^{+} $ 
and since $ \G^{(h)}$ is uniform, $ \Jj_{\G^{(h)}}$ 
is a null subspace by the first part of the proof, 
forcing $ \gamma(u)=0$ and thus also $ u=0$. Therefore, $ \Jj_{\G}$ is a null subspace.
 \end{proof}

We give examples of non-uniformisable hypergraphs such that the corresponding subspace $\Jj$ is not a null subspace.

\begin{example}
    Let $V=\{x,y\}$ and $E=\{e_1,e_2\}$ where $e_1=\{x\}, e_2=\{x,y\}$.
Then $(V,E)$ is not uniformisable. The space $\mathcal J$ is a  kernel, but not a null subspace.
Indeed, we have $$\cl S=\ell^\infty[1] \oplus \ell^\infty[2]= \bb C \oplus (\bb C \oplus   \bb C)$$
and so $\cl J$ is the linear span of $(1, (0,0))-(0,(1,1))$
and $\delta_x^{e_{1}}-\delta_x^{e_{2}}=(1, (0,0))-(0,(1,0))$. Thus
\begin{align*}
\cl  J  & =\mathop{\rm span} \{ (1,(-1,0)), ((1,(-1,-1))\} 
=\mathop{\rm span} \{ (1,(-1,0)), ((0,(0,1))\} \\
&= \{a(1,(-1,0))+b(0,(0,1))\mid a,b\in\bb C\}
= \{(a,(-a,b)) \mid a,b\in\bb C\}.
\end{align*} 
It follows that 
$\cl J =  \ker\phi$,
where 
$$
\phi((x,(y,z))) := \frac 12 (x+y), \quad (x,(y,z)) \in \cl S.
$$
Since the map $\phi$ is clearly (completely) positive and unital, 
$\cl J$ is a kernel. However, $\cl J$ is not a null subspace since  
$((0,(0,1))\in \cl J\cap \cl S^+$. 
\end{example}

\begin{example} \label{e_non-unif}
Let $\G =(V,E)$ with $ V= \{x_{1},\dots,x_{5}\}$ and $E= \{e_1,\dots,e_4\}$ where 
    $e_1= \{x_1,x_2\}$, $e_2= \{x_3,x_4\}$, $ e_3= \{x_2,x_3,x_5\}$ and  $e_4= \{x_1, x_4, x_5\}\}$ (see Figure \ref{fig:noPRs2}); 
    then $\Jj$ contains a non-zero positive element. Moreover, $\G$ is not uniformisable.
Indeed, we have that 
$ \Ss= \oplus_{j=1}^{4}\ell^{\infty}_{e_{j}}$.
The element $ u= (u_{x_{i}})_{i=1}^{5}$ with $ u_{x_{5} } = (0,0,1,1)$ and zero elsewhere belongs in $ \Jj$ and is clearly positive.
\end{example}

It is natural to ask whether the operator systems $\Ss_{\G}$ 
remember the underlying hypergraphs $\G$; in other words,  
if $\G, \Hbb$ are hypergraphs such that $ \Ss_{\G} = \Ss_{\Hbb}$ 
up to a unital complete order isomorphism, does it follow that $\G$ is isomorphic to $\Hbb$? 
We recall here that the hypergraphs $\G = (V,E)$ and $ \Hbb = (W,F)$ are 
called isomorphic if there is a bijection $g : V \rightarrow W $ such that for every set $ X \subseteq W$, $ g^{-1}(X) \in E$ if and only if $ X\in F$.
We point out that there are infinitely many hypergraphs 
for which this is not the case.

\begin{remark}
    For each $n \geq 3$, consider the hypergraphs $\Hbb_{n} = (W_{n},F_{n} ) $ with $ W_{n}= \{x_1,\dots,x_n\}$ and $F_n =\{\{x_1,x_2\},\{x_2,x_3\},\dots,\{x_{n-1},x_n\}, \{x_n,x_1\}\} $ (case $n=3$ is given in Figure \ref{fig:NoPRs}). Then $ \Gg(\Hbb_n) = \{p_n\}$, where $ p_{n}(x_{k}) = \frac{1}{2}$ for all $ k \in [n]$. This implies (Corollary \ref{p_dualSG}) that $ \Ss_{\Hbb_n}^{\rm d} = \C$ and hence $\cl S_{\Hbb_{n}}$ are all one-dimensional. 
    However the hypergraphs $ \Hbb_n, \Hbb_m$ are clearly not isomorphic for $ n\neq m$.
\end{remark}

%%%%%%%%%%%%%%%%%%%%%%%%%%%%%%%%%%%%%%%%%%%%%%%%%%%%%%%%%%%%%%%%%%%%%%%%%

\subsection{The operator system of quantum models}\label{ss_qmodel}

In this subsection, we introduce an operator system, universal for 
quantum models of a contextuality scenario. 
Recall \cite{Acn2015ACA} that, given a non-trivial contextuality scenario 
$\G$, the free hypergraph 
$C^{*}$-algebra $C^{*}(\G)$ is the universal $C^{*}$-algebra generated by orthogonal projections $p_{x}$, $x\in V$, satisfying the relations
$ \sum\limits_{x\in e}p_{x}=1$, $e\in E$; thus, the *-representations 
$\pi : C^{*}(\G)\to \cl B(H)$ of $C^{*}(\G)$
correspond precisely to PR's $(P_x)_{x\in V}$ of $\G$ on the Hilbert space $H$ via the assignment
$\pi(p_x) = P_x$, $x\in V$.

\begin{remark}\label{r_freep}
Let $ \mathcal{F}_{\G}$ be the algebraic free product 
$\ast_{e\in E}\ell^{\infty}_{e}$, amalgamated over the units 
$1^{e}$ of $ \ell^{\infty}_{e}$, $e\in E$, 
and 
\[ \cl I
= \langle \delta_{x}^{e}-\delta_{x}^{e'}: e ,e' \in E,  \; x \in e\cap e'\rangle, \]
as a  two-sided ideal, where
$\delta_{x}^{e} $ are the canonical generators of $ \ell^{\infty}_{e}$. 
Denote by $*_{\G}\ell^{\infty}_{e}$ the 
$C^{*}$-completion of the quotient $\mathcal{F}_{\G}/ \cl I$.
It follows that the *-representations of 
$*_{\G}\ell^{\infty}_{e} $ correspond precisely to PR's of the hypergraph $ \G$; 
thus, up to a canonical *-isomorphism, $*_{\G}\ell^{\infty}_{e} = C^{*}(\G)$.
%The quotient $\mathcal{F}_{\G}/ \cl I$ is thus a natural 
%dense subset of $C^{*}(\G)$, the explicit description of such a set was 
%raised as a question in \cite{Fritz2020CuriousPO}.
\end{remark}
%\marginpar{\tiny We further take another quotient to define the norm so maybe its not true.}

Let
$$\Tt_{\G} := \spann\{ p_{x}: x\in V \},$$
viewed as an operator subsystem of $C^{*}(\G)$,
and note that, up to a canonical complete order isomorphism, 
$$    \Tt_{\G} = \spann\{\delta_{x}^{e}: x\in e, \; e\in E \}$$
as an operator subsystem of $*_{\G}\ell^{\infty}_{e}$.
We do not know if the canonical *-homomorphism  
$\mathcal{F}_{\G}/ \cl I \rightarrow *_{\G}\ell^{\infty}_{e} $ is necessarily injective.

\begin{remark}\label{r_boca}
One could ask if a hypergraph variant of Boca's theorem \cite{boca}
holds true, that is, whether, given a hypergraph $\G = (V,E)$, 
a Hilbert space $ \Hh$ and unital completely positive 
maps $ \phi_{e} : \ell^{\infty}_{e} \rightarrow \Bh$, $e\in E$, 
such that $ \phi_{e}(\delta_{x}^{e}) = \phi_{e'}(\delta_{x}^{e'}) $, $ x\in e\cap e'$, 
there exists a unital completely positive 
map $ \Phi : *_{\G}\ell^{\infty}_{e} \rightarrow \Bh$ such that 
$ \Phi\circ \pi_{e}= \phi_{e} $, $e\in E$. Here $ \pi_{e}: \ell^{\infty}_{e}\rightarrow *_{\G}\ell^{\infty}_{e}$, $e\in E$ denote the canonical $*$-homomorphisms.
This is not the case:
the validity of the statement would imply that
$ \Tt_{\G} = \Ss_{\G}$, which in turn would imply that all positive operator representations always dilate into projective representations,
which does not hold in general 
(an explicit example is given in Figure \ref{fig:NoPRs},
see also Corollary \ref{c_diffo}).
\end{remark}

\begin{theorem} \label{statesonT0}
Let $ \G = (V,E)$ be a contextuality scenario. 
If $ \Phi: \Tt_{\G} \rightarrow \Bh$ is a unital completely positive map
then $(\Phi(p_{x}))_{x\in V}$ is a dilatable POR of $\G$.
Conversely, if $(A_{x})_{x\in V} \subseteq \Bh$ is a dilatable 
POR of $ \G$ then there exist a 
unique unital completely positive map $ \Phi: \Tt_{\G} \rightarrow \Bh$ such that 
$\Phi(p_{x})=A_{x}$, $ x\in V$. 
Up to a canonical complete order isomorphism, $\Tt_{\G}$ is the unique 
operator system satisfying the latter universal property. 
\end{theorem}

\begin{proof}
Let $(A_x)_{x\in V}\subseteq \cl B(H)$ be a dilatable POR of $\G$, and let 
$ (P_{x})_{x\in V} \subseteq \cl B(K)$ be a PR of $\G$ and $V : H\to K$ be an isometry,
such that $A_x = V^*P_xV$, $x\in V$. Let $\pi : C^*(\G)\to \cl B(K)$ be the 
*-representation, such that $\pi(p_x) = P_x$, $x\in V$. 
The map $\Phi : \cl T_G\to \cl B(H)$, given by $\Phi(a) = V^*\pi(a)V$, $a\in \cl T_{\G}$, 
is unital, completely positive, and $\Phi(p_x) = A_x$, $x\in V$. 

Conversely, suppose that $\Phi : \cl T_{\G}\to \cl B(H)$ is a unital completely positive map. 
Applying Arveson's Extension Theorem, let $\tilde{\Phi} : C^*(\G)\to \cl B(H)$ be a 
unital completely positive extension. Letting $\tilde{\Phi} = V^*\pi(\cdot)V$ be a 
Stinespring dilation of $\tilde{\Phi}$, we have that $(\pi(p_x))_{x\in V}$ is a 
PR that is a dilation of $(\Phi(p_x))_{x\in V}$.

The uniqueness of the operator system with the stated universal property is 
standard. 
\end{proof}

For a linear functional $s : \Tt_{\G} \rightarrow \C$, write 
$p^s : V\to \bb{C}$ for the map, given by $p^s(x) = s(p_x)$, $x\in V$.

\begin{corollary} \label{statesonT}
    Let $\G= (V,E)$ be a contextuality scenario and $ p \in \Gg(\G)$ be a probabilistic model. 
    The following are equivalent:
\begin{enumerate}
        \item $ p \in \Qq(\G)$; 
        \item there is a state $s : \Tt_{\G} \rightarrow \C$ such that
        $p = p^s$.
\end{enumerate}
      Moreover, the correspondence $ s \mapsto p^s$ is an affine homeomorphism between the state space of $\Tt_{\G}$, equipped with the weak* topology, 
      and the set of quantum models $ \Qq(\G)$.
      In particular, the set $\Qq(\G)$ is closed.
\end{corollary}

\begin{proof}
The equivalence (i)$\Leftrightarrow$(ii) follows 
from Theorem \ref{statesonT0} and the fact that 
the elements of $\Qq(\G)$ are the dilatable POR's of $\G$ on $\bb{C}$.
The remaining statements are now straightforward.
\end{proof}

Let $ \Ss $ be an operator system and $ (\Aa, \iota)$ be a $C^{*}$-cover of $\Ss$. We say that  \cite[Definition 9.2]{Kavruk2010QuotientsEA} $\Ss$ \textit{contains enough unitaries} in $\Aa$, if the 
elements in $\Ss$ that are unitaries in $\Aa$, generate $\Aa$ as a $C^{*}$-algebra. We recall that \cite[Proposition 5.6]{Kavruk2014} if an operator system $ \Ss $ has enough unitaries in its $C^{*}$-cover $ \Aa$, then $ \Aa = C^{*}_{\rm e}(\Ss)$.

\begin{lemma} \label{l_idCstareTG}
    Let $\G$ be a contextuality scenario. Then $ \Tt_{\G}$ contains enough unitaries in $C^{*}(\G)$ and thus $C^{*}(\G) = C^{*}_{\rm e}(\Tt_{\G})$.
\end{lemma}
\begin{proof}
    For every $x\in V$, the element $ 2p_{x}-1$ is a unitary in $ C^{*}(\G)$. 
    It follows that the operator system $\Tt_{\G}$ contains enough unitaries in $C^{*}(\G)$.
    The result now follows from \cite[Proposition 5.6]{Kavruk2014}.
\end{proof}

%%%%%%%%%%%%%%%%%%%%%%%%%%%%%%%%%%%%%%%%%%%%%%%%%%%%%%%%%%%%%%%%%%%%%%%%%

\subsection{The operator system of classical models}\label{ss_classical}

Let $\G = (V,E)$ be a contextuality scenario with $|E|=d $. Consider the unital abelian C*-algebra
$\cl D= \ell_{e_{1}}^{\infty} \otimes \cdots \otimes \ell_{e_{d}}^{\infty}$, 
and let $\tilde{\iota}_e : \ell_{e}^{\infty}\to \cl D$
be the natural unital embedding. 
Let 
$\cl I$ be the two-sided ideal generated by the elements
\[ \tilde{\iota}_{e_i}(\delta_x^{e_i}) - \tilde{\iota}_{e_j}(\delta_x^{e_j}), 
\ \ \ x \in e_{i}\cap e_{j}, i,j\in [d].  \]
 The quotient $\Dd_{\G} : = \cl D/\cl I$
is a unital abelian C*-algebra; we let  
 \[ d_{x} := 1 \otimes \cdots \otimes \delta_{x}^{e}\otimes \cdots \otimes 1  
 + \cl I, \; \; \; x\in V \]
(note that $d_x$ is well-defined). 
Note that
$(\Dd_{\G})^{\rm d} = (\cl D/\cl I)^{\rm d} = \cl I^{\perp}$
completely isometrically via the dual of the quotient map 
$q^{\rm d} : (\cl D/\cl I)^{\rm d} \rightarrow \cl D^{\rm d}$, and that 
a character on $\Dd_{\G}$ is a character on $\cl D$ that vanishes on $\cl I$; 
thus, the pure states of $\Dd_{G}$ 
can be identified with the pure states of $\cl D$ that vanish on $\cl I$.  

Since $\cl D \cong C(e_{1}\times \cdots \times e_{n} )$, 
its pure states are the evaluations at the points of $e_{1}\times \cdots \times e_{n}$. Thus, for a pure state $s$ on $\Dd_{\G}$, the assignment
$p(x)= s(d_{x})$, $x\in V$, is a deterministic model of $ \G$. 
Set
\begin{equation}\label{eq_RGdx}
    \Rr_{\G} := \spann\{d_{x} : x\in V\},
\end{equation}
viewed as an operator subsystem of $\Dd_{\G}$. 

\begin{proposition}\label{p_repcl}
    Let $\G=(V,E)$ be a contextuality scenario and $ p \in \Gg(\G)$ be a probabilistic model. 
    The following are equivalent:
    \begin{enumerate}
        \item $p \in \Cc(\G)$;
        \item there is a state $s : \Rr_{\G} \rightarrow \C$ such that 
        $p(x) = s(d_{x})$, $x\in V$.
    \end{enumerate}
    Moreover, the correspondence $ s \mapsto p^{s}$ is an affine homeomorphism between 
    $ S(\Rr_{\G}) $ and $ \Cc(\G)$, carrying the set of 
    pure states of $\Rr_{\G}$ onto the set of deterministic models of $\G$.
\end{proposition}

\begin{proof}
          $ (i) \Rightarrow (ii)$ 
          Let $ p \in \Cc(\G) $ be a classical model, and write
     \[ p(x) =\sum_{i=1}^{k}\lambda_{i}p_{i}(x), \ \ \ x\in V,   \] 
     where $k \in \bb{N}$, $\lambda_{i} \in [0,1]$, $ \sum_{i=1}^{k}\lambda_{i}=1$ and 
     $ p_{i}$  is a deterministic model, $i\in [k]$. Set 
     $s_{i}(d_{x}) := p_{i}(x)$,  $x\in V$, $i \in [k]$,
     and note that each $ s_{i}$ defines a state on $ \Rr_{\G}$ by extending linearly 
     to the whole space. As a consequence, 
     $s := \sum_{i=1}^{k}\lambda_{i}s_{i}$
     is a state on $\Rr_{\G}$ with 
     \[ s(d_{x})=\sum_{i=1}^{k}\lambda_{i}s_{i}(d_{x}) = \sum_{i=1}^{k}\lambda_{i}p_{i}(x) =p(x), \ \ \ x\in V. \]

     $ (ii) \Rightarrow (i)$ Let 
     $ s : \Rr_{\G} \rightarrow \C$ be a state.  
     Since $\Dd_{\G}$ is a unital C*-algebra, its state space is the closed convex hull of the pure states. Extend $s$ to a state $ \Tilde{s}$ on $\Dd_{\G}$. Then, since $\Dd_{\G}$ is finite dimensional abelian C*-algebra, 
     $\Tilde{s} =  \sum_{i=1}^{k} \lambda_{i}s_{i}$ as a convex combination, 
     for some pure states $s_{i}$, $i\in [k]$. 
     By the paragraph before the formulation of the proposition, 
     the assignment $p_{i}(x)= s_{i}(d_{x})$, $x\in V$, is a deterministic model
     of $\G$. As
     $p(x) = \sum_{i=1}^{k}\lambda_{i} p_{i}(x)$, $x\in V$,
     we have that $ p \in \Cc(\G)$.
\end{proof}

Let $(A_{x})_{x\in V} \subseteq \Bh$ be a POR. We say that $ (A_{x})_{x\in V}$ is \textit{classically dilatable} 
if there exists a Hilbert space $\Kk$, 
an isometry $ V :\Hh \rightarrow \Kk$, 
and a  PR $ (P_{x})_{x\in V} $ with commuting entries such that 
$A_{x} = V^{*}P_{x}V$, $x\in V$.

The next proposition complements Theorem \ref{statesonT0}; as the proof is 
similar, it is omitted. 

\begin{proposition}\label{p_clasdil}
    Let $\G= (V,E)$ be a contextuality scenario and $(A_{x})_{x\in V}$ be a POR on a Hilbert space $\Hh$. Then, the following are equivalent:
    \begin{enumerate}
        \item $(A_{x})_{x\in V}$ is classically dilatable;
        \item there exists a unital completely positive 
        map $ \phi : \Rr_{\G} \rightarrow \Bh$, such that $ \phi(d_{x}) =A_{x}$, $ x\in V$.
    \end{enumerate}  
\end{proposition}

We recall the minimal operator system structure
from \cite{PTT}, which will be used hereafter. 
Given an ordered vector *-space $V$ with an Archimedean order unit $e$, 
there exists a family $(C_n^{\min})_{n\in \bb{N}}$ 
%(resp. $(C_n^{\max})_{n\in \bb{N}}$)
of matricial cones over $V$ such that 
the triple $(V,(C_n^{\min})_{n\in \bb{N}},e)$ 
%(resp. $(V,(C_n^{\max})_{n\in \bb{N}},e)$) 
is an operator system, denoted ${\rm OMIN}(V)$ 
%(resp. ${\rm OMAX}(V)$), 
and that, 
for every operator system $\cl T$, every positive map 
$\phi : \cl T\to V$ 
%(resp. $\phi : V\to \cl T$) 
is automatically completely positive.
The operator system equalities in the next proposition are understood up to 
a canonical unital complete order isomorphism.

\begin{theorem}\label{p_OMINs}
    Let $\G=(V,E)$ be a contextuality scenario. Then
    \begin{enumerate} 
       \item  $ \Gg(\G)= \Qq(\G) $ if and only if  $\mathrm{OMIN}(\Ss_{\G}) = \mathrm{OMIN}(\Tt_{\G})$;
       \item $\Qq(\G)= \Cc(\G)$ if and only if $\mathrm{OMIN}(\Tt_{\G})=\Rr_{\G}  $;
        \item $\Cc(\G) = \Gg(\G)$ if and only if $\Rr_{\G} = \mathrm{OMIN}(\Ss_{\G})$
    \end{enumerate}
\end{theorem}
\begin{proof}
(i) Note first that by \cite[Theorem 3.2]{PTT}, 
  if $ V$ is an AOU space then 
  $\mathrm{OMIN}(V)$ can be identified as an operator subsystem of $ C(S(V))$ 
  (here, the state space $S(V)$ of $V$ is equipped with the weak* topology). 
  %So, we begin by identifying $ \mathrm{OMIN}(\Ss_{\G}) \subseteq C(S(\Ss_{\G})) $ and  $ \mathrm{OMIN}(\Tt_{\G}) \subseteq C(S(\Tt_{\G})) $. 

   By Corollaries \ref{p_dualSG} and \ref{statesonT}, 
     $\Qq(\G)= \Gg(\G)$ if and only if 
     $S(\Ss_{\G})$ is canonically affinely homeomorphic to $S(\Tt_{\G})$;
     if this happens, let $ h : S(\Tt_{\G}) \rightarrow S(\Ss_{\G}) $ be the
     corresponding homeomorphism.
     %In particular, if $ t $ is a state on $\Tt_{\G}$, then there exists a state $ s$ on $ \Ss_{\G}$ such that 
    %\[ t(p_{x}) = s (i_{e}(\delta_{x}^{e})), \; \; \; x\in V.\]
    It follows that the map $\pi : C(S(\Ss_{\G})) \to C(S(\Tt_{\G}))$, given by 
    $\pi(f) =  f \circ h$, is a *-isomorphism. 
    For $x\in V$, we have 
    \begin{align*}
        \inn{\pi( a_x) ,t} = h(t)(i_{e}(\delta_{x}^{e})) =t(p_{x}) = \inn{p_{x},t}, \ \ \ t \in S(\Tt_{\G}).
    \end{align*}
    Hence, $ \pi( a_x) = p_{x}$ and since $ \pi$ is a C*-algebra isomorphism, 
    in view of the first paragraph,  
    its restriction $ \pi\arrowvert_{\mathrm{OMIN}(\Ss_{\G})} = \Phi$ is a unital complete order isomorphism onto $ \mathrm{OMIN}(\Tt_{\G})$.

Conversely, assume that the linear map 
$ \Phi : \mathrm{OMIN}(\Ss_{\G}) \rightarrow \mathrm{OMIN}(\Tt_{\G})$,
given by $\Phi(a_x) = p_x$, 
is a unital complete order isomorphism.
Then $\Phi$ is a unital order isomorphism  between $\Ss_{\G}$ and $ \Tt_{\G}$. Let 
    $h : S( \Tt_{\G}) \rightarrow S(\Ss_{\G})$ be the map, given by 
    $h(t)  =  t \circ \Phi$; we have that $h$ is an affine bijection
    and a homeomorphism. 
Now let $ p \in \Gg(\G)$. By Corollary \ref{p_dualSG}, 
there exists a state $ s$ on $\Ss_{\G}$ such that  
    $p(x)= s(a_x)$, $x\in V$.
    The state $ t := h^{-1}(s)$ of $\Tt_{\G}$ has the property
    \[ s(a_x) = t( p_{x}), \; \; \; x\in V. \]
    By Proposition \ref{statesonT},  $ (t(p_{x}))_{x\in V}$ is a quantum model. 
    Hence, $ p \in \Qq(\G)$, and so $\Gg(\G)= \Qq(\G)$.

  (ii)  The operator system structure of an 
operator system $\cl R$ of the form $\cl R = {\rm OMIN}(\cl R)$ is 
uniquely determined by the state space $S(\cl R)$ \cite{PTT}. 
By Corollary \ref{p_dualSG}, the operator system $\mathrm{OMIN}(\Ss_{\G})$ is 
determined uniquely by $\Gg(\G)$. 
Since $\Rr_{\G}$ is an operator subsystem of an abelian C*-algebra, 
we have that $\Rr_{\G} = {\rm OMIN}(\Rr_{\G})$. The statement follows. 

(iii) The proof is similar to that of (ii), and is omitted. 
\end{proof}

%%%%%%%%%%%%%%%%%%%%%%%%%%%%%%%%%%%%%%%%%%%%%%%%%%%%%%%%%%%%%%%%%%%%%%%%%
%%%%%%%%%%%%%%%%%%%%%%%%%%%%%%%%%%%%%%%%%%%%%%%%%%%%%%%%%%%%%%%%%%%%%%%%%
%%%%%%%%%%%%%%%%%%%%%%%%%%%%%%%%%%%%%%%%%%%%%%%%%%%%%%%%%%%%%%%%%%%%%%%%%

\section{Dilating contextuality scenarios}\label{s_ds}

%Let $\G = (V,E)$ be a non-trivial contextuality scenario. 
%We say that a POR $ (A_{x})_{x\in V} \subseteq \Bh$ of $\G$
%dilates to a PR, if there exist a Hilbert space $ \Kk$, an isometry 
%$ V :\Hh \rightarrow \Kk$ and a PR $ (P_{x})_{x\in V}$ of $\G$
%such that $ A_{x}= V^{*}P_{x}V$, $ x\in V$.
In this section, we examine dilations of 
POR's of a contextuality scenario, and characterise 
the scenarios whose POR's always admit a dilation to a PR
in terms of equality of canonical operator systems.

 \begin{definition}\label{d_dilating}
 We  say that a contextuality scenario $\G$ is  dilating    
 (resp. classically dilating),     
 if every POR  of $\G$ dilates to a PR of $\G$ (resp. PR of $\G$ with commuting entries). 
\end{definition}

\begin{remark}\label{r_RTG}
Recall the operator system $\Rr_{\G}$ (resp. the C*-algebra $\Dd_{\G}$) 
defined in Subsection \ref{ss_classical},
and note that the elements $d_x$, $x\in V$, defined therein,  
are projections in $ \Dd_{\G}$ such that $\sum_{x\in e}d_{x} =1$, for each $ e\in E$.
By the universal property of the free hypergraph C*-algebra, there exists a unital *-homomorphism $\pi : C^{*}(\G) \rightarrow \Dd_{\G}$, such that $ \pi(p_{x}) = d_{x} $, $ x\in V$. The restriction of $\pi$ to $ \Tt_{\G}$ is a unital completely positive map 
$ \Psi : \Tt_{\G} \rightarrow \Rr_{\G}$ such that 
$\Psi(p_{x}) = d_{x}$, $x\in V$. 
Similarly, by the universal property of  $\Ss_{\G}$, 
there is a unital completely positive map  $ \Phi : \Ss_{\G} \rightarrow \Tt_{\G}$, 
such that $\Phi(a_x) = p_{x}$, $x\in V$. 
We have the following diagram:
\begin{align*}
    \Ss_{\G} \xrightarrow{\Phi} \Tt_{\G} \xrightarrow{\Psi} \Rr_{\G}.
\end{align*}
\end{remark}

The maps $\Phi$ and $\Psi$, defined in Remark \ref{r_RTG}, will be used
subsequently without further mention.

\begin{theorem}\label{th_dila}
    Let $ \G=(V,E)$ be a contextuality scenario. 
    The following are equivalent:
\begin{enumerate}
\item    $ \G$ is dilating; 
       \item the map $\Phi : \Ss_{\G} \rightarrow \Tt_{\G}$
    is a complete order isomorphism.        
   \end{enumerate} 
   \end{theorem}
   
    \begin{proof}
        $(i) \Rightarrow (ii) $ 
        Assume, without loss of generality, 
        that $\Ss_{\G}\subseteq \cl B(\cl H)$ as an operator subsystem, 
        for some Hilbert space $\cl H$. 
        Since $ \G$ is assumed to be dilating, 
        there exist a Hilbert space $ \Kk$, an isometry $ W :\Hh \rightarrow \Kk$ and a PR $ (P_{x})_{x\in V}$ of $\G$ on $\cl K$, 
        such that $ a_{x}= W^{*}P_{x}W$, $ x\in V$.
        Let $\pi :  C^{*}(\G) \rightarrow \Bk$ be the (unique) 
        *-homomorphism, such that $\pi(p_x) = P_x$, $x\in V$, whose existence is 
        ensured by the universal property of $C^{*}(\G)$.
        We notice that $\Phi^{-1}(p_x) = W^{*} \pi(p_{x}) W$, $x\in V$;
        thus, $\Phi^{-1}$ is completely positive.

        $(ii) \Rightarrow (i)$ 
        Let $\cl H$ be a Hilbert space and 
        $ (A_{x})_{x\in V} \subseteq \Bh$ be a POR of $\G$. 
        By the universal property of $\Ss_{G}$ 
(see Theorem \ref{universalprop}) and the assumption that 
        $\cl S_{\G} = \cl T_{\G}$, there exists a 
        unital completely positive map $ \phi : \Tt_{\G} \rightarrow \Bh$ such that $ \phi(p_{x}) = A_{x}$, for all $x\in V$. 
By Theorem \ref{statesonT0}, $(A_x)_{x\in V}$ admits a dilation to a PR.         
    \end{proof}

\begin{corollary}
    If $\G$ is a dilating contextuality scenario then $ \Gg(\G) = \Qq(\G)$.
\end{corollary}

\begin{proof}
The statement follows from Theorem \ref{th_dila} and Theorem \ref{p_OMINs}. 
\end{proof}

The following corollary provides a characterisation of dilatability
in terms of the generators in $\Ss_{\G}$ and will be useful in the sequel.

\begin{corollary}\label{corollaryC*env}
Let $ \G$ be a contextuality scenario. 
The following are equivalent:
\begin{enumerate}
\item 
$\G$ is dilating; 

\item 
there exists a C*-cover $(\Aa,\iota)$ of
$\Ss_{\G}$ such that
$\iota( a_x)$ is a projection, $x\in V$;

\item 
$\iota_{\rm e}( a_x)$ is a projection in 
$ C^{*}_{\rm e}(\Ss_{\G})$, $x\in V$.
\end{enumerate}
\end{corollary}

    \begin{proof}
$(i) \Rightarrow (ii)$ 
By Theorem \ref{th_dila}, the map $\Phi : \Ss_{\G} \to \Tt_{\G}$ is a 
completely order isomorphism; 
thus, (ii) is fulfilled with $\cl A = C^{*}(\G)$ and $\iota = \iota_0\circ\Phi$,
where $\iota_0 : \Tt_{\G}\to C^{*}(\G)$ is the inclusion. 

$(ii) \Rightarrow (iii)$ 
Letting $\pi : \cl A\to C^*_{\rm e}(\Ss_{\G})$ be the 
*-epimorphism, arising from the extremal property of $C^*_{\rm e}(\Ss_{\G})$, 
we have that $\iota_{\rm e}( a_x) = \pi(\iota(p_x))$
is a projection, $x\in V$. 

$(iii) \Rightarrow (i)$     
follows by a standard application of Arveson's and Stinespring's theorems
(see the proof of Theorem \ref{th_dila}). 
\end{proof}

\begin{remark} \label{p_findimclassical} 
    Let $\G=(V,E)$ be a contextuality scenario, $\Hh$ be a finite dimensional 
    Hilbert space, and 
    $ A =(A_{x})_{x\in V} \subseteq \Bh$ be a POR of $\G$. 
    If $A$ is classically dilatable, then it admits a dilation acting
    on a finite dimensional Hilbert space.
Indeed, by Proposition \ref{p_clasdil}, there exists a 
    unital completely positive map $\phi : \Rr_{\G} \rightarrow \Bh$ 
    with $ \phi(d_{x})= A_{x}$, $ x\in V$. Note that $ \Rr_{\G}$ is an  operator subsystem
    in $\Dd_{\G}$, and $ \Dd_{\G}$ is an abelian finite dimensional $C^{*}$-algebra. 
    By Arveson's and Stinespring's theorems, 
    $A$ dilates into a PR with commuting entries acting on a finite dimensional space.
\end{remark}

\begin{example}\label{ex_qms}
\rm 
Let $n\in \bb{N}$. 
A \textit{quantum magic square (QMS)} of size $n\times n$ is 
a matrix $[a_{k,j}]_{k,j=1}^n$
with entries in $\cl B(\cl H)^+$ for some Hilbert space $\cl H$, such that 
$$\sum_{k=1}^n a_{i,k} = \sum_{k=1}^n a_{k,j} = I_{\cl H}, \ \ \ i,j \in [n].$$
A {\it QMS} is called a \textit{quantum permutation matrix (QPM)} if its entries are projections. 
Let $\G_n = (V,E)$, where $V= [n] \times [n]$ and 
$E = \{ \{i\}\times [n], [n]\times \{j\}: i,j \in [n]\}$; 
we call $\G_n$ the \textit{QMS scenario}. 
It is clear that quantum magic squares 
(resp. quantum permutation matrices) are precisely the 
positive operator (resp. projective) representations of $\G_n$.

The universal property of the operator system $ \Ss_{\G_{n}}$ 
(see Theorem \ref{universalprop}) implies that $ \Ss_{\G_{n}}$ 
is canonically completely order isomorphic to the operator system 
$\Ss_{X}$ defined before \cite[Theorem 4.1]{BHTT2} (for $X=[n]$). 
Note that the C*-algebra $C^*(\G_{n})$ coincides with the 
the quantum permutation group $C(S_{X}^+)$ \cite{Wang1998QuantumSG}; 
in view of Theorem \ref{statesonT0}, $ \Tt_{\G_{n}}$ is canonically 
completely order isomorphic to the operator system $ \Pp_{X}$  
(see the discussion before \cite[Proposition 4.4]{BHTT2}).
We finally note that the Birkhoff-von Neumann theorem implies that 
$ \Cc(\G_{n}) = \Qq(\G_{n}) = \Gg(\G_{n})$ and hence, by virtue of Theorem \ref{p_OMINs}, 
we have that $ \Rr_{\G_{n}} = \mathrm{OMIN}(\Tt_{\G_{n}}) = \mathrm{OMIN}(\Ss_{\G_{n}})$
up to canonical complete order isomorphisms. 

Let $ S_{n}$ be the permutation group of $n$ elements.
A quantum magic square $ A \in \Mn(\B(\ell^{2}([s])))$ is called \textit{semiclassical}, if 
$A = \sum_{\pi \in S_{n}}P_{\pi}\otimes q_{\pi}$,
where $ P_{\pi} \in \Mn$ are permutation matrices and $\{q_{\pi}\} \subseteq \B(\ell^{2}([s]))$ is a POVM. 
It was shown in \cite[Theorem 12]{doi:10.1063/5.0022344} 
that a quantum magic square acting on a finite dimensional Hilbert space 
is semiclassical if and only if it admits a dilation to a
quantum permutation matrix with commuting entries, acting on a finite dimensional Hilbert space. 
It thus follows from Proposition \ref{p_clasdil}, Proposition \ref{p_findimclassical} and 
        \cite[Theorem 12]{10.1063/1.4996867} that, 
    for a linear map $ \phi: \Rr_{\G_{n}} \rightarrow M_{s}$, the following are equivalent:
    \begin{enumerate}
        \item $ \phi: \Rr_{\G_{n}} \rightarrow M_{s}$ is a 
        unital completely positive map; 
        \item $ [ \phi(d_{i,j})]_{i,j=1}^{n} \in \Mn(\B(\ell^{2}([s])))$ is a semiclassical quantum magic square.
    \end{enumerate}
    \end{example}

It was shown in \cite[Theorem 19]{doi:10.1063/5.0022344}
that there exists a quantum magic square with $n\geq 3$
(acting on a finite dimensional Hilbert space)
that does not admit a dilation to a quantum permutation matrix acting on a finite dimensional Hilbert space.
In the sequel we will strengthen this result and 
note its consequences for the operator systems examined in Section \ref{s_opsys}.

\begin{corollary}\label{c_dilationsinfdim}
    The QMS scenario $ \G_{n}$ is not classically dilating if $n \geq  3$. 
    In particular, the  scenario $ \G_{3}$ is not dilating, and hence the operator systems 
    $\Ss_{\G_{3}}$ and $\Tt_{\G_{3}}$ are not completely order isomorphic. 
    \end{corollary}

    \begin{proof}
       Assume that $ \G_{n}$, for $ n \geq 3$ is classically dilating. 
       By Remark \ref{p_findimclassical}, any finite dimensional POR of $\G_{n}$, 
       that is, any quantum magic square of size $n\times n$, 
       admits a dilation into a finite dimensional quantum permutation matrix with commuting entries. 
       However, by \cite[Theorem 12]{10.1063/1.4996867} these matrices are exactly the semiclassical 
       quantum magic squares and for any $ n \geq 3$ and $ s \geq 2$, we can find a 
       QMS $ A \in \Mn( \B(\ell^{2}([s])))$ that is not 
       semiclassical \cite[Corollary 15]{10.1063/1.4996867}, a contradiction. 

Note that, for $ n \leq 3$,  the quantum permutation group $ C^{*}(\G_{n})= C(S_{n}^{+})$ is the abelian $C^{*}$-algebra $C(S_{n})$ \cite{Banica2006QuantumPG} which coincides with $\Dd_{\G_{n}}$. In particular, if $\G_{3}$ were dilating, it would be classically dilating which, by the previous paragraph, does not hold true. 
\end{proof}

We next strengthen Corollary \ref{c_dilationsinfdim} by extending 
some of the considerations of \cite{10.1063/1.4996867}. 
Let $ n\in \N$. Denote by $\Mm^{(n)}_{\mathfrak{s} }$  the set of quantum magic squares 
$ A= [A_{i,j}]_{i,j}\in \Mn(\B(\ell^{2}(\mathfrak{s})))$,  
for an arbitrary cardinal $ \mathfrak{s}$. Similarly, let 
$ \Pp^{(n)}_{\mathfrak{s}}$ denote the set of $n\times n$ quantum permutation matrices acting on $ \ell^{2}(\mathfrak{s})$. Let $ \Mm^{(n)}:= \bigsqcup_{\mathfrak{s}}\Mm^{(n)}_{\mathfrak{s}}$ and 
$ \Pp^{(n)}:= \bigsqcup_{\mathfrak{s}}\Pp^{(n)}_{\mathfrak{s}}  $ (disjoint unions). 
Finally, denote by $ \mathrm{ncconv}{(\Pp^{(n)})}_{\mathfrak{s}} $ 
the set of 
dilatable quantum magic squares in $\Mm^{(n)}_{\mathfrak{s} }$. 
We write
$M_{\sfr} = \B(\ell^{2}([\mathfrak{s}]))$ and view the elements of $M_{\sfr}$ 
as $\sfr\times \sfr$ matrices.

For $ A=[A_{i,j}] \in \Mm^{(n)}_{\mathfrak{s} }$, define 
$$    {\rm col}(A) := \sum_{i,j=1}^{n} e_{i}\otimes e_{j} \otimes A_{i,j} \in \C^{n} \otimes \Cn \otimes \; M_{\sfr}$$
and
$$ {\rm diag}(A) := \sum_{i,j=1}^{n}E_{i,i}\otimes E_{j,j} \otimes A_{i,j} \in M_{n}\otimes M_{n} \otimes M_{\sfr},$$
and let  
$\phi : \Mm^{(n)}_{\mathfrak{s} }\to (\Mn \otimes \Mn \otimes M_{\sfr})_{\rm sa}$ be the mapping, 
given by 
\begin{align*}
    \phi(A) = {\rm diag}(A) - {\rm col}(A) {\rm col}(A)^{*}, \ \ \ A\in \Mm^{(n)}_{\mathfrak{s}}.
\end{align*}
Finally, we let $ \mathcal{Z}^{(n)}$ denote the vector space of all matrices in $\Mn$ with zero diagonal.

We say that an element $ A \in \Mm^{(n)}_{\sfr}$ can be {\it positively perturbed}
if there exists an $ X \in (\mathcal{Z}^{(n)} \otimes \mathcal{Z}^{(n)} \otimes M_{\sfr})_{\rm sa} $ such that $ \phi(A) +X \geq 0$.
The following is a generalisation of \cite[Proposition 21 (ii)]{10.1063/1.4996867} in infinite dimensions.

\begin{lemma} \label{l_criterdilate}
    Every $ A \in  \mathrm{ncconv}{(\Pp^{(n)})}_{\mathfrak{s}}$ can be positively perturbed.
\end{lemma}

\begin{proof}
    Let $ A \in  \mathrm{ncconv{(\Pp^{(n)})}}_{\mathfrak{s}}  $ and write $ \Hh = \ell^{2}(\sfr) $. 
    Thus, there exists a Hilbert space $ \Kk$ containing $\Hh$ 
    and a QPM $ P=[P_{i,j}]_{i,j=1}^{n} \in \Mn(\Bk)$ such that, if 
    $ Q_{\Hh}$ is the projection onto $\Hh$, then 
    $A_{i,j} = Q_{\Hh} P_{i,j} \arrowvert_{\Hh}$. 
    With respect to the decomposition $ \Kk= \Hh\oplus \Hh^{\perp}$ we may write

\begin{align*}
    P_{i,j} = \begin{bmatrix}
        A_{i,j} & B_{i,j}^{*}\\
        B_{i,j} & C_{i,j}
    \end{bmatrix},
\end{align*}
    where $ B_{i,j} \in \B(\Hh,\Hh^{\perp})$ and $ C_{i,j} \in \B(\Hh^{\perp})$.
Since $ P$ is a QPM, we have that $ P_{i,j}^{2}=P_{i,j}$, $ P_{i,j}P_{i,k} = 0$ and 
$P_{j,i}P_{k,i}=0 $, whenever $ j \neq k$. Hence,
\begin{align*}
    B_{i,j}^{*}B_{i,j}& = A_{i,j} - A_{i,j}^{2},\\
    B_{i,k}^{*}B_{i,j}& = - A_{i,k}A_{i,j}, \ \mbox{ and }\\
    B^{*}_{j,i}B_{k,i}& = - A_{j,i}A_{k,i}.
\end{align*}
Letting $D =\sum_{i,j=1}^{n} e_{i}\otimes e_{j} \otimes B_{i,j}^{*}  $ and 
$X = DD^{*} - \phi(A)$,
we have that $ \phi(A) +X = DD^{*} \geq 0$.
\end{proof}

\begin{proposition}\label{p_nondilg3}
    For every $n \in \N$, $ n\geq 3$ there is a cardinal $\sfr$ and a QMS $ A \in \Mm^{(n)}_{\sfr}$ that 
    cannot be positively perturbed.
    In particular, $A$ does not admit a dilation into a QPM. 
\end{proposition}

\begin{proof}
    We use induction on $n $. By \cite[Proposition 23]{10.1063/1.4996867}, there exist a finite cardinal $\sfr$
    and an element $ A \in \Mm^{(3)}_{\sfr}$ that cannot be positively perturbed; in other words, 
    the claim holds for $ n=3$. Suppose that 
    the claim holds true for $n=k$, where $ k \geq 3$, 
    and let $ A \in \Mm^{(k)}_{\sfr}$ be such that
    \begin{align*}
        \phi(A) +X \ngeq 0, \ \ \ X\in (\mathcal{Z}^{(k)} \otimes \mathcal{Z}^{(k)} \otimes M_{\sfr})_{\rm sa}.
    \end{align*}
 Set $A' = A \oplus I_{\sfr}$, viewed as an element of $\Mm_{\sfr}^{(k+1)}$.
Let $ V = \begin{bmatrix}
        I_{k}\\
        0
    \end{bmatrix}$ (viewed as an element of $M_{k+1,k}(\C)$), 
    and $ X' \in (\mathcal{Z}^{(k+1)} \otimes \mathcal{Z}^{(k+1)} \otimes M_{\sfr})_{\rm sa} $. Then
    \begin{align*}
        X:= (V\otimes V\otimes I_{\sfr})^{*}X'(V\otimes V\otimes I_{\sfr}) \in (\mathcal{Z}^{(k)} \otimes \mathcal{Z}^{(k)} \otimes M_{\sfr})_{\rm sa}
    \end{align*}
and 
\begin{align*}
    (V\otimes V\otimes I_{\sfr})^{*}(\phi(A') +X')(V\otimes V\otimes I_{\sfr})=\phi(A) +X \ngeq 0,
\end{align*}
which implies that $\phi(A')+X' \ngeq 0$. The proof is complete.
\end{proof}

The following corollary is an immediate consequence of Theorem \ref{th_dila} and Proposition \ref{p_nondilg3}. 

\begin{corollary}\label{c_diffo}
The operator systems    $\Ss_{\G_{n}}$ and $\Tt_{\G_{n}}$ are not completely order isomorphic if $n\geq 3$. 
    \end{corollary}

We next single out a class of dilating contextuality scenarios.
Let $\Ss$ be an operator system, $e$ denote its unit and $\Aa$ be a unital $C^*$-algebra.  We call $\Ss$ an \textit{operator $\Aa$-system} \cite{Pa} if $\Ss$ is an $\Aa$-bimodule such that $ (a \cdot s)^{*} = s^{*} \cdot a^{*}$, $a \cdot e =  e \cdot a$ and $[a_{i,j}] \cdot [s_{i,j}] \cdot [a_{i,j}]^{*} \in \Mn(\Ss)^{+} $
 for all $[a_{i,j}] \in M_{n,m}(\Aa)$,  $ [s_{i,j}] \in M_{m}(\Ss)^{+}$,  $ s \in \Ss$ and $ a \in \Aa$. 
A \textit{faithful} operator $\Aa$-system \cite{coprod23} is an operator $\Aa$-system $\Ss$  such that $ a \cdot e \neq 0$ for all $a \neq 0$ in $\Aa$.

Let $ \Ss$ and $ \Tt$ be two operator $\Aa$-systems. A linear map $ \phi: \Ss \rightarrow \Tt$ is called an {\it$ \Aa$-bimodule map}, if for every $s \in \Ss$ and $ a_{1}, a_{2} \in \Aa$,    $\phi(a_{1} \cdot s \cdot a_{2} )= a_{1} \cdot \phi(s) \cdot a_{2}$.
Let $ \Ss_{1}$ and $ \Ss_{2}$ be two faithful operator $ \Aa$-systems. The {\it amalgamated coproduct} of $\Ss_{1}$ and $\Ss_{2}$ \cite[Definition 3.5]{coprod23} is the unique faithful operator $ \Aa$-system $\Ss_{1} \oplus_{\Aa} \Ss_{2}$, along with unital complete order embeddings $\phi_{i}:\Ss_{i} \hookrightarrow \Ss_{1} \oplus_{\Aa} \Ss_{2} $, $i= 1, 2$ that are also $\Aa$-bimodule maps,  such that the following universal property holds: For every operator $\Aa$-system $ \Rr $ and unital completely positive $\Aa$-bimodule maps $ \psi_{i} : \Ss_{i} \rightarrow \Rr$, $ i=1, 2$, there exists a unique unital completely positive $\Aa$-bimodule map $ \Psi :\Ss_{1} \oplus_{\Aa} \Ss_{2} \rightarrow \Rr $ such that $ \Psi \circ \phi_{i}= \psi_{i}$ for $ i =1, 2$.

We will show that, if the edges of $\G$ have  the same 
pairwise intersections then $\G$ is dilating. 
In fact, we will prove that $\Ss_{\G}$ is in this case an amalgamated 
coproduct of operator $\Aa$-systems.

\begin{theorem}\label{projectivedilation}  
Let $\G=(V,E)$ be a contextuality scenario, such that $e' \cap e''  = \bigcap_{e\in E}e \neq \emptyset$ for all $ e', e'' \in E$ with $ e' \neq e''$. Then $\G$ is dilating. In particular, there exists a unital abelian C*-algebra $\Aa$, such that 
    $\Ss_{\G} = \oplus_{\Aa} \ell^{\infty}_{e_{j}}$ and the subspace $\Jj$ in \ref{thekernel} is a completely proximinal kernel.
\end{theorem}

    \begin{proof}
Suppose that the hypergraph $ \G = (V,E)$ has $V = \{x_{1}, \dots,x_{m} \}$ and $ E = \{e_{1}, \dots,e_{d} \}$, (so that $ |e_{j}|\leq m$ for each $j\in [d]$).  We identify $\ell^{\infty}_{e_{j}}$ with 
the unital C*-algebra 
            $$ \Cc_{e_{j}} := \{ (\chi_{e_{j}}(x_{i})\mu_{i} )_{i=1}^{m} : \mu_{i} \in \C\}, $$
where $\chi_{e_{j}}$ is the characteristic function of $e_j$. 
%$\chi_{e_{j}}(x_{i}) = 1$ if $ x_{i}\in e_{j}$ and zero elsewhere.  
We view the algebras $\ell^\infty_{e_j}$ as linear  subspaces of $\C^m$. 
Set $f = \cap_{e\in E} e$ and 
define 
        $$ \Aa:= C^{*}\{ 1_{m}, (\chi_f(x_i)\mu_i)_{i=1}^{n} : \mu_{i} \in \C\} \subseteq \ell^{\infty}([m]),$$ 
    where $ 1_{m} $ is the unit in $\ell^{\infty}([m])$; 
    thus, $\Aa$ is generated by $1_m$ and only those canonical basis elements 
    $\delta_{x}^{m}$ of $ \C^{m}$ that satisfy $x \in f$. 
Clearly, $ \Aa$ is a unital C*-subalgebra of $ \ell^{\infty}([m])$.

We now prove that $ \G$ is dilating. 
Let  $(A_{x_{i}})_{i=1}^{m}\subseteq \Bh $ be a POR.  
For each $j=1,\cdots,d$, define the maps 
$\e_{j} : \Aa \rightarrow \Cc_{e_{j}}$ by letting
 $$            \e_{j}(\chi_{f}(x_{i})\mu_{i})_{i=1}^{m} = (\chi_{f}(x_{i})\mu_i)_{i=1}^{m}, \ \ 
            \e_{j}(1_m) = 1_{e_{j}}.              $$
We note the maps $\e_{j}$ are injective unital $*$-homomorphisms between $C^{*}$-algebras. Define the  maps 
    $\phi_{e_{j}} : \Cc_{e_{j}} \rightarrow \Bh$
    by letting
    $\phi_{e_{j}}(\delta_x^{e_{j}}) = A_{x}$, 
    and the linear map $ L :\Aa \rightarrow \Bh$ by letting
    $ L(1_{m})= I_{\Hh}$ and $ L(\delta_x^m)= A_{x} $ for $x \in f$. 
     We note that $ \phi_{e_{j}}$ are unital completely positive maps on $ \Cc_{e_{j}}$ that restrict to a common linear map of $\Aa$, 
namely $L$. By \cite[Theorem 3.2.]{davidson_kakariadis_2019}, there 
exists a unital completely positive map  $\Phi : *_{\Aa}\Cc_{e_{j}} \rightarrow \Bh$ 
such that $ \Phi\arrowvert_{\Cc_{e_{j}}} = \phi_{e_j}$, $j\in [d]$
(here $*_{\Aa}\Cc_{e_{j}}$ denotes the free product of the $C^{*}$-algebras $ \Cc_{e_{j}}$, amalgamated over the unital $C^{*}$-algebra $\Aa$, see \cite{davidson_kakariadis_2019}).
Noting that the basis elements $ \delta_{x_i}^{e_{j}}$ of $ \Cc_{e_{j}}$ are projections, 
a standard application of Stinespring's Dilation Theorem shows the existence of a dilation 
and completes the first part of the proof. 

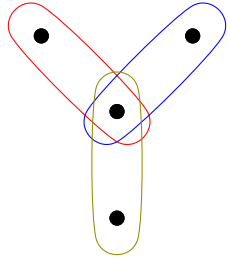
\begin{figure}
    \centering
  \begin{tikzpicture}

    \node (v2) at (1,3) {};
    \node (v3) at (2,2) {};
    \node (v4) at (2,0.59) {};
    \node (v5) at (3,3) {};

      \draw  [red] plot [smooth cycle] coordinates {(1,3.4) (0.6,3) (2,1.6) (2.4,2) };

     \draw [blue] plot [smooth cycle] coordinates { (3,3.4)  (3.4,3) (2,1.6) (1.6,2)  }; 
     
     \draw [olive] plot [smooth cycle] coordinates { (2.275,0.315)  (2.275,2.315) (1.725,2.315) (1.725,0.315)  };

    \foreach \v in {2,3,4,5} {
        \fill (v\v) circle (0.1);
    }

    \fill (v2) circle (0.1) node [below right] { };
    \fill (v3) circle (0.1) node [left] { };
    \fill (v4) circle (0.1) node [below] { };
    \fill (v5) circle (0.1) node [below right] { };

\end{tikzpicture}
    \caption{A dilating contextuality scenario.}
    \label{fig:my_label2}
\end{figure}

    Now note that $\Cc_{e_{j}} $ is a faithful operator $\Aa$-system for every $j\in [d]$ with module action $\Cc_{e_{j}} \times \Aa \rightarrow \Cc_{e_{j}}$ 
    given by pointwise multiplication, 
    $$     (\chi_{e_{j}}(x_{i})\mu_{i} )_{i=1}^{m} \cdot (\chi_f(x_{i})\nu_{i})_{i=1}^{m} 
    := (\chi_f(x_{i}) \mu_{i} \nu_{i})_{i=1}^{m}.$$
    Also, $ (\chi_{e_{j}}(x_{i})\mu_{i} )_{i=1}^{m} \cdot 1_{m} =  (\chi_{e_{j}}(x_{i})\mu_{i} )_{i=1}^{m} $. 
    For faithfulness, simply note that the unit of $\Cc_{e_{j}}$ is $ 1_{e_{j}}:=( \chi_{e_j}(x_i) )_{i=1}^m $; so letting $(\chi_f(x_i)\nu_{i})_{i=1}^m \in \Aa$ we have $( \chi_{e_j}(x_i) )_{i=1}^m \cdot  (\chi_f(x_{i})\nu_{i})_{i=1}^{m} =(\chi_f(x_{i})\nu_{i})_{i=1}^{m} $.
     
Using \cite[Theorem 3.3]{coprod23}, we form the operator $\Aa$-system coproduct 
$ \oplus_{\Aa} \ell^{\infty}_{e_{j}}$. It is straightforward to verify that $ \oplus_{\Aa} \ell^{\infty}_{e_{j}}$ satisfies the universal property of $\Ss_{\G}$ and thus, 
by Proposition \ref{uniqueness}, $ \Ss_{\G} = \oplus_{\Aa} \ell^{\infty}_{e_{j}}$
%\begin{equation*}\label{eq_SsAa}
 %\end{equation*}
 completely order isomorphically.
Finally, the claim that $\Jj$ is a completely proximinal kernel 
follows from \cite[Proposition 3.6]{coprod23}.
\end{proof}

An example of a contextuality scenario satisfying the assumptions of 
Theorem \ref{projectivedilation}, in particular,  
of a dilating scenario, is given in Figure \ref{fig:my_label2}.

\begin{remark} \label{r_bellscenarios}
    Let $X,  A$ be finite sets and $\mathbb{B}_{X,A}$ be the Bell scenario \cite{Acn2015ACA}, 
that is, $\mathbb{B}_{X,A} = (V,E)$, where 
$ V = X \times A$ and $ E= \{ \{x\} \times A: x\in X\}$.  We note that Bell scenarios $ \bb{B}_{X,A}$ are dilating. 
Indeed, 
in this case, $e' \cap e''=  \bigcap_{e\in E} e = \emptyset$, for all $ e', e'' \in E$, 
and hence a POR $ E=(E_{x,a})_{(x,a)\in X \times A}$ of $ \bb{B}_{X,A}$ is in fact a family of POVM's;
Theorem \ref{projectivedilation} now implies 
that any family of POVM's admits a dilation to a family of PVM's. In particular, $\Aa = \C $ and the same arguments show that the operator system $ \Ss_{\bb{B}_{X,A}} $ is unitally completely order isomorphic to   the unital coproduct $ \underbrace{\ell^{\infty}_{A}\oplus_{1}\cdots \oplus_{1}\ell^{\infty}_{A}}_{|X| \ {\rm times}}$ (see also \cite{Paulsen2013QUANTUMCN} and 
\cite{Lupini-etal}).
\end{remark}

We conclude this section by giving a necessary condition for a contextuality scenario to be dilating.
%If $ (V, \{C_{n}\}_{n \in \N},e)$ is an (abstract) operator system, set 
%$$\|u\|_o = \sup\{| \phi(u) |: \phi \text{ state of } V\}, \ \ \ u\in V.$$  

\begin{proposition} \label{testfordilating}
    Let $\G = (V,E)$ be a dilating contextuality scenario.
    For every $ x\in V$   
    there exists a probabilistic model $ p \in \Gg(\G)$ such that $ p(x) =1$.
    \end{proposition}

    \begin{proof}
        Since $\G$ is dilating, by Proposition \ref{corollaryC*env} the generators 
        $(a_{x})_{x\in V}$ of the operator system $ \Ss_{\G}$
        are projections in the C*-envelope $ C^{*}_{\rm e}(\Ss_{\G})$. 
        It follows that $\|a_x\| = 1$ for all $x\in V$.
        By Corollary \ref{p_dualSG},
        there exists a probabilistic model of $\G$ such that $ p(x)=1$.
    \end{proof}

\begin{remark}
Note that for the contextuality scenario $\G_{0}$ of Figure \ref{fig:noPRs2}, 
there is no probabilistic model $ p $ such that $ p(x_5)=1$. Indeed, if there were such a probabilistic model, then we would have that $p(x_i)= 0$ for all $ i=1,2,3,4$, 
contradicting the fact that $\{x_3,x_4\}$ is a hyperedge. 
Thus, according to our test in Proposition \ref{testfordilating}, 
this contextuality scenario is not dilating; in particular, 
$ \Ss_{\G_0} \neq \Tt_{\G_{0}}$.
\end{remark}

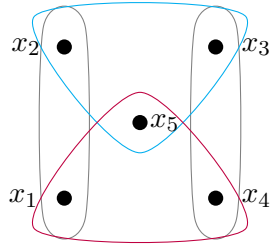
\begin{figure}
    \centering
  \begin{tikzpicture}
   \node (v1) at (0,2) {};
    \node (v2) at (1,3) {};
    \node (v3) at (2,2) {};
    \node (v4) at (4,2) {};
    \node (v5) at (5,3) {};
    \node (v6) at (6,2) {};
    \node (v7) at (3,1) {};
    \node (v8) at (2,0) {};
    \node (v9) at (4,0) {};

   % \draw [blue]  plot [smooth cycle] coordinates { (5,3.4)  (5.4,3) (4,1.6) (3.6,2)  }; 
     
     % \draw [blue]  plot [smooth cycle] coordinates {(3,1.4) (3.4,1) (2,-0.4) (1.6,0) };
      
 %      \draw [blue]  plot [smooth cycle] coordinates { (-0.275,2.275) (-0.275,1.725) (2.275,1.725) (2.275,2.275) };
 
     \draw [gray]  plot [smooth cycle] coordinates {(3.725,2.275) (4.275,2.275)  (4.275,-0.275)(3.725,-0.275) };
     
      \draw [gray]  plot [smooth cycle] coordinates {(1.725,2.275) (2.275,2.275) (2.275,-0.275) (1.725,-0.275) };
      
   %   \draw [blue]  plot [smooth cycle] coordinates {(1,3.4) (0.6,3) (2,1.6) (2.4,2) };
      
   %   \draw [blue]  plot [smooth cycle] coordinates {(5,3.4) (4.6,3) (6,1.6) (6.4,2) };

     % \draw [blue]  plot [smooth cycle] coordinates {(3,1.4) (2.6,1) (4,-0.4) (4.4,0) };
      
      \draw [cyan] plot [smooth cycle] coordinates {(1.6,2.4) (4.4,2.4) (3,0.6)  };

      \draw [purple] plot [smooth cycle] coordinates {(1.6,-0.4) (4.4,-0.4) (3,1.4)  };

   % \foreach \v in {1,2,...,9} {
   %     \fill (v\v) circle (0.1);
   % }

   % \fill (v1) circle (0.1) node [right] { };
  %  \fill (v2) circle (0.1) node [below left] { };
    \fill (v3) circle (0.1) node [left =5] { $x_2$};
    \fill (v4) circle (0.1) node [right =6] { $x_3$ };
  % \fill (v5) circle (0.1) node [below right] { };
  %  \fill (v6) circle (0.1) node [left] { };
   \fill (v7)  circle (0.1) node [right]
   {$x_5$};
   \fill (v8)  circle (0.1) node [left =6]
   { $x_1$};
   \fill (v9)  circle (0.1) node [right=6]
   {$x_4$ };

\end{tikzpicture}
    \caption{A non-dilating contextuality scenario.}
    \label{fig:noPRs2}
\end{figure}

%%%%%%%%%%%%%%%%%%%%%%%%%%%%%%%%%%%%%%%%%%%%%%%%%%%%%%%%%%%%%%%%%%%%%%%%%%
%%%%%%%%%%%%%%%%%%%%%%%%%%%%%%%%%%%%%%%%%%%%%%%%%%%%%%%%%%%%%%%%%%%%%%%%%%

\section{The universal C*-cover of $\cl S_{\G}$}\label{ss_univcov}

Let $\G = (V,E)$ be a non-trivial contextuality scenario. 
Extending the concepts of a stochastic operator matrix and 
of a quantum magic square, in this section
we define the notion of a $\G$-stochastic operator matrix, 
which allows us to obtain a natural 
C*-cover for the operator system $\Ss_{\G}$ defined in Subsection 
\ref{universalopsection}. 

\begin{definition}
    A \emph{$\G$-stochastic operator matrix} over a contextuality scenario $ \G = (V,E)$ is a 
    positive operator $A=[A_{x,x'}]_{x,x'\in V}$ in $M_{V}(\Bh) $ such that 
    $\sum_{x\in e}A_{x,x} = I_{\cl H}$, for every  $e\in E$.
\end{definition}

\begin{remark}\label{r_conntoPOR}
If $ (A_{x})_{x\in V} \subseteq \Bh$ is a POR of $\G$, 
then the assignment $ A_{x,x'}:= \delta_{x,x'}A_{x}$
defines a (diagonal) $\G$-stochastic operator matrix. 
Conversely, the diagonal of any $\G$-stochastic operator matrix 
determines a POR of $\G$.
\end{remark}

\begin{remark}\label{r_scGst}
Let $X, A$ be finite sets and $\mathbb{B}_{X,A}$ be the Bell scenario, 
that is, $\mathbb{B}_{X,A} = (V,E)$, where 
$ V = X \times A$ and $ E= \{ \{x\} \times A: x\in X\}$. 
The $\mathbb{B}_{X,A}$-stochastic matrices coincide with the 
stochastic operator matrices over $(X,A)$ introduced in 
\cite[Section 5]{Todorov2020QuantumNC}. 
Further, recalling the contextuality scenario $\G_n$ of the 
magic square, introduced in Example \ref{ex_qms}, we have that 
the $\G_n$-stochastic matrices are precisely the bistochastic 
operator matrices of size $n\times n$, defined in 
\cite[Section 3]{BHTT2}.
\end{remark}

The next proposition is a hypergraph analogue of 
\cite[Theorem 3.1]{Todorov2020QuantumNC}.

\begin{proposition} \label{theoremgstoch}
    Let $ \G=(V,E)$ be a 
    contextuality scenario and $A \in M_{V}(\Bh)^{+}$. 
    The following are equivalent:
    \begin{enumerate}
\item $A$ is a $\G$-stochastic operator matrix;
\item there exist a Hilbert space $\Kk$ and a family
        $ (U_{x})_{x\in V}\subseteq \B(\Hh,\Kk)$, such that 
        $ \sum_{x\in e} U^{*}_{x} U_{x} = I_{\Kk}$, 
        $ e\in E$, and 
        $A_{x,x'} = U^{*}_{x} U_{x'}$, $x,x' \in V$.
    \end{enumerate}
\end{proposition}

\begin{proof}
$(i) \Rightarrow (ii)$
Let $A = [A_{x,x'}] \in M_{V}(\Bh)$ be a $\G$-stochastic operator matrix. 
Let $ \Phi_{A} :M_{V} \rightarrow \Bh$ be the map, given by 
$\Phi_{A}(\delta_{x}\delta_{x'}^{*}) :=A_{x,x'}$, $x,x'\in V$.
By Choi's theorem, $\Phi_{A}$ is completely positive.
%note that $\Phi_{A}$ is unital if and only if $V$ is an edge of $\G$.
By Stinespring's Dilation Theorem, 
there exists a Hilbert space $\Kk_{1}$, a unital *-homomorphism $ \pi : M_{V} \rightarrow \B(\Kk_{1})$ and a bounded operator 
$ U \in \B(\Hh,\Kk_{1})$ with $ \nor{U}^{2}=\nor{\Phi_{A}(1)}$ such that 
\[ \Phi_{A}(T) = U^{*}\pi(T) U, \ \ \ T\in \cl B(\cl H). \]
 Up to unitary equivalence, we write
 $\Kk_{1} = \C^{V}\otimes \Kk$, 
 where $ \Kk$ is a Hilbert space and  
 $\pi( T )=  T \otimes I_{\Kk}$, $ T \in M_{V}$. 
 Further, write 
$U(\xi) = \sum_{x\in V} \delta_{x}\otimes U_{x}(\xi)$,
 where $ U_{x} : \Hh \rightarrow \Kk$ is a bounded operator, $x\in V$.
Let $ \xi , \eta \in \Hh $ and note that, if $ x,x'\in V$, then 
\begin{align*}
    \sca{A_{x,x'}\xi, \eta}& = \sca{\Phi_{A}(\delta_{x}\delta_{x'}^{*})\xi,\eta} 
    = \sca{U^{*} \pi(\delta_{x}\delta_{x'}^{*}) U \xi , \eta}
    = \sca{(\delta_{x}\delta_{x'}^{*} \otimes I_{\Kk}) U \xi, U\eta} \\
    & = \left\langle(\delta_{x}\delta_{x'}^{*} \otimes I_{\Kk}) 
    \left(\sum_{z\in V}\delta_{z}\otimes U_{z}(\xi)\right), \left(\sum_{w\in V}\delta_{w}\otimes U_{w}(\eta) \right)\right\rangle\\
    & = \left\langle\delta_{x} \otimes U_{x'}(\xi),\sum_{w\in V}\delta_{w}\otimes U_{w}(\eta)\right\rangle
    =\sca{U_{x'}(\xi), U_{x}(\eta)}
    = \sca{U_{x}^{*}U_{x'}\xi,\eta}.
\end{align*}
%Thus, we have that for every $\G$-stochastic operator matrix 
%$ A = [A_{x,x'}] \in M_{V}(\Bh)$, there exist operators $ (U_{x})_{x\in V} \subseteq %\B(\Hh,\Kk)$, where $\Kk$ is a Hilbert space, such that 
It follows that 
$A_{x,x'} = U_{x}^{*}U_{x'}$, $x,x'\in V$.

$(ii) \Rightarrow (i)$ is straightforward. 
\end{proof}

Recall that a {\it ternary ring of operators (TRO)} is a 
closed subspace $\cl U\subseteq \cl B(\cl H,\cl K)$, where $\cl H$ and $\cl K$
are Hilbert spaces, such that $S,T,R\in \cl U \Rightarrow ST^*R\in \cl U$. 
Our next aim is to identify a TRO, universal for $\G$-stochastic matrices.
Its construction follows the steps from 
\cite[Section 5]{Todorov2020QuantumNC}; for the convenience of the 
reader, we provide the basic details. 
Let $ \Vv$ be the ternary ring \cite{hestenes, zettl} 
generated by elements $ (v_{x})_{x\in V}$, 
satisfying the relations 
\begin{align}
    \sum_{x\in e}[v_{x'},v_{x},v_{x}] = v_{x'}, \; \; \; x'\in V, e\in E .
\end{align}
%The above relation implies that 
%\begin{align}
%    \sum_{x\in e}[u,v_{x},v_{x}] = u, \; \; \; u\in \Vv, e\in E .
%\end{align}
The non-degenerate ternary representations 
$ \theta : \Vv \rightarrow \B(\Hh,\Kk)$ 
are in bijective correspondence with the families
$ (U_{x})_{x\in V} \subseteq \B(\Hh,\Kk)$ such that 
\begin{align} \label{familyofop}
    \sum_{x\in e} U_{x}^{*}U_{x} = I_{\cl H}, \ \ \ e\in E,
\end{align}
via the assignment $\theta(v_{x}) = U_{x}$, $x\in V$.
Since the contextuality scenario $\G$ is assumed non-trivial, the collection  $\mathcal{F} $ of all non-degenerate ternary representations of $ \Vv$ is non-empty.  %Indeed, if $\G$ is non-trivial, there exists a POR $ (A_{x})_{x\in V}$, which gives rise to a $\G$-stochastic operator matrix. This, by  \ref{theoremgstoch} implies that there exists such a set of operators  satisfying
%\ref{familyofop}. 
Let $\hat{\theta} := \oplus_{ \theta \in \mathcal{F}}\theta $ 
be the direct sum of all (equivalence classes) of non-degenerate ternary representations of $\Vv$
on Hilbert spaces of dimension bounded by the cardinallity of $\Vv$, define
$\nor{u}_{0}:= \nor{\hat{\theta} (u)}$, 
and note that $\nor{u}_{0} < \infty$ for every $u\in \Vv$. 
Indeed, let $ \theta \in \mathcal{F}$ and set $U_{x} := \theta(v_{x})$, $x\in V$. 
Note that for every $ x \in V$, there exists an $e \in E$ such that $x\in e$. 
So, for $ x\in V$, relation (\ref{familyofop}) implies that 
\[ U_{x}^{*} U_{x} \leq \sum_{x'\in e}U_{x'}^{*} U_{x'} = I,\]
hence, $\nor{\theta(v_{x})} = \nor{U_{x}} \leq I$. 
This implies that $\nor{\hat{\theta}(v_{x})} =\sup_{\theta \in \mathcal{F} }\nor{\theta(v_{x})} \leq 1$ for every $ x \in V$.
Since $ (v_{x})_{x\in V}$ are the generators of $\Vv$, we have that $ \nor{u}_{0} :=\nor{\hat{\theta}(u)} < \infty$ for every $ u \in \Vv$.

Let $\mathcal{N} =\{ u \in \Vv : \nor{u}_{0}=0\}$ and note that $\mathcal{N}$ is a ternary ideal of $\Vv$. Now $ \normdot_{0}$ induces a norm on 
$ \Vv/\mathcal{N}$ denoted by $ \normdot$, 
and the completion with respect to this norm gives rise to a TRO, which we denote by $ \Vv_{\G}$.
We note that every ternary representation 
$\hat{\theta}$ of $\Vv$ gives rise to a canonical 
ternary representation of $ \Vv_{\G}$ onto a concrete TRO and $ \nor{u} = \nor{\hat{\theta}(u)}$ continues to hold for every $ u \in \Vv$. 
%Also, each  $\theta \in \mathcal{F}$ induces a a ternary representation of $\Vv_{\G}$ onto a concrete TRO.
So, $\hat{\theta}$ induces a faithful (isometric) representation of $ \Vv_{\G}$ as a concrete TRO in 
$\B(\Hh,\Kk)$, for some Hilbert spaces $\Hh$ and $\Kk$. 
Identify $ \Vv_{\G}$ with its image under $ \hat{\theta}$ and let $\Cc_{\G}$ be the the right C*-algebra of $\Vv_{\G}$, that is, 
\[ \Cc_{\G} = \overline{\spann}\{S^{*}T : S,T \in \Vv_{\G} \}. \] 
Let $g_{x,x'} := v_{x}^{*}v_{x'}$, $x,x'\in V$, viewed as elements of 
$\Cc_{\G}$, and
\begin{align}
    \mathcal{P}_{\G} :=\spann\{ g_{x,x'} : x,x' \in V\},
\end{align}
viewed as an operator subsystem of $ \Cc_{\G}$; note that 
\begin{equation}\label{eq_gxx}
\sum_{x\in e} g_{x,x} = 1, \; \; e\in E.
\end{equation}

\begin{theorem}\label{th_CstarGst}
    Let 
    $ \phi : \Pp_{\G} \rightarrow \Bh $ be a linear map. The following are equivalent.
    \begin{enumerate}
        \item $\phi $ is unital and completely positive;
        \item $[\phi(g_{x,x'})]_{x,x'\in V}$ is a $\G$-stochastic operator matrix;
        \item there exists a representation $ \pi: \Cc_{\G} \rightarrow \Bh$ such that $ \pi\arrowvert_{\Pp_{\G}} = \phi$.
    \end{enumerate}
\end{theorem}

    \begin{proof}
        $(i) \Rightarrow (ii) $ 
We claim that the matrix $[g_{x,x'}]_{x,x'\in V}$ is a positive element 
of $M_n(\Pp_{\G})$. To see this, assume that $\cl C_{\G}\subseteq \cl B(\cl K)$
as a unital C*-subalgebra, for some Hilbert space $\cl K$, and let 
$\theta : \cl V_{\G}\to \cl B(\cl K,\cl K')$ be a ternary representation, 
for a suitable Hilbert space $\cl K'$, such that 
$\cl C_{\G} = \overline{{\rm span}}(\theta(\cl V_{\G})^*\theta(\cl V_{\G}))$. 
Writing $V= \{x_{1},\dots,x_{m}\}$ and $T_i = \theta(v_{x_i})$, $i\in [m]$, 
we have that 
\[ [g_{x,x'}]_{x,x'\in V} =  \begin{bmatrix}
T_1 & \dots & T_m\\
 &   \bigcirc &
\end{bmatrix}^{*} \cdot \begin{bmatrix}
T_1 & \dots & T_m\\
 &   \bigcirc &
\end{bmatrix},\]
and hence 
$[g_{x,x'}]_{x,x'\in V}\in M_V(\cl B(\cl K))^+$. 
Since $\Pp_{\G}\subseteq \cl B(\cl K)$ as an operator subsystem, 
we have that $[g_{x,x'}]_{x,x'\in V}\in M_V(\Pp_{\G})^+$.
As $ \phi$ is completely positive, the matrix  $[\phi(g_{x,x'})]_{x,x'\in V}$ is positive; as $\phi$ is unital, by (\ref{eq_gxx}), 
$ \sum_{x\in e}\phi(g_{x,x}) = \phi(1)= I_{\cl H}$, $ e\in E$.

$(ii) \Rightarrow (iii)$ By Proposition \ref{theoremgstoch}, 
there exist a Hilbert space $ \Kk$ and operators $ U_{x} \in \cl B(\cl H,\cl K)$, $x \in V$, 
with $ \sum_{x\in e} U_{x}^{*}U_{x} = I_{\cl H}$, $ e\in E$, such that 
\[ \phi(v_{x}^{*}v_{x})= U_{x}^{*}U_{x'}, \; \; x,x' \in V. \]
The family $(U_x)_{x\in V}$ yields a non-degenerate representation $ \theta$ of $\Vv_{\G}$, which in turn gives rise to a unital *-representation $ \pi $ of $\Cc_{\G}$ via the assignment
\[ \pi(S^{*}T)= \theta(S)^{*}\theta(T), \; \; S,T \in \Cc_{\G}; \]
so, $\pi $ extends $ \phi$.

$(iii) \Rightarrow (i)$  $\phi$ is the restriction of a unital *-homomorphism, 
hence it is unital and completely positive.
    \end{proof}

We recall that the {\it universal C*-cover} of an 
operator system $\cl S$ is a C*-cover $(C^*_{\rm u}(\cl S), \iota_{\rm u})$
of $\cl S$ with the property that, if $\cl H$ is a Hilbert space and 
$\phi : \cl S\to \cl B(\cl H)$ is a unital completely positive map 
then there exists a unique *-representation $\pi : C^*_{\rm u}(\cl S)\to \cl B(\cl H)$
such that $\pi\circ\iota_{\rm u} = \phi$. 
We let $\cl B_{\G}$ be the C*-subalgebra of $\Cc_{\G}$, generated by
the elements $g_{x,x}$, $x\in V$.

\begin{corollary}\label{c_unicsta}
The following hold true:
\begin{enumerate}
\item   $\Ss_{\G} \subseteq \Pp_{\G}$ up to a canonical complete order 
embedding;

\item $C^{*}_{\rm u}(\Pp_{\G})$ is canonically *-isomorphic to $\Cc_{\G}$; 

\item $C^{*}_{\rm u}(\Ss_{\G})$ is is canonically *-isomorphic to $\cl B_{\G}$.
\end{enumerate}
\end{corollary}

\begin{proof}
(ii) is immediate from Theorem \ref{th_CstarGst} and
the universal property of $C^{*}_{\rm u}(\Pp_{\G})$.

(iii) 
Write $\cl A$ for the C*-subalgebra of $\cl C_{\G}$, generated by 
the set $\{g_{x,x} : x\in V\}$. 
Let $\cl H$ be a Hilbert space and $\phi : \Ss_{\G} \to \cl B(\cl H)$
be a unital completely positive map. 
Set $A_{x,x'} = \delta_{x,x'}\phi(g_{x,x})$, $x,x'\in V$.
By Remark \ref{r_conntoPOR}, 
the matrix $(A_{x,x'})_{x,x'\in V}$ is $\G$-stochastic. 
By Theorem \ref{th_CstarGst}, there exists a *-representation 
$\pi : \cl C_{\G}\to \cl B(\cl H)$, such that 
$\pi(g_{x,x'}) = A_{x,x'}$, $x,x'\in V$. The restriction of $\pi$ to 
$\cl A$ is a *-homomorphism that extends $\phi$. By the universal property of 
$C^{*}_{\rm u}(\Ss_{\G})$, it is *-isomorphic to $\cl A$. 

(i) The claim follows from (i) and (ii), and the 
facts that $\cl P_{\G}\subseteq C^{*}_{\rm u}(\Pp_{\G})$ and 
$\cl S_{\G}\subseteq C^{*}_{\rm u}(\Ss_{\G})$. 
\end{proof}

%It follows from Corollary \ref{c_unicsta} that 
%the (unital completely positive) map 
%$\phi_{\G} : \cl S_{\G}\to C^*(\G)$, given by 
%$\phi(g_{x,x}) = p_x$, $x\in V$
%extends to a canonical (unital) *-homomorphism 
%$\pi_{\G} : \cl C_{\G}\to C^*(\G)$. 

\begin{corollary}\label{c_dicha}
Let $\G$ be a non-trivial contextuality scenario. The following are equivalent:
\begin{enumerate}
\item $\G$ is dilating; 

\item there exists a unital completely positive map 
$\psi : C^*(\G) \to \cl C_{\G}$ such that 
$\psi(p_x) = g_{x,x}$, $x\in V$. 
\end{enumerate}
\end{corollary}

\begin{proof}
$(i)\Rightarrow(ii)$ 
Assume, without loss of generality, that 
$\cl C_{\G}\subseteq \cl B(\cl H)$ as a unital C*-subalgebra,
for some Hilbert space $\cl H$. 
Using the fact that $\G$ is dilating, let $(P_x)_{x\in V}$ be 
a PR of $\G$ on a Hilbert space $\cl K$ that dilates $(g_{x,x})_{x\in V}$. 
Let $\rho : C^*(\G)\to \cl B(\cl K)$ the *-homomorphism, 
such that $\rho(p_x) = P_x$, $x\in V$. 
We can take $\psi$ to be the compression of $\rho$ to $\cl H$. 

$(ii)\Rightarrow(i)$  
Suppose that $(A_x)_{x\in V}$ is a POR of $\G$. By Theorem \ref{th_CstarGst}, 
there exists a canonical *-homomorphism $\pi : \cl C_{\G}\to \cl B(\cl H)$,
such that $\pi(g_{x,x}) = A_x$, $x\in V$. 
Thus, $\psi\circ \pi : C^*(\G)\to \cl B(\cl H)$ is unital and completely 
positive; an application of Stinespring's Theorem yields 
a dilation of $(A_x)_{x\in V}$ that is a PR of $\G$. 
\end{proof}

%%%%%%%%%%%%%%%%%%%%%%%%%%%%%%%%%%%%%%%%%%%%%%%%%%%%%%%%%%%%%%%%%%%%%
%%%%%%%%%%%%%%%%%%%%%%%%%%%%%%%%%%%%%%%%%%%%%%%%%%%%%%%%%%%%%%%%%%%%%

\section{No-signalling probabilistic models}\label{s_NSpm}

In this section, we examine no-signalling probabilistic models, 
defined over the product of two contextuality scenarios, 
from the viewpoint of operator systems, characterising 
different types of no-signalling models via states on 
operator system tensor products. 
We start with some definitions. 

Let $\G=(V,E)$ and $ \Hbb =(W,F)$ be two contextuality scenarios. 
Writing $E \dot{\times} F = \{e\times f : e\in E, f\in F\}$, we call 
the contextuality scenario
$ \G \times \Hbb = (V\times W,E \dot{\times} F)$ the 
{\it product} of $\G$ and $\Hbb$. 
We set $\Gg(\G,\Hbb) := \Gg(\GH)$ (and write $\Gg = \Gg(\G,\Hbb)$ when no confusion arises).
A probabilistic model $p $ of $\GH $ is called {\rm no-signalling} \cite{Acn2015ACA} if 
$$    \sum_{x\in e}p(x,y) = \sum_{x\in e'}p(x,y), \; \;  y \in W, \;  e,e' \in E,$$
and
$$    \sum_{y\in f}p(x,y) = \sum_{y\in f'}p(x,y), \; \;  x \in V, \;  f,f' \in F.$$
We let $\Gg_{\rm ns}(\G,\Hbb)$ be 
the set of all no-signalling probabilistic models of $\GH$
(and write $\Gg_{\rm ns} = \Gg_{\rm ns}(\G,\Hbb)$ in case 
no confusion arises).
It is straightforward that the notion of a no-signalling probabilistic model
reduces to the notion of a no-signalling correlation in the case where 
$\G$ and $\Hbb$ are Bell scenarios (see Remark \ref{r_bellscenarios}).

We note that the inclusion $\Gg_{\rm ns} \subseteq \Gg$ can be strict.
Indeed, let $\G = (V,E)$, where 
$ V= \{x_1,x_2,x_3\}$, and $ E= \{e_1,e_2\}$ with $ e_1 = \{x_1,x_2\}$ and $e_2 = \{ x_2,x_3\}$. 
By the discussion after \cite[Definition 3.1.2]{Acn2015ACA},
$ \G \times \G$ admits a probabilistic model that fails to be no-signalling.

\begin{definition}
Let $ \G=(V,E)$ and $\Hbb=(W,F)$ be contextuality scenarios.
A no-signalling probabilistic model $ p$ of $\GH$ is called 
\begin{enumerate}
    \item \emph{deterministic}, if $p(x,y) \in  \{0,1 \} $ for all $ (x,y) \in V\times W$; 
    \item \emph{classical} if there exist an $m \in \N$ and 
    deterministic probabilistic models $ p_{i}^{1} $ of $ \G$ and $ p_{i}^{2}$ of $\Hbb$, $i=1,\dots, m$, 
    such that 
    \begin{align*}
        p(x,y) = \sum_{i=1}^{m}\lambda_{i}p_{i}^{1}(x)p_{i}^{2}(y), \; \; x\in V, \; y \in W,
    \end{align*}
as a convex combination;

\item a \emph{commuting probabilistic model}
(resp. a \emph{generalised commuting probabilistic model}), if there exist a Hilbert space $ \Hh$, 
a unit vector $ \xi  \in  \Hh $ and PR's (resp. POR's) 
$ (A_{x})_{x\in V}$ and $ (B_{y})_{y\in W}$ of $ \G$ and $\Hbb$ respectively, 
such that $ A_{x}B_{y}= B_{y} A_{x}$ for every $ x \in V$ and $y \in W$ and
\begin{equation}\label{eq_dqcmo}
    p(x,y) = \sca{A_{x}B_{y}\xi,\xi}, \; \; \; (x,y) \in V \times W;
\end{equation}

\item a \emph{tensor probabilistic model} 
(resp. \textrm{generalised tensor probabilistic model}), if 
there exist Hilbert spaces 
$\Hh_{\G}$ and $ \Hh_{\Hbb}$, 
and PR's (resp. POR's) 
$ (A'_{x})_{x\in V}\subseteq \B(\Hh_{\G})$ and $ (B'_{y})_{y\in W}\subseteq \B(\Hh_{\Hbb})$ of $ \G$ and $\Hbb$, respectively, such that (\ref{eq_dqcmo}) is fulfilled with 
$\Hh = \Hh_{\G} \otimes \Hh_{\Hbb}$, $A_x = A_x'\otimes I_{\Hh_{\Hbb}}$ and $B_y = I_{\Hh_{\G}}\otimes B_y'$, 
$x\in V$, $y\in W$.
\end{enumerate}
\end{definition}

The set of all classical no-signalling probabilistic models of $\GH$ 
will be denoted by $ \Cc(\G,\Hbb)$;
it is clear that $ \Cc(\G,\Hbb)$ is a convex polytope.
We let $ \Qq_{\rm qs}(\G,\Hbb)$ (resp. $\Qqc_{\rm qs}(\G,\Hbb)$)
the set of all tensor probabilistic models
(resp. generalised tensor probabilistic models) of $\GH$. 
We further denote by $ \Qq_{\rm q}(\G,\Hbb)$ (resp. $ \Qqc_{\rm q}(\G,\Hbb)$) 
the set of all tensor probabilistic models 
(resp. generalised tensor probabilistic models), 
such that the Hilbert spaces $ \Hh_{\G} $ and $ \Hh_{\Hbb}$ can be 
chosen to be finite dimensional.
We denote the closure of  $ \Qq_{\rm q}(\G,\Hbb)$ (resp. $ \Qqc_{\rm q}(\G,\Hbb)$) 
by $\Qq_{\rm qa}(\G,\Hbb) $ (resp. $ \Qqc_{\rm qa}(\G,\Hbb)$).
The set of all commuting probabilistic models
(resp. generalised commuting probabilistic models) 
will be denoted by $ \Qq_{\rm qc}(\G,\Hbb)$ (resp. $ \Qqc_{\rm qc}(\G,\Hbb)$).
We drop the dependence on $\G$ and $\Hbb$ in case no confusion arises, and 
note the inclusions
\begin{equation}\label{eq_lonin}
\Cc \subseteq \Qq_{\rm q} \subseteq \Qq_{\rm qs } \subseteq \Qqc_{\rm qs} \subseteq \Qqc_{\rm qc} \subseteq \Gg_{\rm ns} \subseteq \Gg, 
\ \ \ \ \Qq_{\rm q}\subseteq \Qqc_{\rm q } \subseteq \Qqc_{\rm qs} \ 
\mbox{ and }  \ 
    \Qq_{\rm qc} \subseteq \Qqc_{\rm qc}. 
\end{equation}

%%%%%%%%%%%%%%%%%%%%%%%%%%%%%%%%%%%%%%%%%%%%%%%%%%%%%%%%%%%%%%%%%%%%

%\subsection{Characterisation via states on tensor products}\label{ss_viastates}

As we saw in Corollary \ref{p_dualSG}, for a non-trivial contextuality 
scenario $\G$, the states of the universal operator system 
$ \Ss_{\G}$ are in bijective correspondence to the probabilistic models of $\G$. 
We next relate the no-signalling models of a product scenario 
$\GH$ with states on operator system tensor products.
We fix contextuality scenarios $\G=(V,E)$  and $ \Hbb=(W,F)$ and recall that
$ \Ss_{\G}$ and $ \Ss_{\Hbb}$ denote their universal operator systems 
(see Subsection \ref{universalopsection}). 
Recall, further, that 
$ \delta_{x}^{e}$ denotes the 
respective canonical basis element of the $e$-th summand $\ell^{\infty}_{e}$ in the quotient 
$\Ss_{\G} = \oplus_{e\in E} \ell^{\infty}_{e} / \Jjc_{\G}$, and that $a_x$ stands for its image 
in $\Ss_{\G}$ under the quotient map. 
Similarly, we denote by $ \delta_{y}^{f}$ the corresponding element of 
$ \ell^{\infty}_{f}$, and let $b_y$ be its image in $\Ss_{\Hbb} = \oplus_{f\in F} \ell^{\infty}_{f} / \Jjc_{\Hbb}$.

For a linear functional $ s : \Ss_{\G} \otimes \Ss_{\Hbb} \rightarrow \C$, let 
\begin{align*}
    p^{s}(x,y) := s(a_{x} \otimes b_{y}), \; \; x\in e, \; y \in f.
\end{align*}

%On the other hand, for a map
%$ p : V \times W \rightarrow \C$ that satisfies conditions (\ref{eq_nsonx}) and 
%(\ref{eq_nsony}), there exists a (unique) linear functional 
%$ s_{p} : \Ss_{\G} \otimes \Ss_{\Hbb} \rightarrow \C$ with $p^{s_p} = p$. 
%Letting
% $$ \mathcal{L} = \{ p : V \times W \rightarrow \C : p \text{ satisfies (\ref{eq_nsonx}) and 
%(\ref{eq_nsony})} \},$$ 
%we thus have that the correspondence $ s \mapsto p^{s}$ is a bijection between 
%$\cl L$ and the space of all linear functionals on $\Ss_{\G} \otimes \Ss_{\Hbb} $.

\begin{theorem}\label{th_cteprop}
Let $ \G$ and $\Hbb$ be contextuality scenarios and $p \in  \Gg(\G,\Hbb)$ be a probabilistic model. The following are equivalent:
\begin{enumerate}
    \item $ p \in \Gg_{\rm ns}(\G,\Hbb)$; 
    \item there exists a state $s$ on $ \Ss_{\G} \otimes_{\max}\Ss_{\Hbb} $ such that $p = p^s$.
\end{enumerate}
Moreover, the map $ s \mapsto p^{s}$ is an affine isomorphism between the states of 
$ \Ss_{\G} \otimes_{\max}\Ss_{\Hbb}$ and $ \Gg_{\rm ns}(\G,\Hbb)$.
\end{theorem}

\begin{proof}
$(ii) \Rightarrow (i) $ 
The fact that $p^s$ has non-negative values is a consequence of the fact that 
$a_x\otimes b_y\in (\Ss_{\G} \otimes_{\max}\Ss_{\Hbb})^+$, $x\in V$, $y\in W$. 
For $e, e'\in E$, we have 
\begin{eqnarray*}
    \sum_{x\in e}p^{s}(x,y) 
    & = & 
    s\left( \left(\sum_{x\in e}a_{x}\right)\otimes b_{y}\right) 
    = s( 1 \otimes b_{y} )\\
    & = & s\left(\sum_{x\in e'}a_{x} \otimes b_{y}\right) =\sum_{x\in e'}p^{s}(x,y);
\end{eqnarray*}
similarly, we see that 
$$\sum_{y\in f}p^{s}(x,y) =\sum_{y\in f'}p^{s}(x,y), \; \;  x \in V, \;  f,f' \in F.$$

$(i) \Rightarrow (ii)$ 
Recall that 
$\cl S = \oplus_{e\in E} \ell^{\infty}_e$, and write 
$\cl T = \oplus_{f\in F} \ell^{\infty}_f$. 
Let $q_{\G} : \cl S\to \cl S_{\G}$ and $q_{\bb{H}} : \cl T\to \cl S_{\bb{H}}$
be the quotient maps. 
We have that 
$q_{\G}^{\rm d} : \cl S_{\G}^{\rm d} \to \cl S^{\rm d}$ and 
$q_{\bb{H}}^{\rm d} : \cl S_{\bb{H}}^{\rm d} \to \cl T^{\rm d}$
are complete order embeddings; 
identifying $\cl S^{\rm d}$ (resp. $\cl T^{\rm d}$) with $\cl S$
(resp. $\cl T$), we consider $\cl S_{\G}^{\rm d}$
(resp. $\cl S_{\bb{H}}^{\rm d}$) as an operator system in 
$\cl S$ (resp. $\cl T$). 
Further,  by \cite[Proposition 1.16]{Farenick_Paulsen_2012},
$(\cl S_{\G}\otimes_{\max}\cl S_{\bb{H}})^{\rm d} = 
\cl S_{\G}^{\rm d}\otimes_{\min}\cl S_{\bb{H}}^{\rm d}$, 
and thus 
$(\cl S_{\G}\otimes_{\max}\cl S_{\bb{H}})^{\rm d}$ can be viewed as 
an operator system in $\cl S\otimes_{\min}\cl T$. 

Write $s : \cl S\otimes_{\max}\cl T\to \bb{C}$ for the positive linear functional,
given by 
$s(\delta_x^e\otimes \delta_y^f) = p(x,y)$; 
since $\delta_x^e\otimes \delta_y^f$, $x\in V$, $y\in W$, $e\in E$, $f\in F$, 
form a linear basis of $\cl S\otimes\cl T$, the element 
$s\in (\cl S\otimes_{\max}\cl T)^{\rm d}$
is well-defined. 
We show that, in fact, 
$s\in \cl S_{\G}^{\rm d}\otimes_{\min} \cl S_{\bb{H}}^{\rm d}$. 
To this end, it suffices to show that, for every $x\in V$ and $y\in W$, the slices
$L^x(s)$ and $L_y(s)$ of $s$, viewed as an elements of $\cl S$ and 
$\cl T$, respectively, belong to $\cl S_{\G}^{\rm d}$ and 
$\cl S_{\bb{H}}^{\rm d}$, respectively.  

Fix $y\in W$. By the no-signalling assumption, $L_y(s)$ annihilates 
the space $\cl J_{\G}$. Since $L_y(s)$ is a positive functional, 
it annihilates the kernel cover of $\cl J_{\G}$, namely $\tilde{\cl J}_{\G}$,
and hence $L_y(s)\in \cl S_{\G}^{\rm d}$. 
By symmetry, the slices along vertices $x\in V$ belong to 
$\cl S_{\bb{H}}^{\rm d}$.
We have thus shown that 
$s\in \cl S_{\G}^{\rm d}\otimes_{\min} \cl S_{\bb{H}}^{\rm d}$. 
By the definition of $s$, we have that $p = p^s$, and the proof is completed
using the first paragraph of the proof of this implication. 
\end{proof}

\begin{theorem}\label{th_qcqatilde}
Let $ \G$ and $\Hbb$ be contextuality scenarios and $ p\in \Gg(\G,\Hbb)$ be a probabilistic model. Then the following are equivalent:
\begin{enumerate}
    \item $p \in  \Qqc_{\rm qc}(\G,\Hbb)$ (resp. $p \in  \Qqc_{\rm qa}(\G,\Hbb)$); 
    \item there exists a state $s$ on $ \Ss_{\G} \otimes_{\rm c}\Ss_{\Hbb} $ 
    (resp. $ \Ss_{\G} \otimes_{\min}\Ss_{\Hbb}$) such that 
    $p = p^s$; 
    \item there exists a state $s$ on $\cl B_{\G}\otimes_{\max} \cl B_{\Hbb}$ 
    (resp. $\cl B_{\G}\otimes_{\min} \cl B_{\Hbb}$)
    such that $p = p^s$. 
\end{enumerate}
Moreover, the map $ s \mapsto p^{s}$ is an affine isomorphism between the 
state space of 
$ \Ss_{\G} \otimes_{\rm c}\Ss_{\Hbb}$ (resp. $ \Ss_{\G} \otimes_{\min}\Ss_{\Hbb}$) 
and $\Qqc_{\rm qc}(\G,\Hbb)$ (resp. $\Qqc_{\rm qc}(\G,\Hbb)$).
\end{theorem}

\begin{proof}
$(i) \Rightarrow (ii) $ Let $ p \in \Qq_{\rm qc}$; by definition, 
there exist a Hilbert space $ \Hh$, a unit vector $ \xi \in \Hh$ and positive operator representations 
$ (A_{x})_{x\in V}$ and $ (B_{y})_{y\in W}$ of $ \G$  and $ \Hbb$, respectively, such that 
$ A_{x}B_{y}=B_{y}A_{x}$ for all $ x,y$, and 
$    p(x,y)= \sca{A_{x}B_{y}\xi,\xi}$, $x\in V$, $y\in W$.
By Theorem \ref{universalprop}, there exist 
unital completely positive maps $ \phi : \Ss_{\G} \rightarrow \Bh$, 
and $ \psi : \Ss_{\Hbb} \rightarrow \Bh$, with with  $\phi(a_{x})=A_{x} $ and $ \psi(b_{y})=B_{y}$,
$x\in V$, $y\in W$. 
We have that $ \phi$ and $\psi$ have commuting ranges. 
Let $s : \cl S_{\G}\otimes\cl S_{\bb{H}}\to \bb{C}$ be the linear functional, 
given by 
\begin{align*}
     s(a_{x}\otimes b_{y}) = \sca{\phi(a_{x})\psi(b_{y})\xi,\xi}, \ \ \ x\in V,y\in W;
\end{align*}
we have that $s$ a state on $ \Ss_{\G}\otimes_{\rm c} \Ss_{\Hbb}$. 
It is straightforward that $p = p^s$. 

$(ii) \Leftrightarrow (iii)$ 
By Corollary \ref{c_unicsta}, 
$\cl B_{\G} = C^{*}_{\rm u}(\Ss_{\G})$, and now 
the statement follows from the complete order embedding 
\cite[Theorem 6.4]{KAVRUK2011267}
\begin{equation}\label{eq_incinu}
\Ss_{\G}\otimes_{\rm c}\Ss_{\Hbb} \subseteq 
\cl B_{\G}\otimes_{\max} \cl B_{\Hbb}.
\end{equation}

$(ii) \Rightarrow (i)$ Suppose that  $ s$ is a state on $ \Ss_{\G} \otimes_{\rm c} \Ss_{\Hbb}$ such that $p=p^{s}$. 
Using (\ref{eq_incinu}), 
extend $ s$ to a state $ \Tilde{s}$ on 
$\cl B_{\G}\otimes_{\max} \cl B_{\Hbb}$. 
Now the GNS representation of $\Tilde{s}$ produces unital *-representations 
$ \pi $ and $\rho$ of 
$\cl B_{\G}$ and  $\cl B_{\Hbb}$, respectively, on a Hilbert space $ \Hh$, 
with commuting ranges, and a unit vector $\xi \in \Hh$
  such that 
\begin{align*}
    \Tilde{s}(u \otimes v) = \sca{\pi(u)\rho(v)\xi,\xi}, \; \; \; 
    u \in \cl B_{\G}, \; 
    v\in \cl B_{\Hbb}.
\end{align*}
It follows that $ (\pi(a_{x}))_{x\in V}$ and $ (\rho(b_{y}))_{y\in W}$ 
are operator representations of $\G$ and $\Hbb$ respectively, and thus $ p \in \Qqc_{\rm qc}(\G,\Hbb)$.

The characterisations of the elements of $\Qqc_{\rm qa}(\G,\Hbb)$ can be proved 
using a direct modification of the proof of \cite[Theorem 2.8.]{Paulsen2013QUANTUMCN}; 
the detailed arguments are omitted.
\end{proof}

Let $ \G = (V,E)$ and $\Hbb= (W,F) $ be contextuality scenarios. 
We write $q_y$, $y\in W$, for the canonical generators of $C^{*}(\Hbb)$, 
and recall that the spaces $\Tt_{\G} = \spann\{ 1,p_{x}: x\in V \}$ and 
$\Tt_{\Hbb} =\spann\{ 1,q_{y}: y\in W \}$ are 
operator systems in $C^{*}(\G)$ and $C^{*}(\Hbb)$, respectively.
Recall, further, that the 
\emph{enveloping tensor product} $\cl S \otimes_{ \rm e} \Tt$
of two operator systems $\cl S$ and $\cl T$
is the operator system tensor product arising from the inclusion 
$\Ss \otimes \Tt \subseteq C^{*}_{\rm e}(\Ss) \otimes_{\max} C^{*}_{\rm e}(\Tt)$ \cite{KAVRUK2011267}.
For a linear functional  $s : \Tt_{\G} \otimes_{ \rm e} \Tt_{\Hbb }\to \bb{C}$, 
let $\hat{p}^s : V\times W\to \bb{C}$ be given by $\hat{p}^s(x,y) = s(p_x\otimes q_y)$, $x\in V$, $y\in W$.

\begin{theorem}\label{th_Qqcstates}
Let $ \G$ and $\Hbb$ be contextuality scenarios and $ p\in \Gg(\G,\Hbb)$ is a probabilistic model. The following are equivalent:
\begin{enumerate}
    \item $p \in  \Qq_{\rm qc}(\G,\Hbb)$ (resp. $p \in  \Qq_{\rm qa}(\G,\Hbb)$);  
\item there exists a state $s$ on $\Tt_{\G} \otimes_{ \rm e} \Tt_{\Hbb } $ 
(resp. $ \Tt_{\G} \otimes_{\min}\Tt_{\Hbb} $) such that $p = \hat{p}^s$;
\item there exists a state $s$ on $C^{*}(\G)\otimes_{\max}C^{*}(\Hbb)$ 
    (resp. $C^{*}(\G)\otimes_{\min}C^{*}(\Hbb)$) such that $p = \hat{p}^s$.
\end{enumerate}
Moreover, the map $ s \mapsto \hat{p}^{s}$ is an affine isomorphism between the state space of 
$ \Tt_{\G} \otimes_{\rm e}\Tt_{\Hbb}$ (resp. $\Tt_{\G} \otimes_{\min}\Tt_{\Hbb}$) and 
$ \Qq_{\rm qc}(\G,\Hbb)$ (resp. $ \Qq_{\rm qa}(\G,\Hbb)$).
\end{theorem}

\begin{proof}
$(i)\Leftrightarrow (iii)$
By definition, $p\in \Qq_{\rm qc}$ if and only if there exist a Hilbert 
space $\cl H$, a unit vector $\xi\in \cl H$ and commuting PR's
$(P_x)_{x\in V}$ and $(Q_y)_{y\in W}$ of $\G$ and $\Hbb$ respectively, 
on $\cl H$, such that 
$p(x,y) = \langle P_x Q_y \xi,\xi\rangle$, $x\in V$, $y\in W$. 
The claim follows from the fact that such commuting PR's give rise to 
*-representations of $C^{*}(\G)\otimes_{\max}C^{*}(\Hbb)$, and vice versa.

$(ii) \Leftrightarrow (iii)$ 
follows from the fact that $ C^{*}(\G) = C^{*}_{\rm e}(\Tt_{\G})$ (Lemma \ref{l_idCstareTG}).

The characterisations of the elements of $\Qq_{\rm qa}(\G,\Hbb)$ can again be proved 
using a direct modification of the proof of \cite[Theorem 2.8.]{Paulsen2013QUANTUMCN}.
\end{proof}

The following characterisation of the classical probabilistic models follows 
easily from Proposition \ref{p_repcl}; the proof is omitted. 

\begin{theorem}\label{th_Ccstates}
Let $ \G$ and $\Hbb$ be two contextuality scenarios and $ p\in \Gg(\G,\Hbb)$ is a probabilistic model. 
The following are equivalent
\begin{enumerate}
    \item $p \in  \Cc(\G,\Hbb)$;
    \item there exists a state $s$ on $\Dd_{\G}\otimes_{\min}\Dd_{\Hbb}$ such that 
$p(x,y) = s(d_{x} \otimes d_{y})$, $x\in V$, $y \in W$;
\item there exists a state $s$ on $\Rr_{\G} \otimes_{ \rm min} \Rr_{\Hbb } $ such that 
$p(x,y) = s(d_{x} \otimes d_{y})$, $x\in V$, $y \in W$.
\end{enumerate}
Moreover, the map $ s \mapsto p^{s}$ is an affine isomorphism between $ \Cc(\G, \Hbb)$ and 
the state space of $ \Rr_{\G} \otimes_{\rm min}\Rr_{\Hbb}$.
\end{theorem}

%\marginpar{\tiny AC. In particular $S(\Ss_{\G}\otimes_{\rm c} \Ss_{\Hbb}) =S(\Ss_{\G}\otimes_{\rm e} \Ss_{\Hbb}) $}

\begin{corollary}\label{c_dilapro}
Let $\G = (V,E)$ and $\Hbb = (W,F)$ be dilating contextuality scenarios. 
Then 
$\Tilde{\Qq}_{\rm q}(\G,\Hbb) = \Qq_{\rm q}(\G,\Hbb)$ and 
$\Tilde{\Qq}_{\rm qa}(\G,\Hbb) = \Qq_{\rm qa}(\G,\Hbb)$.
\end{corollary}

\begin{proof}
The equality 
$\Tilde{\Qq}_{\rm qa}(\G,\Hbb) = \Qq_{\rm qa}(\G,\Hbb)$ 
is a direct consequence of Theorems \ref{th_qcqatilde}, \ref{th_Qqcstates} and \ref{th_dila}. 
Assume that $p\in \Tilde{\Qq}_{\rm q}(\G,\Hbb)$ and let 
$(A_x)_{x\in V}$ (resp. $(B_y)_{y\in W}$) be a POR of $\G$
(resp. $\bb{H}$) on a finite dimensional 
Hilbert space $H$ (resp. $K$), and $\xi\in H\otimes K$
be a unit vector, such that 
$p(x,y) = \langle (A_x\otimes B_y)\xi,\xi\rangle$, $x\in V$, $y\in W$. 
Let $(P_x)_{x\in X}$ (resp. $(Q_y)_{y\in W}$) be a PR of $\G$ (resp. $\bb{H}$)
on a Hilbert space $\tilde{H}$ (resp. $\tilde{K}$) and 
$U_{\G} : H\to \tilde{H}$ (resp. $U_{\bb{H}} : K\to \tilde{K}$) 
be an isometry, such that 
$A_x = U_{\G}^*P_x U_{\G}$, $x\in V$ (resp. $B_y = U_{\bb{H}}^*Q_y U_{\bb{H}}$, $y\in W$). 
Setting $\tilde{\xi} = (U_{\G}\otimes U_{\bb{H}})\xi$, we have that 
$\tilde{\xi}$ is a unit vector in $\tilde{H}\otimes \tilde{K}$ and 
$$p(x,y) = \langle (P_x\otimes Q_y)\tilde{\xi},\tilde{\xi}\rangle, \ \ \ x\in V, y\in W.$$
Thus, $p\in \Qq_{\rm q}(\G,\Hbb)$. 
We have shown that 
$\Tilde{\Qq}_{\rm q}(\G,\Hbb)\subseteq \Qq_{\rm q}(\G,\Hbb)$
and since the reverse inclusion is trivial, we have that 
$\Tilde{\Qq}_{\rm q}(\G,\Hbb) = \Qq_{\rm q}(\G,\Hbb)$. 
\end{proof}

Let $\G = (V,E)$ and $ \Hbb = (W,F)$ be contextuality scenarios. 
Recall that the \textit{Foulis-Randall product} $\G \otimes \Hbb$ 
of $\G$ and $\Hbb$ is the hypergraph with the properties that 
$ \Gg(\G \otimes \Hbb)$ coincides with the 
set of no-signalling models of $\GH$ and that 
the projective representations of $\G \otimes \Hbb$ necessarily satisfy 
the respective no-signalling conditions \cite{Acn2015ACA}. 
Consider the following universal $C^{*}$-algebra given by generators and relations:
\begin{multline} \label{univers_nosign}
\cl A_{\G,\Hbb} = \Biggl
     \langle (p_{x,y})_{x\in V,y\in W} \Big| p_{x,y}=p_{x,y}^{*}=p_{x,y}^{2}, \\
     \sum_{(x,y)\in e\times f}p_{x,y}=1, e\in E,f\in F,
    \sum_{x \in e}p_{x,y}= \sum_{x'\in e'}p_{x',y} , \sum_{y\in f}p_{x,y} =\sum_{y'\in f'}p_{x,y'} \Biggl \rangle.
\end{multline}
%defined as usual as the completion of the analogous *-algebra, 
%presented by generators satisfying the same relations, with respect to the seminorm 
%$\norm{x} :=\sup \norm{\pi(x)}$, 
%where the supremum is over all representations of the *-algebra into some Hilbert space. Note that if such representations exist, then the above seminorm is going to be finite since the generators are projections. 
It is easy to see that, if
$\mathcal{I}$ is the closed ideal of $C^{*}(\GH)$ generated by the elements 
\begin{align*}
    \sum_{x \in e}p_{x,y}- \sum_{x'\in e'}p_{x',y}, \; \; \; \; \sum_{y\in f}p_{x,y} -\sum_{y'\in f'}p_{x,y'},  
\end{align*}
then $\cl A_{\G,\Hbb} = C^{*}(\GH)/ \mathcal{I}$.
By the arguments in Proposition 5.2.1 in \cite{Acn2015ACA} and the universal property 
of $\cl A_{\G,\Hbb}$,
we have that $\cl A_{\G,\Hbb}= C^{*}(\G \otimes \Hbb)$. 
In the next proposition, we provide a concrete 
C*-algebraic visualisation of $\G\otimes\Hbb$ and of the enveloping tensor product 
$\cl T_{\G}\otimes_{\rm e} \cl T_{\Hbb}$, appearing in the description 
of the class $\cl Q_{\rm qc}(\G,\Hbb)$ in Theorem \ref{th_Qqcstates}.

\begin{proposition}\label{p_maxotHbb}
Let $\G = (V,E)$ and $ \Hbb =(W,F)$ be contextuality scenarios. Then 
$C^{*}(\G\otimes \Hbb) = C^{*}(\G) \otimes_{\max} C^{*}(\Hbb)$, up to a canonical 
*-isomorphism. In particular, $ \Tt_{\G\otimes \Hbb} = \Tt_{\G} \otimes_{\rm e} \Tt_{\Hbb}$, 
up to a canonical unital complete order isomorphism.
\end{proposition}

\begin{proof}
Let $\cl H$ be a Hilbert space and 
$(P_{x,y})_{x\in V,y\in W} \subseteq \Bh$ be a projective representation 
that satisfies the no-signalling conditions 
\[ \sum_{x\in e} P_{x,y} = \sum_{x'\in e'} P_{x',y} \ \ \mbox{ and } \ \ 
\sum_{y\in f} P_{x,y} = \sum_{y'\in f'} P_{x,y'},\]
for all $ x \in V, y\in W$, $e,e' \in E$ and $f, f' \in F$. 
Set 
\begin{align*}
    P_{x}: = \sum_{y\in f} P_{x,y} \ \mbox{ and } \ Q_{y}: = \sum_{x\in e} P_{x,y}, \ \ x\in V, y\in W
\end{align*}
(here $e\in E$ and $f\in F$ are arbitrarily chosen). 
We have that $(P_x)_{x\in V}$ (resp. $(Q_y)_{y\in W}$) is a projective representation of 
$\G $ (resp. $\Hbb$); let  
$ \pi_{\G} $ (resp. $ \pi_{\Hbb}$) be the corresponding *-representation of $ C^{*}(\G)$ 
(resp. $C^{*}(\Hbb)$). 
We have that, if $x\in V$ and $y\in W$, then 
$$P_{x}Q_{y}
= \left(\sum_{y'\in f} P_{x,y'}\right) \left(\sum_{x'\in e} P_{x',y}\right) 
= \sum_{y' \in f, x' \in e}P_{x,y'} P_{x',y}
= P_{x,y} = Q_{y}P_{x}.$$
Thus,  $ \pi_{\G} $ and $ \pi_{\Hbb}$ have commuting ranges 
and therefore give rise to a *-representation 
$\pi_{\G}\times \pi_{\Hbb} : C^{*}(\G) \otimes_{\max } C^{*}(\Hbb) \rightarrow \Bh$
such that 
$$(\pi_{\G}\times \pi_{\Hbb})(p_x\otimes q_y) = P_{x,y}, \ \ \ x\in V, y\in W.$$
The first claim now follows from the 
the universal property of $ C^{*}(\G\otimes \Hbb)$. 
The second claim follows from the first one, the definition of the 
tensor product $\Tt_{\G} \otimes_{\rm e} \Tt_{\Hbb}$ and Lemma \ref{l_idCstareTG}. 
\end{proof}

\begin{remark}
It follows from Theorem \ref{th_Qqcstates} and Proposition \ref{p_maxotHbb}
that the following are equivalent:
    \begin{enumerate}
    \item $p \in  \Qq_{\rm qc}(\G,\Hbb)$;
\item there exists a state $s$ on $\Tt_{\G\otimes \Hbb} $ such that 
$p(x,y) = s( p_{x,y})$, $x\in V$, $y\in W$.
\end{enumerate}
\end{remark}

%%%%%%%%%%%%%%%%%%%%%%%%%%%%%%%%%%%%%%%%%%%%%%%%%%%%%%%%%%%%%%%%%%%
%%%%%%%%%%%%%%%%%%%%%%%%%%%%%%%%%%%%%%%%%%%%%%%%%%%%%%%%%%%%%%%%%%%
%%%%%%%%%%%%%%%%%%%%%%%%%%%%%%%%%%%%%%%%%%%%%%%%%%%%%%%%%%%%%%%%%%%

\section{Coherent probabilistic models}\label{ss_coherent}

In this section, we define coherent probabilistic models, which generalise 
synchronous no-signalling correlations and show that, similarly to the latter ones, 
the former can be described in terms of tracial states on universal C*-algebras. 
For a hypergraph $ \G=(V,E)$, we write $ x \sim y$, if $ x \neq y$ and there exists a 
$ e \in E$ such that $ x,y\in e$.

\begin{definition}
    Let $ \G= (V,E)$ be a contextuality scenario. A probabilistic model $ p \in \Gg(\G,\G)$ 
    is called {\rm coherent} if 
    $p(x,y) =0$ whenever $x \sim y$.
\end{definition}

We denote the set of coherent probabilistic models by  $\Gg^{\rm c}(\G) $. Similarly, 
we let $\Gg^{\rm c}_{\rm ns}(\G)$
(resp. $\Qq_{\rm qs}^{\rm c}(\G), \Qq_{\rm qc}^{\rm c}(\G) $) denote the set of all no-signalling coherent probabilistic models 
(resp. tensor coherent probabilistic models, 
quantum commuting coherent models).

\begin{remark}\label{r_synchr}
Let $X$ and $A$ be finite sets. 
Recall that a no-signalling correlation $p = \{(p(a,b|x,y))_{a,b\in A} : x,y\in X\}$
is called {\it synchronous}  (resp. {\it bisynchronous}) 
\cite{PAULSEN20162188} (resp. \cite{Paulsen2019BisynchronousGA}), if 
$p(a,b|x,x) = 0$ whenever $a\neq b$ (resp. 
$p(a,b|x,x) = 0$ whenever $a\neq b$ and $p(a,a|x,y) = 0$ whenever $x\neq y$). 

Let $\mathbb{B}_{X,A}$ be the Bell scenario 
(see Remark \ref{r_scGst}). It is straightforward to check that 
a no-signalling probabilistic model $p$ of $\mathbb{B}_{X,A}$ is coherent 
if and only if $p$ is a synchronous no-signalling correlation. 
Further, it is straightforward to verify that 
a no-signalling probabilistic model $p$ of the QMS scenario
$\G_{n}$ (see Example \ref{ex_qms}) is coherent if and only if it is 
a bisynchronous no-signalling correlation. 
Thus, coherency can be viewed as a contextual version of synchronicity
and bisynchronicity. 
\end{remark}

For a unital $C^{*}$-algebra  $ \Aa$, we write $\Aa^{ \rm op}$ 
for the opposite $C^{*}$-algebra;
recall that $\Aa^{ \rm op}$ has the same underlying vector space with $\Aa$, the same involution and norm, 
and multiplication given by $ a^{\rm op} b^{\rm op} = (ba)^{\rm op}$, where by $ a^{\rm op}$, for $ a\in \Aa$, we denote the 
corresponding elements of $ \Aa^{\rm op}$. 

Let $\Hh, \Kk$ be two Hilbert spaces. We denote by $ \Hh^{\rm d}$ the Banach space dual of $\Hh$ (similarly for $\Kk$). For $T\in \B(\Hh,\Kk)$, we denote by $ T^{\rm d} \in \B(\Kk^{\rm d}, \Hh^{\rm d})$ the dual operator of $T$; note that 
\[ T^{\rm d} (g) = g \circ T, \; \; \; g \in \Kk^{ \rm d}, \]
and $(T^{\rm d})^{*} = (T^{\rm *})^{\rm d}$. 
A *-representation  
$ \pi :\Aa \rightarrow \Bh$ of a C*-algebra $ \Aa$ induces a representation  
$ \pi^{ \rm op} : \Aa^{\rm op} \rightarrow \B(\Hh^{\rm d})$ 
of $\Aa^{ \rm op}$ by letting $ \pi^{\rm op}(a^{\rm op}) = \pi(a)^{\rm d}$.

\begin{lemma}\label{l_opfreehyper}
    Let $\G=(V,E)$ be a hypergraph. There is a $ *$-isomorphism $ \theta : C^{*}(\G) \rightarrow C^{*}(\G)^{\rm op}$, such that $\theta(p_{x}) = p_{x}^{\rm op}$.
\end{lemma}

\begin{proof}
%Assume, without loss of generality, that $C^{*}(\G) \subseteq\Bh$ as 
%a unital C*-subalgebra; 
%then $C^{*}(\G)^{\rm op} \subseteq \cl B(\cl H^{\rm d})$ as a unital C*-subalgebra. 
We show that $C^{*}(\G)^{\rm op}$ satisfies the universal property of 
$C^{*}(\G)$. To this end, let 
$(P_x)_{x\in V}\subseteq \cl B(\cl K)$ be a PR of $\G$. 
A direct verification shows that $(P_x^{\rm d})_{x\in V}$ is a PR
of $\G$ on the Hilbert space $\cl K^{\rm d}$.
By the universal property of the C*-algebra $C^{*}(\G)$, there exists 
a (unital) *-representation $\rho : C^{*}(\G)\to \cl B(\cl K^{\rm d})$,
such that $\rho(p_x) = P_x^{\rm d}$, $x\in V$. 
It follows that  the map $\rho^{\rm op} : C^{*}(\G)^{\rm op} \to \cl B(\cl K)$
is a *-representation, satisfying $\rho^{\rm op}(p^{\rm op}) = P_x$, $x\in V$. 
The proof is complete. 
\end{proof}

In view of Remark \ref{r_synchr}, 
the following results are generalisations of the classical characterisations 
obtained in \cite{PAULSEN20162188}, \cite{Paulsen2019BisynchronousGA} and \cite{10.1063/1.4996867}.
We recall that a trace $\tau $ of a $C^{*}$-algebra $ \Aa$ 
is called {\it amenable}, 
if there is a sequence of completely positive, contractive maps
        $ \varphi_{k} : \Aa \rightarrow M_{d_{k}}$ such that 
        \begin{align*}
             \nor{\varphi_{k}(ab)- \varphi_{k}(a)\varphi_{k}(b)}_{2}  \rightarrow 0 \; \; \text{ and } \; \; \trace_{d_{k}}{(\varphi_{k}(a))} \rightarrow \tau(a).
        \end{align*}
The following result is due to Kirchberg \cite[Proposition 3.2]{Kirchberg1994} 
(see also \cite[Theorem 6.2.7]{brown2008textrm}).

\begin{theorem}  \label{th_kirchberh}
    If $\Aa$ is a unital, separable $C^{*}$-algebra and $ \tau$ is a tracial state on $\Aa$, the following are equivalent:
    \begin{enumerate}
        \item $\tau$ is an amenable tracial state; 
%        \item there is an embedding $ \pi_{\tau}(\Aa)'' \subseteq \Rr^{\omega}$, such that     $ \pi_{\tau} : \Aa \rightarrow \pi_{\tau}(\Aa)''\subseteq \Rr^{\omega}$ has a ucp lift $\Aa \rightarrow \ell^{\infty}(\Rr)$ with $\tau_{\omega} \circ \pi_{\tau} = \tau$.
        \item the linear functional $\phi : \Aa \otimes \Aa^{\rm op} \rightarrow \C $, with $ \phi(a\otimes b^{\rm op} ) = \tau(ab)$, is bounded with respect to the minimal tensor norm.
    \end{enumerate}
\end{theorem}

\begin{theorem}\label{th_cohtrace}
    Let $ \G = (V,E)$ be a contextuality scenario and $ p \in \Gg(\G, \G)$.  
    The following are equivalent:
    \begin{enumerate}
        \item $ p \in \Qq_{\rm qc}^{\rm c}(\G) $; 
        \item there  exists a tracial state $ \tau : C^{*}(\G) \rightarrow \C$ such that
\begin{equation}\label{eq_taupxpy}
p(x,y) = \tau(p_{x}p_{y}), \; \; \; x,y \in V.
\end{equation}
    \end{enumerate}
Moreover, 
\begin{itemize}
\item[({\rm a})] 
$p\in \Qq_{\rm qa}^{\rm c}(\G)$
if and only if $\tau$ in (\ref{eq_taupxpy}) can be chosen to be amenable; 

\item[({\rm b})] 
$p\in \Qq_{\rm q}^{\rm c}(\G)$ if and only if $\tau$ in (\ref{eq_taupxpy}) can be chosen
of the form $\tau = \tau_0\circ \pi_0$, where $\pi_0$ is a *-representation of $C^{*}(\G)$ such that 
${\rm ran}(\pi_0)$ is finite dimensional, and $\tau_0$ is a tracial 
state on ${\rm ran}(\pi_0)$; 

\item[({\rm c})] 
$p\in \Cc^{\rm c}(\G)$ if and only if $\tau$ in (\ref{eq_taupxpy}) can be chosen
of the form $\tau = \tau_0\circ \pi_0$, where $\pi_0$ is a *-representation of $C^{*}(\G)$ such that 
${\rm ran}(\pi_0)$ is abelian, and $\tau_0$ is a state on ${\rm ran}(\pi_0)$. 
\end{itemize}
\end{theorem}

    \begin{proof}
$(i)\Rightarrow (ii)$
We write $\cl A = C^*(\G)$ for brevity. By Theorem \ref{th_Qqcstates}, 
there exists a state $s : \cl A\otimes_{\max}\cl A\to \bb{C}$, such that 
$p(x,y) = s(p_x\otimes p_y)$, $x,y\in V$. 
Fix $x\in V$ and let $e\in E$ be such that $x\in e$. Since $p$ is coherent, 
\begin{equation}\label{eq_sym0}
s(p_x\otimes 1) = 
\sum_{y\in e} s(p_x\otimes p_y) = s(p_x\otimes p_x) 
= \sum_{y\in e} s(p_y\otimes p_x)
=  s(1\otimes p_x).
\end{equation}
Letting $h_x = p_x\otimes 1 - 1\otimes p_x$, (\ref{eq_sym0}) implies that 
$s(h_x^2) = 0$ which, by the Cauchy-Schwarz inequality, 
shows that 
$s(h_xu) = s(uh_x) = 0$, $u\in \cl A\otimes_{\max}\cl A$. 
Thus
\begin{equation}\label{eq_pxz}
s(zp_x\otimes 1) = s(z\otimes p_x) = s(p_xz\otimes 1), \ \ \ z\in \cl A.
\end{equation}
Assume that 
$s(zw\otimes 1) = s(wz\otimes 1)$ for every $z\in \cl A$ and every word $w$ on 
the set $\{p_x : x\in V\}$ of length at most $n$. Letting $w$ be a word on 
$\{p_x : x\in V\}$ of length $n+1$, write $w = w_0p_x$ for some $x\in V$ and 
a word $w_0$ on $\{p_x : x\in V\}$ of length $n$. Using (\ref{eq_pxz}), we then have 
$$s(zw\otimes 1) = s(zw_0p_x\otimes 1) = s(p_xzw_0\otimes 1) = 
s(w_0p_xz\otimes 1) = s(wz\otimes 1).$$
Thus, the functional $\tau$ on $\cl A$, given by $\tau(z) = s(z\otimes 1)$, 
is a tracial state. 
Finally, (\ref{eq_pxz}) implies that 
$$p(x,y) = s(p_x\otimes p_y) = \tau(p_x p_y), \ \ \ x,y\in V.$$

$(ii)\Rightarrow (i)$
The functional $\tilde{s} : \cl A\otimes_{\max}\cl A^{\rm op}\to \bb{C}$, given by 
$\tilde{s}(u\otimes v^{\rm op}) = \tau(uv)$, $u,v\in \cl A$, is a (well-defined) state. 
By Lemma \ref{l_opfreehyper}, the functional
$s : \cl A\otimes_{\max}\cl A\to \bb{C}$, given by 
$\tilde{s}(u\otimes v) = \tau(uv)$, $u,v\in \cl A$, is a (well-defined) state.
Applying the GNS construction to $s$ yields PR's and a unit vector in the 
corresponding GNS Hilbert space, implementing the probabilistic model $p$.

\smallskip

(a)
Now let $p \in \Qq_{\rm qa}^{\rm c}(\G)$. 
By Theorem \ref{th_Qqcstates}, there exists a state 
$ s : C^{*}(\G) \otimes_{\rm min} C^{*}(\G) \rightarrow \C$ such that 
$p(x,y) = s(p_{x}\otimes p_{y})$, $x,y\in V$.
The arguments from the first part of the proof show that the 
state $\tau$ on $C^{*}(\G)$, given by $\tau(a) = s(a\otimes 1)$, is tracial.
By Lemma \ref{l_opfreehyper}, 
the state $\tilde{s} : C^{*}(\G) \otimes_{\rm max} C^{*}(\G)^{\rm op} \to \bb{C}$, 
given by $\tilde{s}(a\otimes b^{\rm op}) = \tau(ab)$, 
is bounded in the minimal norm. By Theorem \ref{th_kirchberh}, the trace $\tau$ is amenable. 
As in the first part of the proof, $p(x,y) = \tau(p_xp_y)$, $x,y\in V$. 

Conversely, let $ \tau $ on $C^{*}(\G)$ be an amenable trace. 
By the first part of the proof, the assignment $p : V\times V\to \bb{R}_+$, given by
$p(x,y) = \tau(p_{x}p_{y})$, $x,y \in V$, is a quantum commuting coherent model. 
By Theorem \ref{th_kirchberh}, the functional $ \phi :  C^{*}(\G) \otimes C^{*}(\G)^{\rm op} \rightarrow \C$ with 
\[ \phi(a\otimes b^{\rm op}) = \tau(ab),  \]
    is bounded with the minimal tensor norm, and hence extends to a state on $ C^{*}(\G) \otimes_{\rm min} C^{*}(\G)^{\rm op}$, which we denote still by $\phi$. 
    We identify $  C^{*}(\G) = C^{*}(\G)^{\rm op}$ (Lemma \ref{l_opfreehyper}) and invoke Theorem \ref{th_Qqcstates} 
    to obtain a probabilistic model $p' \in \Qq_{\rm qa}(\G,\G)$ such that 
$\phi(p_{x}\otimes p_{y}) = p'(x,y)$, $x,y \in V$. 
    It is straightforward that $ p = p'$ and hence $p$ approximately quantum.

(b) 
Assume that $p\in \Qq^{\rm c}_{\rm q}(\G) $ 
and, using the universal property of $C^{*}(\G)$, 
write $ p(x,y) = \sca{\pi_{1}(p_{x})\otimes \pi_{2}(p_{y})\xi,\xi)}$, 
where $ \pi_{1}$, $ \pi_{2}$ are representations of $ C^{*}(\G) $ on some finite dimensional Hilbert spaces $\Hh_1, \Hh_{2}$ respectively. Consider the map $s_0$ to be the restriction of the vector state $\omega_{\xi}$ on  $ \rm ran (\pi_{1}) \otimes_{\rm min} \rm ran (\pi_{2})  $. Then $ s_0$ is a state such that $ s_0(\pi_{1}(p_{x})\otimes \pi_{2}(p_{y}) ) = p(x,y)$ for all $ x\in V$, $ y\in W$. Using that $ p$ is coherent and following the arguments in the first part of the proof, one can show that the functional $ \tau_0  $ on $ \rm ran (\pi_{1})$ defined by $ \tau_0 (\pi_{1}(z)) = s_0(\pi_{1}(z)\otimes 1) $ is a tracial state such that $ (\tau_0 \circ \pi_1 )(p_{x}p_{y}) = p(x,y)$, $x\in V$, $y \in W$. 
The converse statement is straightforward.

(c) follows using similar arguments to those in (b) and the details are omitted.
\end{proof}

%%%%%%%%%%%%%%%%%%%%%%%%%%%%%%%%%%%%%%%%%%%%%%%%%%%%%%%%%%%%%%%%%%%%%%%%%%%%%%%%%%%%%%%%%%%%
%%%%%%%%%%%%%%%%%%%%%%%%%%%%%%%%%%%%%%%%%%%%%%%%%%%%%%%%%%%%%%%%%%%%%%%%%%%%%%%%%%%%%%%%%%%%
%%%%%%%%%%%%%%%%%%%%%%%%%%%%%%%%%%%%%%%%%%%%%%%%%%%%%%%%%%%%%%%%%%%%%%%%%%%%%%%%%%%%%%%%%%%%

\section{Equivalences with the Connes Embedding Problem}\label{s_CEP}

In this section we establish equivalences of the Connes Embedding Problem 
(CEP) \cite{Connes76}
in terms 
of properties of the operator systems $\cl S_{\G}$ and $\cl T_{\G}$, and their C*-covers
$\cl B_{\G}$ and $C^*(\G)$. 
We stress that, in view of the negative answer to Connes' Embedding Problem 
obtained in \cite{jnvwy}, there exist counterexamples to the 
equivalent statements in Theorem \ref{th_CEPe} 
and Proposition \ref{r_fritzconnes}. 
The main value of these results resides in providing 
possible alternative ways to arrive at the aforementioned resolution.

We start with recalling some basics on the CEP and associated approximation properties. 
We will use the Kirchberg formulation of CEP \cite[Chapter 14]{pisier_2020}, namely, the 
equality 
%$$C^*(\bb{F}_{\infty})\otimes_{\min} C^*(\bb{F}_{\infty}) = C^*%(\bb{F}_{\infty})\otimes_{\max} C^*(\bb{F}_{\infty}),$$ equivalently,
$$C^*(\bb{F}_2)\otimes_{\min} C^*(\bb{F}_2) = C^*(\bb{F}_2)\otimes_{\max} C^*(\bb{F}_2),$$
where 
$C^*(\bb{F}_2)$ is the full group C*-algebra of the 
free group $\bb{F}_2$ on two generators.
We let $g_1, g_2$ denote the generators of $\bb{F}_2$, and set 
$\cl S_2 = {\rm span}\{e, g_1,g_2,g_1^*,g_2^*\}$, viewed as an operator subsystem of $C^*(\bb{F}_2)$. 

Let $\cl S$ be an operator system. We say that $\cl S$ possesses the 
{\it operator system lifting property (OSLP)} if, 
whenever $\cl A$ is a unital C*-algebra and $\cl J\subseteq \cl A$ is a 
closed ideal, for every unital completely positive map $\phi : \cl S\to \cl A/\cl J$ there exists 
a unital completely positive map $\psi : \cl S\to \cl A$ such that 
$\phi = q\circ \psi$ (where $q : \cl A\to \cl A/\cl J$ is the quotient map). 
We say that $\cl S$ possesses the 
{\it operator system local lifting property (OSLLP)} if every finite dimensional operator 
subsystem of $\cl S$ possesses OSLP. 
If $\cl S$ is a unital C*-algebra, we simply refer to these properties as LP and LLP, respectively. 
We further recall that a unital $C^{*}$-algebra 
$\Aa$ is said to 
\begin{itemize}
\item[(i)]
be {\it projective} if, given any unital $C^{*}$-algebra $\B$ and a 
closed ideal $ \Jj \subseteq \B$, for every unital *-homomorphism $ \pi : \Aa \rightarrow \B/\Jj$ 
there exists a unital *-homomorphism $ \psi : \Aa \rightarrow \B$ such that $ \pi = q \circ \psi$;
\item[(ii)] be {\it residually finite dimensional (RFD)} 
if, whenever $a\in \cl A$ is non-zero, there exists a finite dimensional Hilbert space $\cl H$ and a 
*-representation $\pi : \cl A\to \cl B(\cl H)$ such that $\pi(a)\neq 0$;
\item[(iii)] 
have the {\it weak expectation property (WEP)} if 
$\cl A \otimes_{\rm min} C^{*}(\mathbb{F}_{2}) = \cl A \otimes_{\rm max} C^{*}(\mathbb{F}_{2})$. 
\end{itemize}

Let $ \Ss$ and $\Tt$ be operator systems with $ \dim \Ss < \infty$. 
Given an element $u \in \Ss\otimes \Tt$, say, $u = \sum_{i=1}^k s_{i}\otimes t_{i}$, 
let $\phi_{u} : \Ss^{\rm d}\to \Tt$ be the (linear) map, given by 
$\phi_{u}(f) = \sum_{i} f(s_{i})t_{i}$, $ f\in \Ss^{\rm d}$. 
By \cite[Lemma 8.4]{Kavruk2010QuotientsEA},
the correspondence 
\begin{equation}\label{eq_uphiu}
    (\Ss\otimes_{\rm min} \Tt)^{+} \rightarrow {\rm CP}(\Ss^{\rm d},\Tt), \ \ \ \ 
     u \mapsto \phi_{u}
\end{equation}
is a well-defined bijection.

%Recall that $ \Ss_{\G}^{\rm d} = \Ll_{\G}$ which was a subsystem of $ D_V$. Now, $ D_V$ is a nuclear $C^*$-algebra. In particular, it is (min,max)-nuclear operator system. In particular, it is (min,el)-nuclear which is equivalent to being exact \cite[Theorem 5.7]{Kavruk2010QuotientsEA}. Now since exactness passes to operator subsystems \cite[Proposition 4.10]{Kavruk2014}, we conclude that $ \Ll_{\G}= \Ss_{\G}^{\rm d}$ is exact. Finally, by \cite[Theorem 6.6]{Kavruk2014}, since $ \Ss_{\G}^{\rm d}$ is exact, we have that $\Ss_{\G} $ has the lifting property. Also. by 
%\cite[Theorem 6.10]{Kavruk2014} this is equivalent to saying that $ C^*_u(\Ss_{\G})  $ has the LLP.

The next proposition collects several facts that will be needed in the proof of
Theorem \ref{th_CEPe}. 
We point out that item (i) can be deduced using deeper functoriality properties 
involving exactness and OSLLP \cite{Kavruk2010QuotientsEA, Kavruk2014}; 
the elementary proof given here 
follows closely the proof of \cite[Proposition 8.4]{Todorov2020QuantumNC}, and 
we include it for the convenience of the reader.

\begin{proposition}\label{p_apppro}
Let $\G$ be a contextuality scenario. The following hold true:
\begin{itemize}
\item[(i)] 
the operator system $\Ss_{\G}$ has the OSLLP;

\item[(ii)] 
if $\G$ is dilating then the C*-algebra $C^{*}(\G)$ has the LLP;

\item[(iii)]    
the  $C^{*}$-algebra $\cl B_{\G}$ is projective.
\end{itemize}
\end{proposition}

\begin{proof}
(i) 
    We will prove that $ \Ss_{\G}\otimes_{\rm min} \B(\Hh) =  \Ss_{\G}\otimes_{\rm max} \B(\Hh)$ for every Hilbert space $ \Hh$ which, by \cite[Theorem 8.6]{Kavruk2010QuotientsEA}, is equivalent to $\Ss_{\G}$ having the OSLLP. 
    To this end, it suffices to show that 
    $$ ( \Ss_{\G}\otimes_{\rm min} \B(\Hh))^{+} \subseteq ( \Ss_{\G}\otimes_{\rm max} \B(\Hh))^{+} $$ 
    for every Hilbert space $\Hh$. 

    Fix a Hilbert space $\Hh $, let $ u \in (\Ss_{\G}\otimes_{\rm min} \Bh)^{+}$ and let 
    $ \phi_{u} : \Ss_{\G}^{\rm d} \rightarrow \Bh$ be the completely positive map
    arising according to (\ref{eq_uphiu}). 
%By Corollary \ref{p_dualSG}, $\Ss_{\G}^{\rm d} = \Ll_{\G}$.
Note that $\cl S_{\G}^{\rm d}$ is an operator subsystem of the operator 
abelian (hence nuclear) C*-algebra 
$\Ss= \oplus_{e\in E}\ell^{\infty}_{e}$. 
By Arveson's Extension Theorem, there exists a completely positive map 
$ \psi : \Ss \rightarrow \Bh$ extending $ \phi_{u}$. By the correspondence 
(\ref{eq_uphiu}), there exists a $ v \in (\Ss\otimes_{\rm min}\Bh)^{+}$, 
such that $ \psi= \phi_{v}$. 
The nuclearity of $\Ss$ implies that $ v \in (\Ss\otimes_{\rm max}\Bh)^{+} $ and thus 
$w:= (q \otimes \id) (v) \in (\Ss_{\G} \otimes_{\rm max} \Bh)^{+}$, 
where $ q : \Ss \rightarrow \Ss_{\G} $ is the canonical quotient map. 
The fact that $w = u$ is now straightforward.

(ii) 
Theorem \ref{th_dila} and the fact that $\G$ is assumed dilating 
imply that $ \Ss_{\G} = \Tt_{\G}$; by (i), $ \Tt_{\G}$ has the OSLLP. In particular,
using \cite[Theorem 6.6]{KAVRUK2011267}, we have 
\[  \Tt_{\G} \otimes_{\rm min} \B(\Hh) = \Tt_{\G} \otimes_{\rm max} \B(\Hh )  
= \Tt_{\G} \otimes_{\rm c} \B(\Hh).\]
By Lemma \ref{l_idCstareTG}, $ \Tt_{\G}$ contains enough unitaries in $C^{*}(\G)$, 
so \cite[Corollary 5.8]{Kavruk2014} implies that 
\[ C^{*}(\G) \otimes_{\rm min} \B(\Hh) = C^{*}(\G) \otimes_{\rm max} \B(\Hh)  \]
for arbitrary Hilbert space $\Hh$, which again is equivalent to 
$ C^{*}(\G)$ having LLP.

(iii) 
Let $\cl B$ be a unital C*-algebra,  $\Jj \subseteq \cl B$ be a closed ideal and
$\pi : \cl B_{\G} \rightarrow \B/ \Jj$ be a $*$-homomorphism.
Write $ q : \B \rightarrow \B/\Jj $ for the quotient map. 
By (i) and \cite[Theorem 6.10]{Kavruk2014}, $\cl B_{\G}$ has the LLP.
Thus, there exists a unital completely positive map 
$ \psi : \Ss_{\G}  \rightarrow \B$ such that $ \pi \arrowvert_{\Ss_{\G}} = q \circ \psi$.
By the universal property of $\cl B_{\G}$
(see Corollary \ref{c_unicsta}), the map $ \psi$ extends to a unital $*$-homomorphism $ \rho : \cl B_{\G} \rightarrow \B$. Therefore,
$\pi \arrowvert_{\Ss_{\G}} = q \circ \rho \arrowvert_{\Ss_{\G}}$;
since $ \Ss_{\G}$ generates $\cl B_{\G}$ as a $C^{*}$-algebra, 
we conclude that $ \pi = q \circ \rho  $. 
\end{proof}

\begin{remark} \label{r_rfd}
    By
 \cite[Corollary 13.1.4]{brown2008textrm},
 a unital $C^{*}$-algebra $\cl A$ has the LP if and only if every unital $*$-homomorphism $\theta : \Aa \rightarrow \B/ \Jj$ into a quotient $C^{*}$-algebra, admits a unital completely positive lift
 to a map into $\cl B$. 
By Proposition \ref{p_apppro} (iii), $\cl B_{\G}$ has the LP. 
Since $\cl B_{\G}$ is separable  and projective, 
it is RFD and has no non-trivial projections (see \cite[ Theorem 11.2.1]{loring1997lifting}, \cite[Proposition 4.5]{courtney2019universal}). 
\end{remark}

\begin{theorem}\label{th_CEPe}
The following statements are equivalent:
    \begin{enumerate}
        \item CEP has an affirmative answer;
        
        \item $\Qq_{\rm qa}(\G,\G)= \Qq_{\rm qc}(\G,\G)$ for every dilating 
        contextuality scenario $\G$;
        
        \item $ C^{*}(\G) \otimes_{\rm min} C^{*}(\G) = C^{*}(\G) \otimes_{\rm max} C^{*}(\G)$ for every dilating contextuality scenario $\G$;
        
        \item $ \Tt_{\G} \otimes_{\rm min} \Tt_{\G} = \Tt_{\G}\otimes_{\rm c} \Tt_{\G}$ for every dilating contextuality scenario $\G$;
        
        \item $\Qqc_{\rm qa}(\G,\G)= \Qqc_{\rm qc}(\G,\G)$ for every contextuality scenario $\G$;
        
        \item $\cl B_{\G} \otimes_{\rm min} \cl B_{\G} = 
        \cl B_{\G} \otimes_{\rm max} \cl B_{\G}$ 
        for every contextuality scenario $\G$;
        
        \item $\Ss_{\G} \otimes_{\rm min} \Ss_{\G} = \Ss_{\G} \otimes_{\rm c} \Ss_{\G}$ 
        for every contextuality scenario $\G$;
        
        \item $\cl B_{\G} \otimes_{\rm max} \cl B_{\G} $ is RFD for every 
        contextuality scenario $\G$;
        
        \item $\Qq_{\rm qa}^{\rm c}(\G)= \Qq_{\rm qc}^{\rm c}(\G)$ for every dilating 
        contextuality scenario $\G$;
        
        \item $ C^{*}(\G)$ has the WEP for every dilating contextuality scenario $\G$;
        
        \item $ \Tt_{\G} \otimes_{\rm min}
        \Ss_{2} = \Tt_{\G} \otimes_{\rm c}  \Ss_{2}$ 
        for every dilating 
        contextuality scenario $\G$.
    \end{enumerate}
\end{theorem}

\begin{proof}
    $(ii) \Rightarrow  (i)$    Follows by letting $\G$ vary over all Bell scenarios $\bb{B}_{X,A}$, 
    after recalling that $\bb{B}_{X,A}$ are dilating (Remark \ref{r_bellscenarios})
and using the equivalence between the Tsirelson Problem and CEP 
\cite{Ozawa_2013, MR2790067}.

    $(i) \Rightarrow (iii)$
    By Proposition \ref{p_apppro}, if $\G$ is dilating then $ C^{*}(\G) $ has the LLP. 
    Hence, it also possesses the WEP (see e.g. \cite[Proposition 13.1]{pisier_2020}), 
    which implies that 
    $ C^{*}(\G) \otimes_{\rm min} C^{*}(\G) = C^{*}(\G) \otimes_{\rm max} C^{*}(\G)$
    by the generalised Kirchberg theorem (see \cite[Corollary 9.40]{pisier_2020}).

    $(iii) \Rightarrow (ii) $
    follows by the characterisation of probabilistic models in terms of states
    given in Theorem \ref{th_Qqcstates}. 

     $(iv) \Rightarrow (iii) $
     follows from Lemma \ref{l_idCstareTG} 
     and \cite[Corollary 5.8]{Kavruk2014}.

     $(i) \Rightarrow (iv) $
     Assume CEP has an affirmative answer. 
     By Theorem \ref{th_dila}, $\Tt_{\G} = \cl S_{\G}$ up to a (canonical) unital order isomorphism; 
     thus, by Proposition \ref{p_apppro}, $\Tt_{\G}$ has the OSLPP. 
     By \cite[Theorem 9.1]{Kavruk2010QuotientsEA}, 
          $ \Tt_{\G} \otimes_{\rm min} \Tt_{\G} = \Tt_{\G}\otimes_{\rm c} \Tt_{\G}$. 
          
    $(v) \Rightarrow (i)$ follows again by letting $\G$ vary over all Bell scenarios. 

    $(i) \Rightarrow (vi)$ By (the proof of) 
    Proposition \ref{p_apppro}, $\cl B_{\G}$ has LLP; by the assumption and 
    \cite[Proposition 13.1]{pisier_2020}, it has WEP, and (vi) follows. 
    
    $(vi) \Rightarrow (vii) \Rightarrow (v)$ 
    follows from Theorem \ref{th_qcqatilde}.

$(vi) \Rightarrow (viii)$  
Since $\cl B_{\G}$ is RFD (Remark \ref{r_rfd}), we have that 
$\cl B_{\G} \otimes_{\rm min} \cl B_{\G}$ is
also RFD; by the assumption, 
$\cl B_{\G} \otimes_{\rm max} \cl B_{\G}$ is RFD.

$(viii) \Rightarrow (vi)$
can be proved in a similar fashion to 
\cite[Proposition 3.19]{Ozawa2003ABOUTTQ}, using the fact that 
the C*-algebra $\cl B_{\G}$ has the lifting property
(see also 
\cite[Proposition 13.1]{pisier_2020}). 
%Let $\cl H$ be a finite dimensional Hilbert space and 
%$\sigma : \cl B_{\G} \otimes_{\rm max} \cl B_{\G}\to \cl B(\cl H)$
%be an irreducible representation. 
%\marginpar{\tiny IT. Find reference.}
%Up to unitary equivalence, $\sigma = \pi\otimes \rho$, where 
%$\pi$ and $\rho$ are *-representations of $\cl B_{\G}$ on 
%finite dimensional Hilbert spaces. 
%Thus, if $u$ is an element of the algebraic tensor product of two copies of 
%$\cl B_{\G}$, we have that 
%$\|\sigma(u)\| \leq \|u\|_{\min}$. Since
%$\cl B_{\G} \otimes_{\rm max} \cl B_{\G}$ is assumed RFD, we have that 
%$\|u\|_{\max}\leq \|u\|_{\min}$, that is, 
%$\cl B_{\G} \otimes_{\rm min} \cl B_{\G} = \cl B_{\G} \otimes_{\rm max} \cl B_{\G}$.

    $(ii) \Rightarrow (ix)$ is straightforward from the fact that 
    $\Qq_{\rm qa}^{\rm c}(\G)\subseteq\Qq_{\rm qa}(\G)$ and $\Qq_{\rm qc}^{\rm c}(\G)
    \subseteq \Qq_{\rm qc}(\G)$. 
    
    $(ix) \Rightarrow (i)$
    Applied to the case where $\G$ is the Bell scenario, the assumption shows that 
    the class of synchronous quantum approximate correlations coincides with that of
    synchronous quantum commuting ones, and the conclusion follows from \cite{10.1063/1.4996867}.

    $(i) \Rightarrow  (x)$
    Assume CEP has an affirmative answer. By Proposition \ref{p_apppro}, if $\G$ is dilating then 
    $C^{*}(\G)$ has the LLP. By \cite[Proposition 13.1]{pisier_2020}, $C^{*}(\G)$ has the WEP.

    $(x) \Rightarrow (iii)$ Follows from the generalised Kirchberg's Theorem
    (see \cite[Corollary 9.40]{pisier_2020}). 

    $(xi) \Rightarrow  (x)$
    is immediate from \cite[Corollary 5.8]{Kavruk2014} and the fact that 
$\cl T_{\G}$ and $\cl S_2$ contain enough unitaries (see Lemma \ref{l_idCstareTG}). 

    $(x) \Rightarrow (xi) $
    By \cite[Theorem 5.9]{Kavruk2014}, 
    the assumption that $C^{*}(\G)$ has the WEP is equivalent to the identity 
    $ C^{*}(\G) \otimes_{\rm min}  \Ss_{2} = C^{*}(\G) \otimes_{\rm max}  \Ss_{2}$. 
    By the injectivity of the minimal tensor product, we have that 
    $ \Tt_{\G} \otimes_{\rm min}  \Ss_{2} \subseteq  C^{*}(\G) \otimes_{\rm min}  \Ss_{2}  $ canonically. 
    By Lemma \ref{l_idCstareTG} and the definition of the left injective operator system tensor product 
    ${\rm el}$ \cite{Kavruk2010QuotientsEA}, we have that
    $ \Tt_{\G} \otimes_{\rm el}  \Ss_{2} \subseteq  C^{*}(\G) \otimes_{\rm max}  \Ss_{2}$.
    Thus $ \Tt_{\G} \otimes_{\rm min}  \Ss_{2} = \Tt_{\G} \otimes_{\rm el}  \Ss_{2}$. 
    On the other hand, the assumption is equivalent to CEP, and now
    \cite[Theorem 9.1]{Kavruk2010QuotientsEA} implies that the operator systems with the OSLLP 
    possess the DCEP which, by \cite[Theorem 7.3]{Kavruk2010QuotientsEA}, 
    is equivalent to (el,c)-nuclearity. 
    Recalling that $\cl T_{\G}$ possesses OSLLP (Proposition \ref{p_apppro}), we now conclude that 
    $\Tt_{\G} \otimes_{\rm min}  \Ss_{2} = \Tt_{\G} \otimes_{\rm el} \Ss_{2} 
    =  \Tt_{\G} \otimes_{\rm c}  \Ss_{2}  $.
 \end{proof}

\begin{remark}
\rm 
Proposition \ref{p_maxotHbb}, 
together with the proofs of \cite[Proposition 3.19]{Ozawa2003ABOUTTQ} and 
\cite[Proposition 13.1]{pisier_2020}, show that 
the statements in Theorem \ref{th_CEPe} are equivalent to 
the existence of a faithful tracial state on the 
C*-algebra $C^*(\G\otimes\G)$, for every dilating contextuality scenario $\G$. 
\end{remark}

We conclude with a discussion of a particular contextuality scenario, 
motivated by 
%Ozawa's result \cite{Ozawa_2013} that the CEP
%has an affirmative answer is and only if 
%the full group $C^{*}$-algebra $ C^{*}(\mathbb{F}_{2} \times \mathbb{F}_{2})$ is RFD. 
%Also, 
T. Fritz's results \cite{Fritz2020CuriousPO} that 
the (full) group $C^{*}$-algebra 
$ C^{*}((\mathbb{Z}_{2}* \mathbb{Z}_{3})\times(\mathbb{Z}_{2}* \mathbb{Z}_{3}))$ is a free 
hypergraph $C^{*}$-algebra, and that CEP has an affirmative answer if and only if 
the latter $ C^{*} $-algebra is RFD. 
Let $ \G_{2,3}$ be the hypergraph with two disjoint edges
having two and three vertices, respectively. By Theorem \ref{projectivedilation}, 
$ \G_{2,3}$ is dilating, and hence, by Theorems \ref{th_dila} and \ref{projectivedilation}, 
$$ \Tt_{\G_{2,3}} = \Ss_{\G_{2,3}} = \ell^{\infty}([2]) \oplus_{1} \ell^{\infty}([3]).$$

\begin{proposition} \label{r_fritzconnes}
    The following are equivalent:
    \begin{enumerate}
        \item CEP has an affirmative answer; 
         \item $\Tt_{\G_{2,3}} \otimes_{\min} \Tt_{\G_{2,3}} = \Tt_{\G_{2,3} \otimes \G_{2,3}}$; 
        %\item $ C^{*}( \mathbb{F}_{2} \times \mathbb{F}_{2})$ is RFD; 
        %\item $C^{*}( \G_{2,3} \otimes \G_{2,3}) $ is RFD; 
        \item $C^{*}(\G_{2,3}) \otimes_{\min}C^{*}(\G_{2,3}) = C^{*}(\G_{2,3}) \otimes_{\max}C^{*}(\G_{2,3})$.
    \end{enumerate}
Further, any of these conditions imply that
$\Qq_{\rm qa}(\G_{2,3},\G_{2,3}) = \Qq_{\rm qc}(\G_{2,3},\G_{2,3}) $.
\end{proposition}

\begin{proof}
   %$(i) \Leftrightarrow (iii) $
   %  Follows from 
   %  \cite[Proposition 2.13]{Fritz2020CuriousPO} and Proposition \ref{p_maxotHbb}, 
   %  since $ C^{*}((\mathbb{Z}_{2}* \mathbb{Z}_{3})\times(\mathbb{Z}_{2}* \mathbb{Z}_{3})) = C^{*}(\G_{2,3}) \otimes_{\max}C^{*}(\G_{2,3})$.
     $(i) \Rightarrow (iii)$ 
follows from Theorem \ref{th_CEPe} and the fact that $\G_{2,3}$ is dilating. 

     $(iii) \Rightarrow (i)$   
     We have that $C^{*}(\G_{2,3}) = C^{*}(\mathbb{Z}_{2}* \mathbb{Z}_{3})$.
     Thus, by the assumption, the C*-algebra 
     $C^{*}((\mathbb{Z}_{2}* \mathbb{Z}_{3})\times (\mathbb{Z}_{2}* \mathbb{Z}_{3}))$
     is RFD. By \cite{Fritz2020CuriousPO}, CEP has an affirmative answer. 
     %The equivalence now follows from the fact that 
     %$ C^{*}(\G_{2,3} \otimes \G_{2,3})=  C^{*}(\G_{2,3}) \otimes_{\max}C^{*}(\G_{2,3}) $ 
     %(see Proposition \ref{p_maxotHbb}) 
     %and the fact that the finite dimensional representations of the maximal tensor %product 
     %factor through the minimal one.
     
$(ii) \Leftrightarrow  (iii) $
%Further, by Proposition \ref{p_maxotHbb},
%We have that 
%
By Proposition \ref{p_maxotHbb}, 
$ C^{*}(\G_{2,3}) \otimes_{\max} C^{*}(\G_{2,3})= C^{*}(\G_{2,3}\otimes\G_{2,3})$.  
Thus, by \cite[Proposition 5.7]{Kavruk2014} the equality $\Tt_{\G_{2,3}} \otimes_{\min} \Tt_{\G_{2,3}} = \Tt_{\G_{2,3} \otimes \G_{2,3}}$ 
implies the equality 
$C^{*}(\G_{2,3}) \otimes_{\min}C^{*}(\G_{2,3}) = C^{*}(\G_{2,3}) \otimes_{\max}C^{*}(\G_{2,3})$. 
\end{proof}

\begin{remark}\label{r_G0only}
    There exist a contextuality scenario $\G_0$ such that 
    \[ \Qq_{\rm qa}(\G_0,\G_0) \neq \Qq_{\rm qc}(\G_0,\G_0). \]
    Indeed, using \cite[Theorem 3.5]{Fritz2020CuriousPO}, let 
    $\G_0$ be a contextuality scenario that has only infinite dimensional representations. 
    Since there are no finite dimensional unital representations of $ C^{*}(\G_0)$, there are also no finite dimensional representations of 
    $C^{*}(\G_0) \otimes_{\min} C^{*}(\G_0)$. 
    By Theorem \ref{th_Qqcstates}, 
    $ \Qq_{\rm qa}(\G_0,\G_0) = \emptyset$ while, on the other hand, 
    $ \Qq_{\rm qc}(\G_0,\G_0)\neq \emptyset$. 
\end{remark}

%%%%%%%%%%%%%%%%%%%%%%%%%%%%%%%%%%%%%%%%%%%%%%%%%%%%%%%%%%%%%%%
%%%%%%%%%%%%%%%%%%%%%%%%%%%%%%%%%%%%%%%%%%%%%%%%%%%%%%%%%%%%%%%
%%%%%%%%%%%%%%%%%%%%%%%%%%%%%%%%%%%%%%%%%%%%%%%%%%%%%%%%%%%%%%%

\section{Summary and outlook}

In this section we summarise the main contributions and indicate the limitations of our framework.  
We develop an operator-system approach to contextuality scenarios, providing a general setting that specialises to nonlocality. For any scenario $\G$, the central objects are the universal operator system $\Ss_{\G}$, encoding POR's, and the companion system $\Tt_{\G}\subseteq C^*(\G)$, encoding precisely those POR's that dilate to PR's. Equivalently, a scenario is dilating iff $\Ss_{\G}=\Tt_{\G}$. Unlike other cases (e.g. graph operator systems),  $\Ss_{\G}$ is not a faithful invariant of the underlying hypergraph.

After realising $\Ss_{\G}$ as a universal quotient, we show that the natural matrix-order cones need not be proper (Fig.~\ref{fig:counterexample_hypergraph}). We repair this by passing to the canonical normalisation/Archimedean quotient and by isolating a natural combinatorial class—uniformisable scenarios—for which the quotient kernel is a null subspace and properness is guaranteed. This quotient perspective also allows us to identify the dual as a function system, which later plays a role in establishing the LLP for $\Ss_{\G}$ and its associated $C^*$-algebras.  

We further demonstrate that not every scenario is dilating, which justifies the introduction of $\Ss_{\G}$. Already the trivial instance (Fig.~\ref{fig:NoPRs}) fails to dilate; the scenario $\G_0$ (Fig.~\ref{fig:noPRs2}) provides a structured obstruction, captured by a practical test for dilatability (Proposition~\ref{testfordilating}), and the quantum magic-square family yields additional robust counterexamples (Corollary~\ref{c_diffo}). The operator system $\Tt_{\G}$, typically distinct from $\Ss_{\G}$, is the canonical subsystem of $C^*(\G)$. The latter C*-algebra is its C*-envelope  and can be realised as a quotient of a free product of $\ell^{\infty}$-spaces determined by the edge sizes. We note, however, that a hypergraph analogue of Boca’s theorem cannot hold in full generality, as it would imply dilatability for all scenarios. In light of these obstructions, we isolate a natural class of dilatable scenarios (Fig.~\ref{fig:my_label2}).  

Beyond these structural considerations, we identify the maximal $C^*$-cover of $\Ss_{\G}$ via a canonical TRO generated by $\G$-stochastic operator matrices, characterise no-signalling classes as states on suitable operator-system tensor products, and introduce coherent models, which we classify through traces on the hypergraph $C^*(\G)$. These results yield new reformulations of the Connes and Tsirelson problems within the hypergraph setting.

\section{Questions}

In this section, we collect several questions related to this work.
In Corollary \ref{c_dilapro}, we showed that, if 
$\G$ and $\Hbb$ are dilating contextuality scenarios then 
$\Tilde{\Qq}_{\rm q}(\G,\Hbb) = \Qq_{\rm q}(\G,\Hbb)$ and 
$\Tilde{\Qq}_{\rm qa}(\G,\Hbb) = \Qq_{\rm qa}(\G,\Hbb)$.
The quantum commuting case is outstanding:

\begin{question}\label{q_qcqc}
\rm 
Let $\G$ and $\Hbb$ be dilating contextuality scenarios. Is it true that $\Tilde{\Qq}_{\rm qc}(\G,\Hbb) = \Qq_{\rm qc}(\G,\Hbb)$? 
\end{question}

Under the assumption that $\G$ and $\Hbb$ are dilating, one can show that 
every element $p$ of $\Tilde{\Qq}_{\rm qc}(\G,\Hbb)$ arises from a state 
on the commuting tensor product 
$\cl T_{\G}\otimes_{\rm c}\cl T_{\Hbb}$; on the other hand, the elements of 
$\Qq_{\rm qc}(\G,\Hbb)$ arise, by Theorem \ref{th_Qqcstates}, from 
the states on the enveloping tensor product $\cl T_{\G}\otimes_{\rm e}\cl T_{\Hbb}$.
Thus, an affirmative answer to the following question will imply an 
affirmative answer of Question \ref{q_qcqc}:

\begin{question}\label{q_ten1}
\rm 
Let $\G$ and $\Hbb$ be contextuality scenarios. 
Is it true that $ \Tt_{\G} \otimes_{\rm c}\Tt_{\Hbb} = \Tt_{\G} \otimes_{\rm e}  \Tt_{\Hbb}$?
Does the equality hold in the case where $\G$ and $\Hbb$ are dilating? 
\end{question}

We note that the equality in Question \ref{q_ten1} is equivalent 
to the validity of a canonical complete order inclusion 
$\Tt_{\G} \otimes_{\rm c}\Tt_{\Hbb} \subseteq C^{*}(\G) \otimes_{\rm max} C^{*}(\Hbb)$.

\begin{question}\label{q_CEP}
\rm 
(i) Is it true that CEP implies (and is hence equivalent to) the equality
$\Qq_{\rm qa}(\G,\G) = \Qq_{\rm qc}(\G,\G)$
for all 
contextuality scenarios $ \G$?

(ii) Is it true that CEP implies the equality $ \Qq^{\rm c}_{\rm qa}(\G,\G) = 
\Qq^{\rm c}_{\rm qc}(\G,\G)$ for all 
contextuality scenarios $ \G$? 
\end{question}

We note that if Question \ref{q_CEP} (i) has an affirmative answer
then the scenario in Remark \ref{r_G0only} would witness 
the negative answer to the CEP. 
We also point out that Question \ref{q_CEP} (ii) is a hypergraph version 
of the problem of whether an inequality 
between the classes of quantum commuting and quantum approximate 
bisynchronous correlations witnesses the negative answer to CEP
\cite{Paulsen2019BisynchronousGA}.

\begin{question}\label{q_LLP}
\rm 
Is it true that $ C^{*}(\G)$ has LLP and that $\Tt_{\G}$ has OSLLP? 
\end{question}

We note that an affirmative answer to Question \ref{q_LLP} implies an 
affirmative answer to Question \ref{q_CEP}. Indeed, if CEP holds true then 
the LLP of the C*-algebra $C^{*}(\G)$ implies the equality
$C^{*}(\G)\otimes_{\min} C^{*}(\G) = C^{*}(\G)\otimes_{\max} C^{*}(\G)$, and by 
Theorem \ref{th_Qqcstates} we have that 
$\Qq_{\rm qa}(\G,\G) = \Qq_{\rm qc}(\G,\G)$. 

The next question is of combinatorial nature; 
indeed, 
we were only able to explicitly identify the dual of the operator system 
$\cl S_{\G}$ in case the hypergraph $\G$ is uniform. Note that
the question is related to the problem of whether the subspace 
$\cl J_{\bb{G}}$ is in fact an operator system kernel (see Section \ref{s_opsys}). 

\begin{question}\label{q_dual}
\rm 
Can one provide an explicit identification of 
the dual $\cl S_{\G}^{\rm d}$ of $\cl S_{\G}$ in the case where $\G$ is not uniform? 
Is the subspace $\cl J_{\bb{G}}$ always a kernel?
\end{question}

While we have a sufficiently concrete description of the maximal 
C*-cover of the operator system $\Ss_{\G}$ (see Section \ref{ss_univcov}), 
we do not know what its C*-envelope looks like, except in the cases of 
dilating scenarios $\G$ (when it coincides with $C^*(\G)$):

\begin{question}\label{q_cstaren}
\rm 
Can $C^{*}_{\rm e}(\Ss_{\G})$ be explicitly identified?
\end{question}

A satisfactory answer to Question \ref{q_cstaren} will lead to more economical 
descriptions of no-signalling probabilistic models and would allow 
the addition of new equivalences of CEP in Theorem \ref{th_CEPe}. 

It would be of interest to provide classes of examples of both dilating and 
non-dilating contextuality scenarios. As a test question, we mention the 
following:

\begin{question}\label{q_KS}
\rm 
Is the Kochen-Specker scenario $\G_{KS} $ dilating?
\end{question}

An affirmative answer to Question \ref{q_KS} would imply an 
affirmative answer to the question, posed in \cite{Acn2015ACA}, 
of whether all probabilistic models of the Kochen-Specker scenario are quantum.

\medskip

\noindent 
{\bf Data availability statement. } 
No data was used for the research described in the article.

\smallskip

\noindent {\bf Conflict of interest statement.}
On behalf of all authors, the corresponding author states that there is no conflict of interest.

\smallskip

\noindent 
{\bf Acknowledgements. } 
We thank the anonymous referee for their useful suggestions. We express our deepest gratitude to Aristides Katavolos for 
numerous discussions on the topic of this paper which helped improving it 
substantially. 
We are grateful to Vern Paulsen for his interest in our work, and to 
John Byrne for discussions that led to a refinement of Proposition \ref{p_nullsubspace} and for pointing out  Example \ref{e_non-unif}.
The second named author thanks the Department of Mathematical Sciences 
at the University of Delaware for the hospitality during a year-long 
visiting position. 
The third named author was partially supported by NSF grants 2115071 and 2154459.

The research described in this paper was carried out within the framework of the National Recovery and Resilience Plan Greece 2.0, funded by the European Union - NextGenerationEU (Implementation Body: HFRI. Project name: 
Noncommutative analysis: operator systems and non-locality. 
HFRI Project Number: 015825).

%%%%%%%%%%%%%%%%%%%%%%%%%%%%%%%%%%%%%%%%%%%%%%

%%%%%%%%%%% Back matter starts %%%%%%%%%%%%%%%
%\backmatter
%%%%%%%%%%%%%%%%%%%%%%%%%%%%%%%%%%%%%%%%%%%%%%
%\printindex
%\clearpage
%\addcontentsline{toc}{section}{References}
\bibliography{contextuality7}

\nocite{*} 
\bibliographystyle{abbrv}

%%%%%%%%%%%%%%%%%%%%%%%%% End of file %%%%%%%%%%%%%%%
\end{document}